\documentclass[a4paper]{article}
\usepackage{geometry}
\usepackage{hyperref}
\usepackage{graphicx}
\usepackage{amsthm,amsmath,amsfonts,amssymb,stmaryrd,mathrsfs,mathdots,amsbsy,float,cmll,xcolor,mathpartir,lmodern,centernot,colonequals,stackrel,tcolorbox}
\usepackage{bm}
\usepackage{tikz}
\usepackage{tikz-cd}
\newsavebox{\pullback}
\sbox\pullback{%
\begin{tikzpicture}%
\draw (0,0) -- (1ex,0ex);%
\draw (1ex,0ex) -- (1ex,1ex);%
\end{tikzpicture}}
\usetikzlibrary{tikzmark}
\usepackage[small]{diagrams}
\usepackage{float}
\usetikzlibrary{positioning,calc}
\usepackage{bussproof}

\makeatletter
\DeclareRobustCommand{\notsquare}{\mathord{\mathpalette\generic@not\square}}
\newcommand{\generic@not}[2]{%
  \sbox\z@{$\m@th#1/$}%
  \sbox\tw@{$\m@th#1#2$}%
  \sbox\z@{\raisebox{\dimexpr(\ht\tw@-\dp\tw@-\ht\z@+\dp\z@)/2\relax}{$\m@th#1/$}}%
  \vphantom{\usebox{\z@}}%
  \ooalign{\hidewidth\usebox{\z@}\hidewidth\cr$\m@th#1#2$\cr}%
}
\makeatother

\theoremstyle{definition}
\newtheorem{definition}{Definition}[section]
\newtheorem{example}[definition]{Example}
\theoremstyle{theorem}
\newtheorem{theorem}[definition]{Theorem}
\newtheorem*{theorem*}{Theorem}
\newtheorem*{lemma*}{Lemma}
\newtheorem*{proposition*}{Proposition}
\newtheorem{proposition}[definition]{Proposition}
\newtheorem{lemma}[definition]{Lemma}
\newtheorem{corollary}[definition]{Corollary}
\newtheorem*{corollary*}{Corollary}
\theoremstyle{remark}
\newtheorem*{remark}{Remark}
\newtheorem*{notation}{Notation}
\newtheorem*{convention}{Convention}

\newarrow{Corresponds}<--->
\newarrow {Dashto} {}{dash}{}{dash}>
\newarrow{Equals}=====

\DeclareFontFamily{U}{MnSymbolC}{}
\DeclareFontShape{U}{MnSymbolC}{m}{n}{
    <-6>  MnSymbolC5
   <6-7>  MnSymbolC6
   <7-8>  MnSymbolC7
   <8-9>  MnSymbolC8
   <9-10> MnSymbolC9
  <10-12> MnSymbolC10
  <12->   MnSymbolC12}{}
\DeclareSymbolFont{MnSyC}{U}{MnSymbolC}{m}{n}
\DeclareMathSymbol{\dotminus}{\mathbin}{MnSyC}{24}

\begin{document}
\title{\Huge \bfseries Logic and computation as combinatorics}
\author{\Large \bfseries Norihiro Yamada \\ \\
{\tt yamad041@umn.edu} \\
University of Minnesota
}

\maketitle

\begin{abstract}
The \emph{syntactic} nature of logic and computation separates them from other fields of mathematics. 
Nevertheless, syntax has been the only way to adequately capture the \emph{dynamics} of proofs and programs such as cut-elimination, and the \emph{finiteness} and the \emph{atomicity} of syntax are preferable for foundational aims as seen in Hilbert's program.
Another issue is that a \emph{uniform} basis for logic and computation has been missing, and this problem hampers a coherent view on them.
For instance, formal proofs in proof theory are far from (ordinary) proofs of the validity of a formula in model theory. 
Our goal is to solve these fundamental problems by rebuilding central concepts in logic and computation such as formal systems, validity (in such a way that it coincides with the existence of proofs), cut-elimination and computability uniformly in terms of finite graphs based on \emph{game semantics}. 
Unlike game semantics, however, we do not rely on anything infinite or extrinsic to the graphs.
A key idea that enables our finitary, autonomous approach is the shift from graphs in game semantics to \emph{dynamic} ones. 
The resulting combinatorics establishes a single, syntax-free, finitary framework that recasts formal systems admitting proofs with \emph{cuts}, validity, the \emph{finest} computational steps of cut-elimination and \emph{higher-order} computability.
This subsumes \emph{fully complete} semantics of intuitionistic linear logic, which solves a problem open for thirty years, and even extends the full completeness to proofs with cuts.
As a byproduct, our dynamic graphs give rise to \emph{Hopf algebras}, which opens up new applications of algebras to logic and computation. 
\end{abstract}

\tableofcontents

\section{Introduction}
\label{Introduction}
Logic and computation have been suffering from the lack of a syntax-free or uniform foundation. 
We take a step towards a solution to this fundamental problem by reducing them to the study of finite dynamic graphs. 
The resulting work recasts formulae, proofs (not only provability) and computability \emph{syntax-freely} and \emph{uniformly}, going beyond the combinatorial characterisation of classical provability by Hughes \cite{hughes2006proofs}.
This result includes a complete solution to a bottleneck of the problem: only syntactically defined \emph{dynamics} and \emph{intensionality} in logic and computation.
It also solves a problem open for thirty years: \emph{full completeness} for intuitionistic linear logic.

\subsection{Foreword}
\label{Foreword}
On the one hand, the \emph{syntactic} nature of (formal) proofs and programs separates mathematical logic and theoretical computer science, or logic and computation for short, from other branches of mathematics.
As cited by Hughes \cite{hughes2006proofs}, for instance, De Morgan \cite{de1868review} states that
\begin{center}
[M]athematicians care no more for logic than logicians for mathematics.
\end{center}
The issue is that mathematics studies the \emph{general, abstract essence} of the universe, but syntactic objects such as proofs and programs are bound by \emph{how to write}, i.e., inessential details.

Besides, from a mathematician's view, syntax-independent objects such as sets and functions come first, and syntax is nothing but their notations. 
For this point of view, syntax \emph{per se} is mysterious and cumbersome. 
For example, what proofs are and why they validate formulae are central questions in logic, but the standard approach of the field takes \emph{how to write proofs} as the definition of proofs and justifies these notations as validations of formulae only in terms of how to rewrite proofs \cite{gentzen1935untersuchungen,prawitz1971ideas,dummett1993logical} or indirectly by their completeness \cite{godeluber}.
This syntactic paradigm does not clarify that much what proofs are or why they validate formulae. 

Strictly speaking, the field of \emph{mathematical semantics} \cite{scott1970outline,gunter1992semantics} has established a variety of syntax-free interpretations of the \emph{extensional} aspects of proofs (though most of them do not look like proofs in the practice of mathematics) and programs.
However, no syntax-free method has completely replaced proofs or programs because syntax has been the only way to adequately model their \emph{dynamics} such as cut-elimination \cite{gentzen1935untersuchungen}.
This \emph{intensional} concept is central in logic and computation yet hard to capture syntax-freely. 
Abramsky \cite{abramsky2014intensionality} describes this problem as
\begin{quote}
Extensionality is enshrined in mathematically precise axioms with a clear conceptual meaning. 
Intensionality, by contrast, remains elusive. 
It is a ``loose baggy monster" into which all manner of notions may be stuffed, and a compelling and coherent general framework for intensional concepts is still to emerge.
\end{quote}

Last but not least, syntactic concepts are often defined \emph{inductively}, lacking direct or intuitive descriptions. 
This inductive nature blocks a deeper analysis of logic or computation.
For instance, consider classical logic \cite{frege1879begriffsschrift}, intuitionistic logic \cite{heyting1930formalen} and linear logic \cite{girard1987linear}. 
They are very basic concepts in logic, but it is hard to compare or relate them in a direct or intuitive fashion because their proofs are defined only inductively. 

In addition to the syntactic nature, another problem in logic and computation is the lack of a \emph{common} framework, which makes it difficult to obtain a coherent view, a universal language or a shared mathematical technique on the fields.
For instance: 
\begin{itemize}

\item Formal proofs of a formula are not a formalisation of (ordinary) proofs of the validity of the formula, and they are only indirectly connected via soundness or completeness theorems.
This illustrates a \emph{schism} between proof theory and model theory.

\item Set theory is based on not only sets but also (classical) logic. 
In other words, sets do not serve as a single principle of mathematics since otherwise they would not require logic.

\item Recursion theory entails models of computations such as Turing machines \cite{turing1937computable}, which are \emph{extrinsic} to the structure of sets or logic \cite[\S 1.2]{abramsky2014intensionality}.
For instance, the computability of an arithmetic function depends on the existence of a Turing machine that computes the function, but the Turing machine itself is a priori extrinsic to the function.

\end{itemize}
This situation is in contrast with topology since general topology serves as a \emph{single} foundation of the subfields of topology (on the basis of sets and logic).
It is not only coherent but also elegant and significant if a single principle founds logic and computation, thus mathematics too. 
Such a principle will build new connections between their subfields and advance the fields as a whole. 

Finally, logic and computation, being foundations of mathematics, ought to be \emph{justifiable}. 
In particular, they should only rely on simple, primitive or \emph{atomic} concepts. 
For instance, Hilbert's program \cite{hilbert1931grundlegung} was motivated by this standpoint, and its aim was to reduce all mathematics to \emph{finitary} one.
The point of such finitary mathematics, in Hilbert's own words \cite{hilbert1926uber}, is that
\begin{quote}
If logical inference is to be reliable, it must be possible to survey these objects completely in all their parts, and the fact that they occur, that they differ from one another, and that they follow each other, or are concatenated, is immediately given intuitively, together with the objects, as something that can neither be reduced to anything else nor requires reduction. 
\end{quote}
Also, a justification of a framework for logic or computation should entail \emph{conceptual naturality} as well, but we do not regard this point as primary because it can be subjective. 

Motivated in this way, we propose:
\begin{tcolorbox}
{\bfseries Our research program.} To reduce logic and computation to single, syntax-free, finitary mathematics, viz., \emph{combinatorics}, that is conceptually natural for logic and computation. 
\end{tcolorbox}

The aim of the present work is to take a significant step for this research program:
\begin{tcolorbox}
{\bfseries The goal of this work.} To reformulate formal systems for classical, intuitionistic and linear propositional logics, the validity of a formula, cut-elimination and computability uniformly and naturally by a class of finite dynamic graphs, called \emph{combinatorial arenas}. 
\end{tcolorbox}

\begin{remark}
We leave it to another article to extend the present framework to other central concepts such as predicate logics, sets and computational complexity; see \S\ref{OurContributionsAndRelatedWork} for an outline.
\end{remark}


Our framework has the following features.
First, inherited from \emph{game semantics} \cite{hyland1997game}, a subfield of mathematical semantics, our approach is \emph{conceptually natural}.
Indeed, a combinatorial arena $\mathscr{A}$ represents a \emph{game} between an agent, called \emph{Player}, and an oracle, called \emph{Opponent}, and a class of walks on $\mathscr{A}$ alternating between Player and Opponent, called \emph{plays} on $\mathscr{A}$, depict their \emph{dialogical arguments} or \emph{computations}.
A class of algorithms for Player about how to play on $\mathscr{A}$ are called \emph{strategies} on $\mathscr{A}$, and a strategy is said to be \emph{winning} if it always leads to Player's `win.'
We then interpret formulae and types by combinatorial arenas, and proofs (respectively, programs) by winning strategies (respectively, strategies).
In this way, our combinatorics formalises the intuitive idea that reasoning and computing are certain kinds of playing games.

Besides, this approach resolves the schism between proof theory and model theory by defining a winning strategy to be not only a formal proof but also a \emph{validation} of a formula. 
This definition makes sense because intuitively a winning strategy defends the validity of a formula against refutations by Opponent.
That is, winning strategies can be seen, unlike other mathematical semantics, as an idealisation of proofs.
In addition to this compelling intuition, technical virtues of the unification are that the consistency of our combinatorial formal systems follows immediately from a nature of winning strategies, and it is vacant to ask their soundness or completeness.

While our approach is based on game semantics, it has distinguished features:

\begin{itemize}

\item Our method is \emph{combinatorial}, viz., it only uses finite, atomic structures intrinsic to combinatorial arenas, while game semantics is not (\S\ref{OurContributionsAndRelatedWork}).
This is desirable not only for foundational but also mathematical reasons: Combinatorial arenas form \emph{Hopf algebras}, and it is \emph{polynomial time decidable} whether a given strategy on a combinatorial arena is winning.


\item We establish bijections between proofs, which may contain \emph{cuts}, and winning strategies for classical, intuitionistic and linear logics, respectively, in a \emph{non-inductive} fashion (\S\ref{MainResults}).

\item Game semantics does not formalise cut-elimination or computability, but our combinatorics does both, significantly improving \emph{dynamic game semantics} \cite{yamada2020dynamic,yamada2019game} (\S\ref{MainResults}).


\end{itemize}

\begin{remark}
Although our combinatorial structures do not require anything infinite, we shall employ infinitary objects on the way for convenience but eventually dispense with them.
\end{remark}

\subsection{Main results}
\label{MainResults}
Concretely, we achieve the aim of \S\ref{Foreword} by proving the following theorems.
To state the theorems, we first need the following definitions. 
There are constructions on combinatorial arenas, \emph{negation} $\neg$, \emph{linear implication} $\multimap$, \emph{of-course} $\oc$ and \emph{why-not} $\wn$, coming from linear logic.
Intuitively, $\neg \mathscr{A}$ is the negation of $\mathscr{A}$, $\mathscr{A} \multimap \mathscr{B}$ is the space of linear maps from $\mathscr{A}$ to $\mathscr{B}$, $\oc \mathscr{A}$ is the countable copies of $\mathscr{A}$, and $\wn \mathscr{A}$ is $\mathscr{A}$ itself except that strategies on $\wn \mathscr{A}$ can \emph{backtrack} any number of times.

We then define a bicategory $\mathcal{LG}$, whose 0-cells are combinatorial arenas, 1-cells $\mathscr{A} \rightarrow \mathscr{B}$ are a class of finite winning strategies whose values under a combinatorial recast of cut-elimination are on $\mathscr{A} \multimap \mathscr{B}$, and 2-cells are the equivalence relation that identifies 1-cells up to the combinatorial cut-elimination.
By this bicategorical framework, which was originally introduced in \cite{yamada2020dynamic}, $\mathcal{LG}$ admits intensionality or \emph{cuts} in strategies. 
We also define a bicategory $\mathcal{LG}_{\neg\neg}$ (respectively, $\mathcal{LG}_\oc$, $\mathcal{LG}_{\oc\wn}$), whose 0-cells are a class of combinatorial arenas, 1-cells $\mathscr{A} \rightarrow \mathscr{B}$ are the same strategies whose values are on $\mathscr{A} \multimap \neg \neg \mathscr{B}$ (respectively, $\oc \mathscr{A} \multimap \mathscr{B}$, $\oc \wn \mathscr{A} \multimap \wn \mathscr{B}$), and 2-cells are as those in $\mathcal{LG}$.
Categorically, these bicategories are obtained from $\mathcal{LG}$ by (co-)Kleisli constructions.

Correspondingly, we define a syntactic bicategory $\mathsf{ILL}$ of intuitionistic linear logic, whose 0-cells are formulae, 1-cells $A \rightarrow B$ are (formal) proofs whose values under cut-elimination are on the sequent $A \vdash B$, and 2-cells are the equivalence relation that identifies 1-cells up to cut-elimination. 
In the same vein, we also define syntactic bicategories $\mathsf{CLL}$, $\mathsf{IL}$ and $\mathsf{CL}$ of classical linear, intuitionistic and classical logics, respectively.
We can now state our first theorem:

\begin{theorem*}[combinatorial formal systems]
There are biequivalences $\mathsf{ILL} \simeq \mathcal{LG}$, $\mathsf{CLL} \simeq \mathcal{LG}_{\neg\neg}$, $\mathsf{IL} \simeq \mathcal{LG}_\oc$ and $\mathsf{CL} \simeq \mathcal{LG}_{\oc\wn}$, and it is polynomial time decidable if a strategy is a 1-cell in all cases. 
\end{theorem*}

This theorem characterises formulae and proofs of the logics syntax-freely by the same combinatorics.
Also, for each case, it takes only polynomial time to check if a strategy is a 1-cell. 
Thus, our combinatorics constitutes \emph{proof (formal) systems} in the sense of \cite{cook1979relative}.
As a result, this theorem covers not only what is equivalent to Hughes' combinatorial reformulation of provability of classical logic \cite{hughes2006proofs} but also that of intuitionistic and linear logics \emph{uniformly}.

Whilst Hughes only characterises \emph{provability} of classical logic, our biequivalences $\llbracket \_ \rrbracket$ consist of bijections between strategies and \emph{proofs}.
Such a bijection is one of the strongest theorems in mathematical semantics, and it is often very difficult to establish. 
For instance, fully complete semantics of intuitionistic linear logic (with respect to cut-free proofs) was missing for thirty years (since the emergence of fully complete semantics of some fragments of linear logic \cite{abramsky1994games,hyland1993fair}), but the biequivalence $\mathsf{ILL} \simeq \mathcal{LG}$ solves this well-known problem.
Besides, while full completeness has been focusing on cut-free proofs, the biequivalence admits proofs with \emph{cuts}.
Moreover, albeit semantics is usually given inductively, our biequivalences directly and non-inductively read off proofs as strategies, and vice versa. 
These progresses are made possible by the novel structure of combinatorial arenas. 
Lastly, although there is no known intuitive reading of Hughes' approach, ours has one: A combinatorial arenas is a game on the truth of a formula, and a winning strategy is an algorithm for Player to defend the truth of a formula against refutations by Opponent.

For this intuition, it makes sense to define a formula $A$ to be \emph{valid} if there is a 1-cell $1 \rightarrow \llbracket A \rrbracket$ in the four bicategories, where $1$ is a terminal object that admits the isomorphism $\llbracket A \rrbracket \cong \llbracket A \rrbracket^1$.
This definition makes the existence of a proof and the validity of a formula \emph{coincide}.

Besides, our theorem characterises \emph{non-linearity} and \emph{classicality} of logic non-inductively by the (co-)Kleisli constructions.
This classification has an intuitive reading too: \emph{Non-linearity} of logic is obtained by the co-Kleisli construction $(\_)_\oc$, which enables strategies to \emph{consume inputs any number of times}, and \emph{classicality} by the Kleisli one $(\_)_{\wn}$, which permits strategies to \emph{backtrack any number of times}.
For example, the law of excluded middle, i.e., the disjunction $A \vee \neg A$ between each formula $A$ and its negation $\neg A$, is provable in classical logic but not in intuitionistic logic. 
Our method explains this difference non-inductively and intuitively as follows. 
If the same symbol $\vee$ denotes the corresponding operation on combinatorial arenas, then the law means the existence of a 1-cell $1 \rightarrow \llbracket A \rrbracket \vee \neg \llbracket A \rrbracket$.
There is no such 1-cell in $\mathcal{LG}_\oc$ as the choice between $\llbracket A \rrbracket$ and $\neg \llbracket A \rrbracket$ is not always decidable. 
In contrast, there \emph{is} one in $\mathcal{LG}_{\oc\wn}$ since why-not $\wn$ on $\llbracket A \rrbracket \vee \neg \llbracket A \rrbracket$ allows the 1-cell to go back and forth between $\llbracket A \rrbracket$ and $\neg \llbracket A \rrbracket$ (without deciding which).
That is, classical logic permits backtracks or \emph{reasoning do-overs}, while intuitionistic logic does not.

Next, our second theorem follows immediately from the biequivalences that subsume cuts:
\begin{theorem*}[combinatorial cut-elimination]
The cut-elimination on 1-cells in $\mathcal{LG}$ (respectively, $\mathcal{LG}_{\neg\neg}$, $\mathcal{LG}_\oc$, $\mathcal{LG}_{\oc\wn}$) corresponds precisely to the one in $\mathsf{ILL}$ (respectively, $\mathsf{CLL}$, $\mathsf{IL}$, $\mathsf{CL}$). 
\end{theorem*}

To the best of our knowledge, this is the first syntax-free characterisation of cut-elimination, while there were partial solutions in the literature.
For example, geometry of interaction \cite{girard1989geometry} models cut-elimination for the first time, but it deletes all (semantic) cuts \emph{in one go}; dynamic game semantics \cite{yamada2020dynamic} improves this pioneering work by modelling \emph{step-by-step} reduction, but it is still much coarser than cut-elimination. 
In contrast, our second theorem completely captures cut-elimination, including its \emph{finest} computational steps. 
Recall that the bottleneck in releasing logic from the syntactic occupation is dynamics and intensionality (\S\ref{Foreword}).
Our first and second theorems solve this problem (on propositional logics) entirely for the first time. 

Further, our combinatorics, despite its finitary nature, forms a strong model of higher-order computation even without relying on any other models of computation. 
This is possible roughly because the combinatorics has intensional structures sufficient to define computation.
Concretely, we obtain bicategories $\mathcal{G}$, $\mathcal{G}_{\neg\neg}$, $\mathcal{G}_\oc$ and $\mathcal{G}_{\oc\wn}$ respectively from $\mathcal{LG}$, $\mathcal{LG}_{\neg\neg}$, $\mathcal{LG}_\oc$ and $\mathcal{LG}_{\oc\wn}$ by removing the winning constraint on 1-cells and replacing their finiteness with \emph{finite presentability} in a suitable sense. 
1-cells in these bicategories model computations, more general than proofs, and in particular some of them are non-terminating or \emph{partial}.
Then, our last theorem is

\begin{theorem*}[combinatorial computation]
The bicategory $\mathcal{G}_\oc$ forms a model of computation that can simulate the higher-order functional programming language PCF \cite{scott1993type,plotkin1977lcf}. 
\end{theorem*}

This theorem significantly improves Yamada \cite{yamada2019game}, which shows that strategies \emph{presentable by finitely presentable strategies} can simulate PCF, as our theorem proves that finitely presentable strategies suffice.
The mechanism behind this improvement is that our combinatorial approach frees strategies from the computation on infinitely many copies of games in \cite{yamada2019game}.
Another notable feature of this novel model of computation is its semantic or \emph{abstract} nature: It is free from the symbolic computation of Turing machines or $\lambda$-calculi. 
This feature is quite desirable because it saves us from being bothered by inessential details in symbolic computation. 

The bicategories $\mathcal{G}$, $\mathcal{G}_{\neg\neg}$ and $\mathcal{G}_{\oc\wn}$ model computation too, where $\mathcal{G}_{\oc \wn}$ interestingly combines computation and classical reasoning. 
We leave it as future work to analyse these bicategories.


Last but not least, our combinatorics gives rise to a well-known algebraic structure:
\begin{proposition*}[combinatorial Hopf algebras]
Combinatorial arenas constitute Hopf algebras. 
\end{proposition*}

Over the past decades, Hopf algebras arising in combinatorics, or \emph{combinatorial Hopf algebras} \cite{joni1979coalgebras,aguiar2006combinatorial}, have been extensively studied. 
The present result uncovers a bridge between our combinatorics and this vibrant line of research. 
In addition to the value of connecting previously unrelated notions, this bridge enables one to apply methods and results in combinatorial Hopf algebras to logic and computation, e.g., for counting the number of formulae and proofs.

\subsection{Our contributions and related work}
\label{OurContributionsAndRelatedWork}
Our first contribution is the \emph{biequivalences} between formal systems and combinatorics. 
The main breakthrough here is that the biequivalences admit proofs with \emph{cuts}, which in turn enables the combinatorics to model \emph{cut-elimination}. 
In this way, we resolve the bottleneck in releasing logic from the syntactic occupation, i.e., dynamics and intensionality, completely for the first time (to the best of our knowledge).
In a broader perspective, this result establishes a framework that tames the `loose baggy monster' cited in \S\ref{Foreword} so that mathematics may extend its scope from static, extensional structures such as sets and functions to dynamic, intensional ones. 

The significance of the biequivalences is visible even if one focuses on cut-free proofs: Fully complete semantics of intuitionistic linear logic with respect to cut-free proofs was open for thirty years, but a solution to this long-standing problem just follows from the cut-free part of the biequivalence $\mathsf{ILL} \simeq \mathcal{LG}$. 
This result `beats the end boss' since even the fully complete semantics of classical linear logic \cite{abramsky1999concurrent,mellies2005asynchronous} and of the multiplicative fragment of intuitionistic linear logic \cite{murawski2003exhausting} was established about twenty years ago. 
Another implication of the biequivalences is that they provide a direct reading of our combinatorics as formal calculi equipped with cut-elimination.
In this fashion, our method retains the \emph{mechanical} nature of syntax. 

Our second contribution is to extend the combinatorics for logic to computation.
Thus, we have achieved our goal to reduce logic and computation \emph{uniformly} to combinatorics (\S\ref{Foreword}).
The significance of the resulting model of computation is that it greatly improves the main result of Yamada \cite{yamada2019game} by showing that finitely presentable strategies suffice to simulate PCF.
Also, the computational steps of our model are more \emph{explicit} and \emph{atomic} than those of \cite{yamada2019game}, similarly to Turing machines yet in a \emph{non-symbolic}, \emph{higher-order} setting. 
Let us leave it as future work to extend this framework to a basis of \emph{higher-order computational complexity}, whose mathematical foundation has not been established yet, by exploiting the explicit, atomic nature.  

Our last contribution is the Hopf algebras induced by combinatorial arenas. 
This result uncovers a link between our combinatorics and a well-known algebraic structure. 
We add that combinatorial arenas also form Hausdorff topological spaces, and strategies continuous maps, while \emph{domains}, the best-known structure in mathematical semantics, only give rise to non-Hausdorff topological spaces.
Although we leave it as future work to explore these connections, they indicate fruitful interplays between our combinatorics and traditional branches of mathematics. 

Game semantics has been extended to \emph{Martin-L\"{o}f type theory (MLTT)} \cite{abramsky2015games,yamada2019game}, which subsumes intuitionistic higher-order predicate logic, and Aczel \cite{aczel1986type} has shown that constructive set theory is translatable into MLTT.
We shall therefore extend the present framework to predicate logics and set theory through the game semantics of MLTT. 

For related work, we have made comparisons with Hughes' combinatorial recast of provability of classical logic \cite{hughes2006proofs}, geometry of interaction \cite{girard1989geometry}, dynamic game semantics \cite{yamada2020dynamic} and the game-semantic model of computation \cite{yamada2019game} (\S\ref{MainResults}).
In the following, we list related work in game semantics.
A major approach pioneered by Hyland and Ong \cite{hyland2000full} or \emph{HO} is to replace games with a class of finite rooted dags, called \emph{arenas}, and then \emph{derive} possible developments or \emph{positions} in an arena as a class of walks on the arena. 
This approach does not require anything infinite, and our method is based on HO for this reason.
However, HO cannot model linear logic, in particular copying $\oc$ nor backtracking $\wn$ which are crucial for this work, due to the lack of a structure to control positions. 
Besides, positions in an arena may grow \emph{infinitely} even when they model finitary objects such as natural numbers, and this problem makes it impossible for HO to model a \emph{termination} of a play.  
Our initial idea was to shift from arenas to combinatorial arenas, or from dags to dynamic simple graphs, so that it solves the aforementioned problems of HO.

Another major approach pioneered by Abramsky et al. \cite{abramsky1994games,abramsky2000full} or \emph{AJM} is to \emph{assign} finite sequences to a game as positions in the game. 
This method is more abstract and general than HO or ours as it does not go through (combinatorial) arenas. 
Meanwhile, the AJM method entails \emph{infinite} sets of positions.
Also, while our combinatorics achieves semantics not only fully faithful but also \emph{essentially surjective on objects}, it is unclear how the AJM approach does it.

Based on AJM, Murawski and Ong \cite{murawski2003exhausting} achieved fully complete game semantics of the multiplicative fragment of intuitionistic linear logic.
Laurent \cite{laurent2004polarized,laurent2005syntax} constructed, on the basis of HO and AJM, fully complete game semantics of the \emph{polarised} fragment of linear logic.
His games, however, do not even give rise to a category. 
The fully complete semantics of classical linear logic \cite{abramsky1999concurrent,mellies2005asynchronous} uses \emph{concurrent} games, while our and the aforementioned semantics are based on \emph{sequential} games. 
Similarly to AJM, concurrent games entail infinitary structures.  


\subsection{Structure of the present article}
\label{StructureOfThePresentArticle}
We first introduce combinatorial arenas and finitely presentable strategies in \S\ref{CombinatorialSemiArenasAndStrategies}.
We next recall formal systems for the logics and then establish the biequivalences between the syntax and the semantics in \S\ref{CombinatorialFormalSystems}.
Finally, we show that finitely presentable strategies suffices for PCF in \S\ref{Computation}.

\begin{convention}
We employ the following conventions:
\begin{itemize}

\item We write $\wp(X)$ for the \emph{power set} of a set $X$, and $|X|$ for the \emph{cardinality} of $X$; 

\item We use bold small letters $\boldsymbol{s}, \boldsymbol{t}, \boldsymbol{u}, \boldsymbol{v}$, etc. for sequences, in particular $\boldsymbol{\epsilon}$ for the \emph{empty sequence}, and small letters $a, b, m, n, x, y$, etc. for elements of sequences;

\item We write $X^\ast$ for the set of all finite sequences of elements of $X$, and given a map $f : X \rightarrow Y$ we write $f^\ast : X^\ast \rightarrow Y^\ast$ for its free-monoid map $x_1x_2 \dots x_n \mapsto f(x_1) f(x_2) \dots f(x_n)$;


\item We define $\overline{n} \colonequals \{ \, 1, 2, \dots, n \, \}$ for each $n \in \mathbb{N}_+ \colonequals \mathbb{N} \setminus \{ 0 \}$, and $\overline{0} \colonequals \emptyset$;

\item We write $x_1 x_2 \dots x_{|\boldsymbol{s}|}$ for a sequence $\boldsymbol{s} = (x_1, x_2, \dots, x_{|\boldsymbol{s}|})$, with $x^n \colonequals \underbrace{x x \dots x}_n$, where $|\boldsymbol{s}|$ is the number of elements or \emph{length} of $\boldsymbol{s}$, and define $\boldsymbol{s}(i) \colonequals x_i$ ($i \in \overline{|\boldsymbol{s}|}$) and $\boldsymbol{s}_{\leqslant i} \colonequals x_1 x_2 \dots x_i$;

\item If $L$ is a set of finite sequences, then $\underline{L} \colonequals \bigcup_{\boldsymbol{s} \in L} \underline{\boldsymbol{s}}$, where $\underline{\boldsymbol{s}} \colonequals \{ \, \boldsymbol{s}(i) \mid i \in \overline{|\boldsymbol{s}|} \, \}$;


\item A \emph{concatenation} of finite sequences $\boldsymbol{s}$ and $\boldsymbol{t}$ is represented by their juxtaposition $\boldsymbol{s}\boldsymbol{t}$ (or $\boldsymbol{s} . \boldsymbol{t}$), but we often write $a \boldsymbol{s}$, $\boldsymbol{t} b$, $\boldsymbol{u} c \boldsymbol{v}$ for $(a) \boldsymbol{s}$, $\boldsymbol{t} (b)$, $\boldsymbol{u} (c) \boldsymbol{v}$, and so on;




\item We write $\mathrm{Even}(\boldsymbol{s})$ (respectively, $\mathrm{Odd}(\boldsymbol{s})$) if $|\boldsymbol{s}|$ is even (respectively, odd), and define $S^\mathrm{P} \colonequals \{ \, \boldsymbol{s} \in S \mid \mathrm{P}(\boldsymbol{s}) \, \}$ for a set $S$ of sequences and a predicate $\mathrm{P} \in \{ \mathrm{Even}, \mathrm{Odd} \}$;

\item We write $\boldsymbol{s} \preceq \boldsymbol{t}$ if $\boldsymbol{s}$ is a \emph{prefix} of a sequence $\boldsymbol{t}$, and given a set $S$ of sequences, $\mathrm{Pref}(S)$ for the set of all prefixes of sequences in $S$, i.e., $\mathrm{Pref}(S) \colonequals \{ \, \boldsymbol{s} \mid \exists \boldsymbol{t} \in S . \, \boldsymbol{s} \preceq \boldsymbol{t} \, \}$.



\end{itemize}
\end{convention}

\section{Combinatorial arenas and finitely presentable strategies}
\label{CombinatorialSemiArenasAndStrategies}
This section introduces two central concepts of the present work, \emph{combinatorial arenas} (\S\ref{CombinatorialSemiArenas}--\ref{CombinatorialSequents}) and \emph{finitely presentable strategies}  (\S\ref{FinitelyPresentableStrategies}), and studies their basic properties.
These combinatorial structures recast dynamic game semantics \cite{yamada2020dynamic}, yet requiring nothing infinite, nonatomic or extrinsic to them, and constitute our basis.
We conclude the present section with bicategories of combinatorial arenas and finitely presentable strategies for logic and computation (\S\ref{Categories}).

\subsection{Combinatorial arenas}
\label{CombinatorialSemiArenas}
We begin with a review of some basic concepts in graph theory that are necessary for defining combinatorial arenas.
Some of them can be infinitary, but we later focus on finitary ones. 

An \emph{\bfseries edge} is any two-element set $\{ x, y \}$, and it is said to be \emph{\bfseries on} a set $S$ if $x, y \in S$.
A \emph{\bfseries (simple) graph} is a pair $G = (\mathscr{V}_G, \mathscr{E}_G)$ of a set $\mathscr{V}_G$, whose elements are called \emph{\bfseries vertices}, and a set $\mathscr{E}_G$ of edges on $\mathscr{V}_G$.
A graph $G$ is said to be \emph{\bfseries finite} (respectively, \emph{\bfseries empty}) if so is $\mathscr{V}_G$, \emph{\bfseries null} if $\mathscr{E}_G = \emptyset$, and \emph{\bfseries complete} if $\{ x, y \} \in \mathscr{E}_G$ for all $x, y \in \mathscr{V}_G$ with $x \neq y$.
A graph $H$ is called a \emph{\bfseries subgraph} of $G$, written $H \sqsubseteq G$, if $\mathscr{V}_H \subseteq \mathscr{V}_G$ and $\mathscr{E}_H \subseteq \mathscr{E}_G$.
We always assume $\mathscr{V}_G \cap \wp(\mathscr{V}_G) = \emptyset$. 

A \emph{\bfseries path} in $G$ is a finite sequence $v_0v_1 \dots v_{n} \in \mathscr{V}_G^\ast$ of pairwise distinct vertices with $\{ v_i, v_{i+1} \} \in \mathscr{E}_G$ for $i = 0, 1, \dots, n-1$. 
A nonempty graph is said to be \emph{\bfseries connected} if it has a path between each pair of vertices. 
If $\boldsymbol{s} = v_0v_1 \dots v_{n-1}$ is a path in $G$ with $n \geqslant 3$, then $\boldsymbol{s}v_0$ is called a \emph{\bfseries cycle} in $G$.
A graph is said to be \emph{\bfseries acyclic} if it does not contain a cycle.

A \emph{\bfseries tree} is a connected, acyclic graph, or equivalently a graph with exactly one path between each pair of vertices. 
A \emph{\bfseries rooted tree} is a tree $T$ together with a distinguished vertex $\mathrm{rt}_T$, called the \emph{\bfseries root}, and a \emph{\bfseries rooted forest} is a disjoint union $F$ of rooted trees.
The partial order $\leqslant_F$ on $\mathscr{V}_F$ defines $x \leqslant_F y$ if $x$ is in the path from a root to $y$.
If $p \leqslant_F c$ and $\{ p, c \} \in \mathscr{E}_F$, then $p$ is called a \emph{\bfseries parent} of $c$, or $c$ a \emph{\bfseries child} of $p$, in $F$, where $p \rightarrow_F c$ or $p \rightarrow c$ denotes the edge $e \colonequals \{ p, c \}$ with $\mathrm{src}_F(e) \colonequals p$ and $\mathrm{tgt}_F(e) \colonequals c$ called its \emph{\bfseries source} and \emph{\bfseries target}, respectively. 
A \emph{\bfseries leaf} of $F$ is a vertex of $F$ with no children.\footnote{A leaf can be a \emph{root} with no children. We adopt this definition for convenience; e.g., see Definition~\ref{DefAdditiveStructures}.} 
Two vertices of $F$ are called \emph{\bfseries siblings} in $F$ if they are both roots or have the same parent, and a subset $S \subseteq \mathscr{V}_F$ is said to be \emph{\bfseries fraternal} in $F$ if $S \neq \emptyset$ and its elements are pairwise siblings in $F$.
The \emph{\bfseries depth} $\mathrm{dep}_F(x)$ of $x \in \mathscr{V}_F$ is the length of the path from a root to $x$, lifted to $f \in \mathscr{E}_F$ by $\mathrm{dep}_F(f) \colonequals \mathrm{dep}_F(\mathrm{src}_T(f))$, and to $F$ by $\mathrm{dep}(F) \colonequals \sup(\{ \, \mathrm{dep}_F(x) \mid x \in \mathscr{V}_F \, \})$.

Let $\mathscr{V}_F^{i} \subseteq \mathscr{V}_F$ ($i \in \{ 0 \} \cup \overline{\mathrm{dep}(F)}$) and $\mathscr{E}_F^{j-1} \subseteq \mathscr{E}_F$ ($j \in \overline{\mathrm{dep}(F)}$)  consist of elements of depth $i$ and $j-1$, respectively; let $\mathscr{V}_F^{\geqslant i} \colonequals \bigcup_{k = i}^{\mathrm{dep}(F)}\mathscr{V}_F^k$ and $\mathscr{E}_F^{\geqslant j-1} \colonequals \bigcup_{l = j-1}^{\mathrm{dep}(F)-1}\mathscr{E}_F^l$.


\begin{notation}
For readability, we omit \emph{tags} for disjoint union $\uplus$. 
For instance, given sets $A$ and $B$, we write $x \in A \uplus B$ if $x \in A$ or $x \in B$; also, if $E$ and $F$ are respectively sets of edges on $A$ and $B$, then we write $E \uplus F$ for the disjoint union of $E$ and $F$ whose elements are edges on $A \uplus B$.

We represent a rooted forest by the triple $F = (\mathscr{R}_F, \mathscr{V}_F, \mathscr{E}_F)$ of the set $\mathscr{R}_F$ of all roots, the set $\mathscr{V}_F$ of all vertices and the set $\mathscr{E}_F$ of all edges. 
If $\mathscr{R}_F$ is a singleton set $\mathscr{R}_F = \{ \mathrm{rt}_F \}$, then we abbreviate it as $\mathrm{rt}_F$.
We simply write $\mathscr{R}_F$ for $F$ if $\mathscr{R}_F = \mathscr{V}_F$ and $\mathscr{E}_F = \emptyset$.
We also depict a rooted forest in the usual diagrammatic form, where edges are written $\rightarrow$ or $\Rightarrow$.

Given rooted forests $F$ and $G$, an element $\star$ and sets $X$ and $Y$ such that the intersection $X_F \colonequals X \cap \mathscr{V}_F$ is fraternal in $F$, i.e., the set $\mathrm{prt}(X_F) \colonequals \bigcup_{x \in X_F} \{ \, p \in \mathscr{V}_F \mid p \rightarrow_F x \, \}$ of all parents of elements of $X_F$ is empty or singleton, we define the following rooted forests:
\begin{itemize}

\item $\star . F \colonequals (\star, \{ \star \} \uplus \mathscr{V}_F, \{ \, \star \rightarrow r \mid r \in \mathscr{R}_F \, \} \uplus \mathscr{E}_F)$, i.e., $\star . F$ is the rooted tree obtained from $F$ by disjointly adding the element $\star$ as the new root;

\item $F \cup G \colonequals (\mathscr{R}_F \cup \mathscr{R}_G, \mathscr{V}_F \cup \mathscr{V}_G, \mathscr{E}_F \cup \mathscr{E}_G)$, i.e., $F \cup G$ is the \emph{union} of $F$ and $G$;

\item $F \uplus G \colonequals (\mathscr{R}_F \uplus \mathscr{R}_G, \mathscr{V}_F \uplus \mathscr{V}_G, \mathscr{E}_F \uplus \mathscr{E}_G)$, i.e., $F \uplus G$ is the \emph{disjoint union} of $F$ and $G$;

\item $F_X \colonequals (X_F, \bigcup_{x_0 \in X_F}\{ \, x \in \mathscr{V}_F \mid x_0 \leqslant_F x \, \}, \bigcup_{x_0 \in X_F}\{ \, e \in \mathscr{E}_F \mid x_0 \leqslant_F \mathrm{src}_F(e) \, \})$, i.e., $F_X$ is the largest rooted \emph{subforest} (i.e., a subgraph that is a forest) of $F$ with roots in $X_F$;


\item $F_X^\bullet \colonequals F_X$ if $\mathrm{prt}(X_F) = \emptyset$, and $F_X^\bullet \colonequals r . F_X$ if $\mathrm{prt}(X_F) = \{ r \}$;

\item $F{\upharpoonright_Y} \colonequals (\{ \, y \in Y_F \mid \forall y_0 \in Y_F . \, y_0 \not \rightarrow_F y \, \}, Y_F,  \mathscr{E}_F \cap \wp(Y_F))$, i.e., $F{\upharpoonright_Y}$ is the rooted subforest of $F$ induced by $Y_F$.

\end{itemize}
\end{notation}

Next, recall that a \emph{\bfseries partition} of a finite set $S$ is a (necessarily finite) set $\{ S_i \}_{i \in \overline{n}}$ of nonempty subsets $S_i \subseteq S$ such that $\bigcup_{i = 1}^nS_i = S$ and $S_i \cap S_j = \emptyset$ for all $i, j \in \overline{n}$.
We generalise a partition to a \emph{\bfseries recursive partition} of $S$, which is a (necessarily finite) rooted tree $P$ such that
\begin{enumerate}

\item The root is $S$, and other vertices are nonempty subsets of $S$;


\item If a vertex $V$ has children $V_i$ ($i \in \overline{n}$), then the set $\{ V_i \}_{i \in \overline{n}}$ is a partition of $V$ with $n > 1$.

\end{enumerate}
The recursive partition $P$ is said to be \emph{\bfseries empty} if $S = \emptyset$, \emph{\bfseries trivial} if $S \neq \emptyset$ and $\mathrm{dep}(P)$ = $0$, and \emph{\bfseries exhaustive} if its leaves are all singleton sets.

Having recalled these preliminary concepts, let us proceed to the central concept of combinatorial arenas (Definition~\ref{DefCombinatorialArenas}).
Roughly, a combinatorial arena is a finite rooted forest equipped with three combinatorial structures that correspond respectively to \emph{multiplicatives}, \emph{additives} and \emph{exponentials} in linear logic \cite{girard1987lazy}.
We first introduce these auxiliary structures.

First, the structure for multiplicatives is to specify a class of finite sequences of vertices of the underlying finite rooted forest, which are to serve as \emph{moves} in a game:
\begin{definition}[multiplicative structures]
\label{DefMultiplicativeStructures}
A \emph{\bfseries multiplicative structure} on a (necessarily finite) rooted forest $F$ is a finite rooted forest $\mu$ with $\mathscr{V}_\mu = \mathscr{V}_F$ such that the set $\underline{\boldsymbol{s}}$ for each nonempty path $\boldsymbol{s}$ in $\mu$ is fraternal in $F$. 
A vertex of $F$ is said to be \emph{\bfseries switching} in $\mu$ if it is not a leaf in $\mu$.
\end{definition}

\begin{example}
\label{ExMultSt}
Consider the finite rooted forest
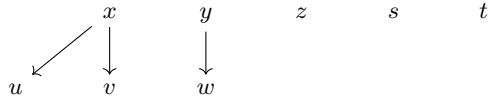
\begin{small}
\begin{mathpar}
\begin{tikzcd}
& \arrow[ld] \arrow[d] x & \arrow[d] y & z & s & t \\
u & v & w
\end{tikzcd}
\end{mathpar} 
\end{small}for which we write $F$.
For instance, the following two finite rooted forests 
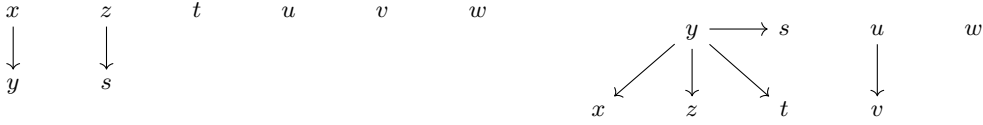
\begin{small}
\begin{mathpar}
\begin{tikzcd}
x \arrow[d] & z \arrow[d] & t & u & v & w \\
y & s \\
\end{tikzcd}
\and
\begin{tikzcd}
& \arrow[ld] y \arrow[d] \arrow[rd] \arrow[r] & s & u \arrow[d] & w \\
x & z & t & v
\end{tikzcd}
\end{mathpar}
\end{small}are both multiplicative structures on $F$.
For a record purpose, we write $\mu$ for the right one. 
For convenience, we represent the pair $(F, \mu)$ by the union of $F$ and $\mu$ with edges of $\mu$ written $\Rightarrow$ to distinguish them from those $\rightarrow$ of $F$, i.e., 
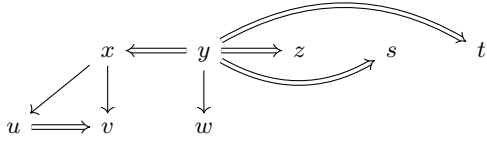
\begin{small}
\begin{mathpar}
\begin{tikzcd}
& \arrow[ld] \arrow[d] x & \arrow[l, Rightarrow] \arrow[d] y \arrow[r, Rightarrow] \arrow[rr, Rightarrow, bend right] \arrow[rrr, Rightarrow, bend left] & z & s & t \\
u \arrow[r, Rightarrow] & v & w
\end{tikzcd}
\end{mathpar} 
\end{small}in which $y$ and $u$ are switching in $\mu$.
We employ this representation throughout this article. 

On the other hand, for counterexamples, neither of the two graphs
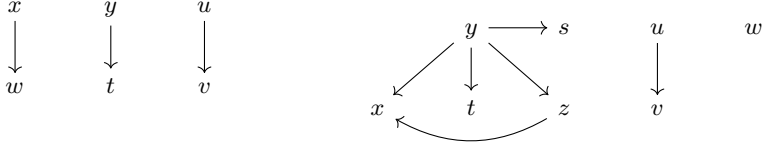
\begin{small}
\begin{mathpar}
\begin{tikzcd}
x \arrow[d] & y \arrow[d] & u \arrow[d] \\
w & t & v \\
\end{tikzcd}
\and
\begin{tikzcd}
& \arrow[ld] y \arrow[d] \arrow[rd] \arrow[r] & s & u \arrow[d] & w \\
x & t & \arrow[ll, bend left] z & v
\end{tikzcd}
\end{mathpar}
\end{small}is a multiplicative structure on $F$; the left one does not satisfy the first nor the second axiom of Definition~\ref{DefMultiplicativeStructures}, and the right one is not even a (simple) forest. 
\end{example}

The intuition is that a finite rooted forest $F$ determines components of a game similarly to an \emph{arena} \cite{hyland2000full} except that nonempty maximal paths in a multiplicative structure $ \mu$ on $F$, not vertices of $F$, serve as \emph{\bfseries moves} in the game.
Arenas are a class of rooted \emph{dags}, where implication $\Rightarrow$ between arenas entails the use of dags (going beyond simple graphs). 
Yet, it makes our dynamic graph method introduced later much simpler to stay in rooted (simple) forests since dynamics on them can be defined in terms of that on vertices. 
For this reason, we recast implication via multiplicative structures in such a way that it preserves rooted forests (Definition~\ref{DefConstructionsOnCombinatorialArenas}).

More specifically, a switching vertex functions as a \emph{hub} between the domain and the codomain of our implication (Definition~\ref{DefConstructionsOnCombinatorialArenas}), and this structure enables combinatorial arenas to stay in rooted forests under implication. 
Moreover, the distinction between switching and non-switching vertices is, albeit plain, responsible for not only this simple graph approach but also many of the present results, which existing game semantics in the literature could not attain, as we shall see. 
For instance, the second part of Example~\ref{ExCombinatorialSequents} illustrates that the distinction is essential for our combinatorial characterisation of formal systems embracing cuts or \emph{intensionality} (\S\ref{Bijections}).

The first axiom $\mathscr{V}_\mu = \mathscr{V}_F$ ensures that the set $\mathscr{V}_F$ has no redundancy, i.e., each vertex of $F$ is used for a move.
The second axiom facilitates the inheritance of \emph{pointers} from HO \cite{hyland2000full} to our setting (Definition~\ref{DefJustifiedSequences}).
Besides, this axiom guarantees that every move consists of vertices of the same depth in $F$.
If it consists of vertices of odd-depth in $F$, then it is called an \emph{\bfseries Opponent's move} or \emph{\bfseries O-move}, and otherwise a \emph{\bfseries Player's move} or \emph{\bfseries P-move}.

\begin{convention}
We call the pair $(F, \mu)$ of a finite rooted forest $F$ and a multiplicative structure $\mu$ on $F$ a \emph{\bfseries multiplicative pair}, and write $\mathcal{M}(\mu)$ for the set of all moves, i.e., nonempty maximal paths, in $\mu$.
Because we often talk about sequences of moves, we use the vector notation $\vec{m}$ for moves (instead of the bold letter $\boldsymbol{m}$) so that those sequences are written $\boldsymbol{s} = \vec{m}_1 \vec{m}_1 \dots \vec{m}_n$.
\end{convention}

Next, the structure for additives is to determine moves to be \emph{unavailable}:
\begin{definition}[additive structures]
\label{DefAdditiveStructures}
An \emph{\bfseries additive structure} on a multiplicative pair $(F, \mu)$ is the (necessarily disjoint) union $\alpha = \alpha[0] \cup \alpha[2]$ of a set $\alpha[0]$ of leaves in $F$ that are not switching in $\mu$ and a set $\alpha[2]$ of edges between siblings in $\mu$.
\end{definition}

\begin{example}
\label{ExAddSt}
The set $\alpha \colonequals \{ z, w, \{ x, z \}, \{ x, t \}, \{ z, t \} \}$ is an additive structure on the multiplicative pair $(F, \mu)$ given in Example~\ref{ExMultSt}.
For a counterexample, the set $\{ \{ x, y \}, \{ y, w \} \}$ is not an additive structure on the multiplicative pair $(F, \mu)$ since $y$ and $w$ are not siblings in $\mu$. 
\end{example}

The idea of an additive structure $\alpha$ on a multiplicative pair $(F, \mu)$ is: Each move $\vec{m} \in \mathcal{M}(\mu)$ such that $\underline{\vec{m}} \cap \alpha[0] \neq \emptyset$ \emph{can never be played} in a game; when a move $\vec{n} \in \mathcal{M}(\mu)$ with $x \in \underline{\vec{n}}$ and $\{ x, y \} \in \alpha[2]$ is made, each move containing a vertex of $F_{\{ y \}}$ becomes \emph{unavailable}.
Such graph-theoretic \emph{dynamics} of $\alpha$ is to model additives in linear logic in a finitary way, where the subset $\alpha[0] \subseteq \alpha$ corresponds to 0-ary additives, and the other $\alpha[2] \subseteq \alpha$ to binary ones (Definition~\ref{DefConstructionsOnCombinatorialArenas}).
We make the dynamics of $\alpha$ precise in \S\ref{Positions}.
The axioms on $\alpha$ are for Theorem~\ref{ThmFreeCharacterisation} to hold. 

Finally, the structure for exponentials is to define what to be \emph{duplicated} during a play:
\begin{definition}[exponential structures]
An \emph{\bfseries exponential structure} on a multiplicative pair $(F, \mu)$ is a finite rooted forest $\epsilon$ whose vertices are pairs $(S, i)$ of a fraternal set $S$ in $\mu$ and a natural number $i$ unique among the second elements of vertices of $\epsilon$, and edges $(S, i) \rightarrow (T, j)$ are the superset relation $S \supseteq T$ such that the first elements of roots of $\epsilon$ are pairwise disjoint.
\end{definition}

\begin{convention}
We use the pair $(S, i)$, not the set $S$ itself, as a vertex of an exponential structure because $S$ may occur more than once.
Nevertheless, as in the case of occurrences in a sequence, we usually abbreviate $(S, i)$ as $S$; this convention does not bring confusion in practice. 
\end{convention}

\begin{notation}
For a vertex $S$ of an exponential structure $\epsilon$ on a multiplicative pair $(F, \mu)$, we define $\overline{S} \colonequals \mathscr{V}_{\mu_{S}} = \bigcup_{x_0 \in S} \{ \, x \in \mathscr{V}_{F} \mid x_0 \leqslant_{\mu} x \, \}$.
\end{notation}

\begin{example}
\label{ExExpSt}
Consider again the multiplicative pair $(F, \mu)$ in Example~\ref{ExMultSt}.
The rooted forest
\begin{small}
\begin{mathpar}
\{ x, z \} \rightarrow \{ x \}
\and
\{ v \} \rightarrow \{ v \}
\and
\{ w \}
\end{mathpar}
\end{small}written $\epsilon$, is an exponential structure on $(T, \mu)$. 
Strictly speaking, the two occurrences of $\{ v \}$ are distinguished by two distinct natural numbers attached on them. 
On the other hand, the rooted tree $\{ u, v \} \rightarrow \{ v \}$, written $\epsilon'$, is not because the vertex $\{ u, v \}$ is not fraternal in $\mu$.
\end{example}

To explain the idea of an exponential structure $\epsilon$ on a multiplicative pair $(F, \mu)$, let us define the finite sequence $\epsilon(x)$ of fraternal sets in $\mu$ for each $x \in \mathscr{V}_F$ by $\epsilon(x) \colonequals \epsilon(x)_1 \epsilon(x)_2 \dots \epsilon(x)_n$ if there is the longest nonempty path $\epsilon(x)_1 \supseteq \epsilon(x)_2 \supseteq \dots \supseteq \epsilon(x)_n$ in $\epsilon$ whose elements $\epsilon(x)_j$, $1 \leqslant j \leqslant n$, all contain $x$, and by $\epsilon(x) \colonequals \boldsymbol{\epsilon}$ otherwise.
Then, the idea is: When a move $\vec{m} \in \mathcal{M}(\mu)$ is made during a play, the rooted subforests $F_{\overline{\epsilon(x)_j}}$ of $F$ are \emph{duplicated} in a suitable way for each $x \in \underline{\vec{m}}$.
This graph-theoretic \emph{dynamics} of $\epsilon$ is to model exponentials in a finitary way (\S\ref{Positions}). 

We are now ready to introduce our combinatorial reformulation of games:
\begin{definition}[combinatorial arenas]
\label{DefCombinatorialArenas}
A \emph{\bfseries combinatorial arena} $\mathscr{A}$ is a multiplicative pair $(|\mathscr{A}|, \mu_{\mathscr{A}})$ together with an additive structure $\alpha_{\mathscr{A}}$ and an exponential structure $\epsilon_{\mathscr{A}}$ on $(|\mathscr{A}|, \mu_{\mathscr{A}})$ that admits an exhaustive recursive partition $P_S$ of each maximal fraternal set $S$ in $\mu_{\mathscr{A}}$ such that, for each partition $S_I = \{ S_i \}_{i \in I}$ occurring in $P_S$, the following holds: 
\begin{enumerate}

\item The relation $[\alpha_{\mathscr{A}}[2]]_I$ between elements $S_i, S_j \in S_I$ ($i, j \in I$) given in terms of arbitrarily chosen representatives $x \in S_i$ and $y \in S_j$ by
\begin{equation*}
[\alpha_{\mathscr{A}}[2]]_I(S_i, S_j) \ratio \Leftrightarrow \{ x, y \} \in \alpha_{\mathscr{A}}[2] 
\end{equation*}
is well-defined, i.e., independent of the choice of the representatives $x$ and $y$;

\item The graph $(S_I, [\alpha_{\mathscr{A}}[2]]_I)$ is null or complete;

\item $V \cap S_i \neq \emptyset$ implies $V \subseteq S_i$ or $S_i \subseteq V$ for all $V \in \mathscr{V}_{\epsilon_{\mathscr{A}}}$ and $i \in I$.

\end{enumerate}

\emph{\bfseries Vertices} and \emph{\bfseries edges} of $\mathscr{A}$ are those of $|\mathscr{A}|$, and \emph{\bfseries moves} in $\mathscr{A}$ are elements of $\mathcal{M}_{\mathscr{A}} \colonequals \mathcal{M}(\mu_{\mathscr{A}})$.
A vertex of $\mathscr{A}$ is said to be \emph{\bfseries switching} if so is it in $\mu_{\mathscr{A}}$.
A combinatorial arena $\mathscr{S}$ is called a \emph{\bfseries combinatorial subarena} of $\mathscr{A}$ if it satisfies $|\mathscr{S}| \sqsubseteq |\mathscr{A}|$, $\mu_{\mathscr{S}} \sqsubseteq \mu_{\mathscr{A}}$, $\alpha_{\mathscr{S}} \subseteq \alpha_{\mathscr{A}}$ and $\epsilon_{\mathscr{S}} \sqsubseteq \epsilon_{\mathscr{A}}$.
\end{definition}

The three axioms of Definition~\ref{DefCombinatorialArenas} are again for Theorem~\ref{ThmFreeCharacterisation} to hold. 
See Example~\ref{ExCombinatorialArenas}, Definition~\ref{DefConstructionsOnCombinatorialArenas} and Theorem~\ref{ThmFreeCharacterisation} on how multiplicative (respectively, additive, exponential) structures correspond to multiplicatives (respectively, additives, exponentials) in linear logic. 

\begin{example}
\label{ExCombinatorialArenas}
Consider again the multiplicative pair $(F, \mu)$ given in Example~\ref{ExMultSt} together with the additive structure $\alpha$ in Example~\ref{ExAddSt} and the exponential structure $\epsilon$ in Example~\ref{ExExpSt}.
They constitute a combinatorial arena, which is written $\mathscr{A}$ for a record purpose.

For counterexamples, the additive structure $\{ \{ x, z \}, \{ z, t \}, \{ s, t \} \}$ on $(F, \mu)$ does not satisfy the first two axioms of Definition~\ref{DefCombinatorialArenas}, and the exponential structure on $(F, \mu)$ that consists of the unique root $\{ s, t \}$ does not satisfy the third axiom with respect to the additive structure $\alpha$.
\end{example}


Let us next introduce constructions on combinatorial arenas, which correspond to those on formulae in intuitionistic linear logic \cite{girard1987lazy} (recalled in Definition~\ref{DefILL}):

\begin{definition}[constructions on combinatorial arenas]
\label{DefConstructionsOnCombinatorialArenas}
Let $\mathscr{A}$ and $\mathscr{B}$ be combinatorial arenas.
\begin{itemize}

\item \emph{\bfseries Top} is the combinatorial arena $\top \colonequals (\emptyset, \emptyset, \emptyset, \emptyset)$;

\item \emph{\bfseries One} is the combinatorial arena $1 \colonequals (1, 1, \{ 1 \}, \emptyset)$, where $1$ is an arbitrarily fixed element;

\item The \emph{\bfseries tensorial negation}\footnote{This terminology comes from the naming of a similar construction on games in \cite{mellies2010resource}.} of $\mathscr{A}$ is the combinatorial arena $\neg \mathscr{A}$ defined by
\begin{mathpar}
|\neg \mathscr{A}| \colonequals \neg . |\mathscr{A}|
\and
\mu_{\neg \mathscr{A}} \colonequals \neg \uplus \mu_{\mathscr{A}}
\and
\alpha_{\neg \mathscr{A}} \colonequals \alpha_{\mathscr{A}}
\and
\epsilon_{\neg \mathscr{A}} \colonequals \epsilon_{\mathscr{A}},
\end{mathpar}
where $\neg$ is an arbitrarily fixed element;

\item The \emph{\bfseries tensor} of $\mathscr{A}$ and $\mathscr{B}$ is the combinatorial arena $\mathscr{A} \otimes \mathscr{B}$ defined by
\begin{mathpar}
| \mathscr{A} \otimes \mathscr{B} | \colonequals |\mathscr{A}| \uplus |\mathscr{B}|
\and
\mu_{\mathscr{A} \otimes \mathscr{B}} \colonequals \mu_{\mathscr{A}} \uplus \mu_{\mathscr{B}}
\and
\alpha_{\mathscr{A} \otimes \mathscr{B}} \colonequals \alpha_{\mathscr{A}} \uplus \alpha_{\mathscr{B}}
\and
\epsilon_{\mathscr{A} \otimes \mathscr{B}} \colonequals \epsilon_{\mathscr{A}} \uplus \epsilon_{\mathscr{B}};
\end{mathpar}

\item The \emph{\bfseries product} or \emph{\bfseries with} of $\mathscr{A}$ and $\mathscr{B}$ is the combinatorial arena $\mathscr{A} \mathbin{\&} \mathscr{B}$ defined by
\begin{mathpar}
| \mathscr{A} \mathbin{\&} \mathscr{B} | \colonequals | \mathscr{A} \otimes \mathscr{B} |
\and
\mu_{\mathscr{A} \mathbin{\&} \mathscr{B}} \colonequals \mu_{\mathscr{A} \otimes \mathscr{B}}
\and
\alpha_{\mathscr{A} \mathbin{\&} \mathscr{B}} \colonequals \alpha_{\mathscr{A} \otimes \mathscr{B}} \uplus \{ \, \{ a, b \} \mid a \in \mathscr{R}_{|\mathscr{A}|} \cap \mathscr{R}_{\mu_{\mathscr{A}}}, b \in \mathscr{R}_{|\mathscr{B}|} \cap \mathscr{R}_{\mu_{\mathscr{B}}} \, \}
\and
\epsilon_{\mathscr{A} \mathbin{\&} \mathscr{B}} \colonequals \epsilon_{\mathscr{A} \otimes \mathscr{B}};
\end{mathpar}

\item The \emph{\bfseries of-course} of $\mathscr{A}$ is the combinatorial arena $\oc \mathscr{A}$ defined by 
\begin{mathpar}
|\oc \mathscr{A}| \colonequals |\mathscr{A}|
\and
\mu_{\oc \mathscr{A}} \colonequals \mu_{\mathscr{A}}
\and
\alpha_{\oc \mathscr{A}} \colonequals \alpha_{\mathscr{A}}
\and
\epsilon_{\oc \mathscr{A}} \colonequals \mathscr{R}_{\mathscr{A}} \ast \epsilon_{\mathscr{A}} \colonequals \big((\mathscr{R}_{|\mathscr{A}|} \cap \mathscr{R}_{\mu_{\mathscr{A}}}) . (\epsilon_{\mathscr{A}})_{\wp(\mathscr{R}_{|\mathscr{A}|} \cap \mathscr{R}_{\mu_{\mathscr{A}}})}\big) \uplus \big( \epsilon_{\mathscr{A}} {\upharpoonright_{\wp(\mathscr{V}_{|\mathscr{A}|} \setminus (\mathscr{R}_{|\mathscr{A}|} \cap \mathscr{R}_{\mu_{\mathscr{A}}}))}}\big);
\end{mathpar}

\item The \emph{\bfseries linear implication} from $\mathscr{A}$ to $\mathscr{B}$ is the combinatorial arena $\mathscr{A} \multimap \mathscr{B}$ defined by
\begin{mathpar}
| \mathscr{A} \multimap \mathscr{B} | \colonequals {\multimap} . |\mathscr{A}| \uplus |\mathscr{B}|
\and
\mu_{\mathscr{A} \multimap \mathscr{B}} \colonequals \mu_{\mathscr{A}} \uplus \mu_{\mathscr{B}}{\upharpoonright_{\mathscr{V}_{|\mathscr{B}|}^{\geqslant 1}}} \uplus {\multimap} . (\mu_{\mathscr{B}}{\upharpoonright_{\mathscr{R}_{|\mathscr{B}|}}})
\and
\alpha_{\mathscr{A} \multimap \mathscr{B}} \colonequals \alpha_{\mathscr{A}} \uplus \alpha_{\mathscr{B}}
\and
\epsilon_{\mathscr{A} \multimap \mathscr{B}} \colonequals \epsilon_{\mathscr{A}} \uplus \epsilon_{\mathscr{B}},
\end{mathpar}
where $\multimap$ is an arbitrarily fixed element, and the \emph{\bfseries implication} from $\mathscr{A}$ to $\mathscr{B}$ is 
\begin{equation*}
\mathscr{A} \Rightarrow \mathscr{B} \colonequals \oc \mathscr{A} \multimap \mathscr{B}.
\end{equation*}

\end{itemize} 
\end{definition}

\begin{convention}
Linear implication is right associative, while other binary operations are left associative. 
Each unary operation precedes all binary ones, and tensorial negation all other operations; linear implication is preceded by all other operations. 
\end{convention}

\begin{example}
\label{SimpleExamplesOfCombinatorialArenas}
The linear implication $\bot \multimap \bot \otimes \bot$ comprises of simple graphs only, while the corresponding arena \cite{hyland2000full} does not.
Note that the switching vertex $\multimap$ makes this difference.

It is instructive to observe that switching vertices also contribute to the inequalities
\begin{mathpar}
\oc \top = \top = \top \otimes \top = \top \mathbin{\&} \top \neq \top \multimap \top \neq (\top \multimap \top) \otimes (\top \multimap \top)
\and
\oc (\top \multimap \bot) \neq \top \multimap \oc \bot 
\and
\top \neq \mathscr{A} \multimap \top \neq \mathscr{B} \multimap \top
\and
\neg \mathscr{A} \neq \mathscr{A} \multimap \bot
\and
\top \multimap \mathscr{A} \neq \mathscr{A}
\and
\top \multimap \bot \mathbin{\&} \bot \neq \bot \mathbin{\&} (\top \multimap \bot) \neq (\top \multimap \bot) \mathbin{\&} (\top \multimap \bot),
\end{mathpar}
where $\mathscr{A}$ and $\mathscr{B}$ are arbitrary combinatorial arenas such that $\mathscr{A} \neq \mathscr{B}$.
These inequalities are quite remarkable because existing game semantics \cite{abramsky2000full,hyland2000full,mccusker1998games} does not attain them. 
We shall see later that some of the inequalities are crucial for our fully complete interpretation of intuitionistic linear logic (Theorem~\ref{ThmFullCompleteness}).
\end{example}

\begin{proposition}[well-defined constructions on combinatorial arenas]
\label{PropWellDefinedConstructionsOnCombinatorialSemiArenas}
Top and one form combinatorial arenas, and combinatorial arenas are closed under tensorial negation, tensor, with, of-course and linear implication. 
\end{proposition}
\begin{proof}
Straightforward and left to the reader. 
\end{proof}

\begin{example}
\label{ExSecondCombinatorialArenas}
The combinatorial arena $(F, \mu, \alpha, \epsilon)$ given in Example~\ref{ExCombinatorialArenas} can be constructed as $\oc 1_{[w]} \multimap_{[y]} (\oc (1_{[z]} \mathbin{\&} \oc \neg_{[x]} (\top \multimap_{[u]} \oc \oc \bot_{[v]})) \mathbin{\&} \bot_{[t]}) \otimes \bot_{[s]}$, where the subscripts are tags to indicate the vertices corresponding to the constructions. 
We use this notation throughout this article. 
\end{example}

The proof of the following proposition describes how to inductively construct a given combinatorial arena. 
In particular, the proof explains how the tags in Example~\ref{ExSecondCombinatorialArenas} are chosen. 

\begin{theorem}[a free characterisation of combinatorial arenas]
\label{ThmFreeCharacterisation}
Every combinatorial arena $\mathscr{A}$ can be obtained, up to graph isomorphisms $\cong$ on the underlying finite rooted forest $|\mathscr{A}|$, from top and/or one by tensorial negation, linear implication, tensor, with and/or of-course.
\end{theorem}
\begin{proof}
Let $\mathscr{A}$ be a combinatorial arena.
We proceed by induction on $\mathrm{Size}(\mathscr{A}) \colonequals |\mathscr{V}_{|\mathscr{A}|}| \uplus |\mathscr{V}_{\epsilon_{\mathscr{A}}}|$.

The base case corresponds to $\mathrm{Size}(\mathscr{A}) = 0$, in which $\mathscr{A} = \top$. 
Hence, we are done.
For the inductive steps, we henceforth assume $\mathrm{Size}(\mathscr{A}) > 0$, which implies $\mathscr{R}_{|\mathscr{A}|} \neq \emptyset$.

First, assume $\epsilon_{\mathscr{A}} = \mathscr{R}_{|\mathscr{A}|} \ast \epsilon'$ for some exponential structure $\epsilon'$. 
Then, $\mathscr{A}' \colonequals (|\mathscr{A}|, \mu_{\mathscr{A}}, \alpha_{\mathscr{A}}, \epsilon')$ is a combinatorial arena with $\mathrm{Size}(\mathscr{A}') < \mathrm{Size}(\mathscr{A})$ and $\oc \mathscr{A}' = \mathscr{A}$.
Thus, the induction hypothesis verifies the claim for $\mathscr{A}$.
Now, we can assume $\epsilon_{\mathscr{A}} \neq \mathscr{R}_{|\mathscr{A}|} \ast \epsilon''$ for all exponential structures $\epsilon''$.

Next, suppose that $\mathscr{R}_{|\mathscr{A}|}$ is singleton. 
If $\alpha_{\mathscr{A}} \cap \mathscr{R}_{|\mathscr{A}|} \neq \emptyset$, then $\mathscr{A} \cong 1$; thus, assume otherwise. 
Then, $\alpha_{\mathscr{A}} \cap \mathscr{R}_{|\mathscr{A}|} = \emptyset$ and $\epsilon_{\mathscr{A}} \neq \mathscr{R}_{|\mathscr{A}|} . \epsilon''$ for all exponential structures $\epsilon''$, so $\mathscr{A} \cong \neg \mathscr{A}''$ for some combinatorial arena $\mathscr{A}''$.
Because $\mathrm{Size}(\mathscr{A}') < \mathrm{Size}(\mathscr{A})$, the induction hypothesis on $\mathscr{A}''$ verifies the claim for $\mathscr{A}$.
Hence, in the following, it suffices to suppose that $\mathscr{R}_{|\mathscr{A}|}$ is not singleton.

Lastly, consider the case where $\mathscr{R}_{|\mathscr{A}|}$ is not singleton. 
The rooted subforest $\mu' \colonequals \mu_{\mathscr{A}}{\upharpoonright_{\mathscr{R}_{|\mathscr{A}|}}}$ of $\mu_{\mathscr{A}}$ then has more than one vertex.
We then proceed by case analysis on $\mu'$. 
If $\mu'$ is a rooted tree, where we write $x_0$ for its unique root, then we define 
\begin{mathpar}
V_1 \colonequals \{ \, x \in \mathscr{V}_{|\mathscr{A}|} \mid x_0 \rightarrow_{|\mathscr{A}|} x \, \}
\and
V_2 \colonequals \mathscr{R}_{|\mathscr{A}|} \setminus \{ x_0 \}.
\end{mathpar}
Let $\mathscr{A}_1$ be the combinatorial subarena of $\mathscr{A}$ induced by $|\mathscr{A}_1| \colonequals |\mathscr{A}|_{V_1}$, and $\mathscr{A}_2$ by $|\mathscr{A}_2| \colonequals |\mathscr{A}|_{V_2}$.
Clearly, $\mathscr{A} \cong \mathscr{A}_1 \multimap \mathscr{A}_2$ and $\mathrm{Size}(\mathscr{A}_i) < \mathrm{Size}(\mathscr{A})$ ($i = 1, 2$).
Hence, the induction hypotheses on $\mathscr{A}_i$ prove the claim for $\mathscr{A}$.
It remains to assume that $\mu'$ is not a rooted tree.
In this case, there is a nonempty, nontrivial, exhaustive recursive partition $P$ of $\mathscr{R}_{\mu'}$ that satisfies the three axioms of Definition~\ref{DefCombinatorialArenas}.
Let $S_I = \{ S_i \}_{i \in I}$ be $\mathscr{V}_P^1$, where $n \colonequals |I| > 1$.
The graph $(S_I, [\alpha_{\mathscr{A}}[2]]_I)$ is well-defined and either null or complete by the first two axioms.
If it is complete, then $\mathscr{A} \cong \mathbin{\&}_{j = 1}^n \mathscr{A}'_j$ for some combinatorial arenas $\mathscr{A}'_j$ with $\mathrm{Size}(\mathscr{A}'_j) < \mathrm{Size}(\mathscr{A})$.
If it is null, then $\mathscr{A} \cong \otimes_{j = 1}^n \mathscr{A}'_j$.
In either case, the induction hypotheses on $\mathscr{A}'_j$ deduce the claim for $\mathscr{A}$ by the third axiom.
\end{proof}

This result is the first step towards our biequivalences between logic and combinatorics (\S\ref{Bijections}).
In addition, this free characterisation provides us with a concise representation for an arbitrary combinatorial arena in terms of the inductive constructions.
This representation, combined with the tags $(\_)_{[\_]}$ introduced in Example~\ref{ExSecondCombinatorialArenas}, is useful because it can be quite intricate and tedious to directly describe a combinatorial arena as seen in previous examples.

At the end of this section, we show that combinatorial arenas form \emph{Hopf algebras} \cite{hopf1964topologie}:
\begin{proposition}[combinatorial Hopf algebras]
\label{PropHopf}
Combinatorial arenas yield Hopf algebras. 
\end{proposition}
\begin{proof}[Proof (sketch)]
Schmitt \cite[\S 3]{schmitt1993hopf} proved that the vector space over any commutative ring with identity, whose basis consists of all finite rooted forests, gives rise to a Hopf algebra through what is called the \emph{restriction} operation.
The restriction operation calculates the rooted subforest $F_X$ of a given finite rooted forest $F$ induced by a given subset $X$ of the vertex set. 
The claim of the theorem follows because the restriction operation easily extends to combinatorial arenas. 

Schmitt presented another way of producing a Hopf algebra of finite rooted forests through the \emph{quotient} operation \cite[\S 4]{schmitt1993hopf}, and this method is applicable to combinatorial arenas as well.
Thus, we have obtained two different Hopf algebras induced by combinatorial arenas.
\end{proof}

To the best of our knowledge, this proposition is the first bridge between game semantics and Hopf algebras. 
The literature of combinatorial Hopf algebras has shown that the results and the methods of Hopf algebras are quite useful for the study of combinatorial structures.
We leave it as future work to apply Hopf algebras to logic and computation via combinatorial arenas. 

Finally, let us demonstrate that Proposition~\ref{PropHopf} does not hold for arenas \cite{hyland2000full}.
Recall that an axiom on an arena requires that all paths from a root to a vertex in the arena have the same length.
However, the restriction operation does not respect this axiom, where the problem is that arenas are not simple (directed) graphs. 
For instance, consider the arena $x \rightarrow y \rightarrow z \leftarrow y' \leftarrow x'$.
Its subgraph induced by the set $\{ x, y, y', z \}$ is the one $x \rightarrow y \rightarrow z \leftarrow y'$, which is no longer an arena because the path from $x$ to $z$ and the one from $y'$ to $z$ have different lengths.

\subsection{Positions}
\label{Positions}
Our next step is to adapt possible developments or \emph{positions} in a game to combinatorial arenas.
HO \cite{hyland2000full} defines positions to be a class of finite walks on an arena in a \emph{finitary} fashion.
These positions are, however, too coarse to interpret linear logic.
In contrast, other variants of games \cite{abramsky1994games,abramsky2000full} possess finer controls on positions at the cost of carrying an (often) \emph{infinite} set of positions as part of a game.
The idea of a combinatorial arena $\mathscr{A}$ is to overcome this dilemma by allowing $\mathscr{A}$ to \emph{vary along the growth of a position} in such a way that the positions compatible with (or \emph{filtered} by) such graph-theoretic \emph{dynamics} interpret linear logic in a finitary way.

The aim of this section is to make this idea precise.
Let us first define the dynamics of the \emph{binary part} $\alpha_{\mathscr{A}}[2]$ of the additive structure $\alpha_{\mathscr{A}}$. 
This dynamics provides $\mathscr{A}$ with a finer control on positions to model binary additives in linear logic without referring to infinite sets of positions.
\begin{definition}[additive dynamics]
\label{DefAdditiveAction}
The \emph{\bfseries (binary) additive dynamics} of a combinatorial arena $\mathscr{A}$ on a vertex $x \in \mathscr{V}_{|\mathscr{A}|}$ is the subset $\mathscr{A}_\alpha(x) \subseteq \mathscr{V}_{|\mathscr{A}|}$ defined by
\begin{equation*}
\mathscr{A}_\alpha(x) \colonequals \begin{cases} \{ \, x' \in \mathscr{V}_{|\mathscr{A}|} \mid \{ x, x' \} \in \alpha_{\mathscr{A}}[2] \, \} &\text{if $x$ is switching in $\mu_{\mathscr{A}}$;} \\ \{ x \} \cup \{ \, x' \in \mathscr{V}_{|\mathscr{A}|} \mid \{ x, x' \} \in \alpha_{\mathscr{A}}[2] \, \} &\text{otherwise,}  \end{cases}
\end{equation*}
and it is lifted to a move $\vec{m} \in \mathcal{M}_{\mathscr{A}}$ by $\mathscr{A}_\alpha(\vec{m}) \colonequals \bigcup_{x \in \underline{\vec{m}}} \mathscr{A}_\alpha(x)$.
\end{definition}

This structure given by $\alpha_{\mathscr{A}}[2]$ yields the following graph-theoretic \emph{dynamics}: When a move $\vec{m}$ is made during a play in $\mathscr{A}$, moves $\vec{n}$ with $\underline{\vec{n}} \cap \mathscr{A}_\alpha(\vec{m}) \neq \emptyset$ become \emph{unavailable} (Definition~\ref{DefJustifiedSequences}), where a switching vertex never makes itself unavailable because it must be always there, unless made unavailable by another vertex, to bridge the domain and the codomain of linear implication.
This dynamics matches the intuition on with: selecting one of two sides. 
The \emph{0-ary part} $\alpha_{\mathscr{A}}[0]$ has nothing to do with any dynamics of $\mathscr{A}$, and its role is defined later (Definition~\ref{DefPositionsInCombinatorialSemiArenas}). 

Let us next formulate the graph-theoretic dynamics of the exponential structure $\epsilon_{\mathscr{A}}$, which is to model exponentials in a finitary fashion.
To this end, we need some notations: 
\begin{convention}
Let $\mathscr{A}$ be a combinatorial arena, and $F$ a finite rooted forest.
\begin{itemize}


\item If $\epsilon_{\mathscr{A}}$ has a nonempty path whose elements all contain $x \in \mathscr{V}_{|\mathscr{A}|}$, then $\epsilon_{\mathscr{A}}(x) \colonequals (\epsilon_{\mathscr{A}}(x)_i)_{i \in \overline{n}}$, where $\epsilon_{\mathscr{A}}(x)_1 \supseteq \epsilon_{\mathscr{A}}(x)_2 \supseteq \dots \supseteq \epsilon_{\mathscr{A}}(x)_n$ is the longest one, and otherwise $\epsilon_{\mathscr{A}}(x) \colonequals \boldsymbol{\epsilon}$;

\item Given an element $x$ and sequences $\boldsymbol{y} = y_1y_2 \dots y_n$ and $\boldsymbol{j} = j_1j_2 \dots j_n$, let $x \{ \boldsymbol{y} ; \boldsymbol{j} \}$ denote the nested pair $(\dots ((x, y_1^{j_1}), y_2^{j_2}), \dots, y_n^{j_n})$, where $y_i^{j_i} \colonequals (y_i, j_i)$; it is lifted to a set $X$ by $X \{ \boldsymbol{y} ; \boldsymbol{j} \} \colonequals \{ \, x \{ \boldsymbol{y} ; \boldsymbol{j} \} \mid x \in X \, \}$, and to a finite sequence $\boldsymbol{s}$ by $\boldsymbol{s} \{ \boldsymbol{y} ; \boldsymbol{j} \} \colonequals (\boldsymbol{s}(i) \{ \boldsymbol{y} ; \boldsymbol{j} \} )_{i \in \overline{|\boldsymbol{s}|}}$;

\item Because $x \{ \boldsymbol{\epsilon} ; \boldsymbol{\epsilon} \} = x$, we suppose that each vertex of $\mathscr{A}$ is of the form $x \{ \boldsymbol{y} ; \boldsymbol{j} \}$; without loss of generality, we also assume $x \{ \boldsymbol{y} ; \boldsymbol{j} \} \neq x'$ if $x, x' \in \mathscr{V}_{|\mathscr{A}|}$, $\boldsymbol{y} \in \mathscr{V}_{|\mathscr{A}|}^{\ast}$, $\boldsymbol{j} \in \mathbb{N}_+^\ast$ and $|\boldsymbol{y}| > 0$;


\item Given a subset $S \subseteq \mathscr{V}_F$, let $F[S \{ \boldsymbol{y}; \boldsymbol{j} \}]$ denote the finite rooted forest obtained from $F$ by replacing each vertex $x$ of $F$ that is in $S$ with the element $x \{ \boldsymbol{y}; \boldsymbol{j} \}$. 

\end{itemize}

\end{convention}

Then, the dynamics of $\epsilon_{\mathscr{A}}$ is that, when a move $\vec{m}$ is made during a play, the combinatorial subarenas of $\mathscr{A}$ specified by the elements of the sequence $\epsilon_{\mathscr{A}}(x)$ for all $x \in \underline{\vec{m}}$ are \emph{duplicated} and \emph{disjointly adjoined} to $\mathscr{A}$.
This dynamics matches the intuition on of-course: producing countably many copies of formulae. 
Note that $\epsilon_{\mathscr{A}}(x)$ corresponds to a finite sequence of nested occurrences of of-course in a formula, e.g., $\oc (A \otimes (\oc B \mathbin{\&} C))$.
For computation on those copies of combinatorial subarenas, we make tags on the disjoint unions precise; we use $\_\{ \boldsymbol{y} ; \boldsymbol{j} \}$ for this aim.

\begin{definition}[exponential dynamics]
\label{DefExpAction}
Given a combinatorial arena $\mathscr{A}$, assume $y' \in \mathscr{V}_{|\mathscr{A}|}$, $\tilde{y} = y' \{ \boldsymbol{y} ; \boldsymbol{j} \} \in \mathscr{V}_{|\mathscr{A}|}$ and $j' \in \overline{|\epsilon_{\mathscr{A}}(\tilde{y})|}$. 
Let $\mathscr{A}_\epsilon(\tilde{y}, j')$ be the combinatorial arena obtained from $\mathscr{A}$ by disjointly adding its combinatorial subarena induced by the set $\overline{\epsilon_{\mathscr{A}}(\tilde{y})(j')}$, or
\begin{itemize}

\item $|\mathscr{A}_\epsilon(\tilde{y}, j')| \colonequals |\mathscr{A}| \cup (|\mathscr{A}|^\bullet[\mathscr{V}_{|\mathscr{A}|^\flat} \{ y'; j' \}])$ ($|\mathscr{A}|^\bullet \colonequals |\mathscr{A}|^\bullet_{\overline{\epsilon_{\mathscr{A}}(\tilde{y})(j')}}$ and $|\mathscr{A}|^\flat \colonequals |\mathscr{A}|_{\overline{\epsilon_{\mathscr{A}}(\tilde{y})(j')}}$);

\item $\mu_{\mathscr{A}_\epsilon(\tilde{y}, j')} \colonequals \mu_{\mathscr{A}} \cup (\mu_{\mathscr{A}} {\upharpoonright_{\mathscr{V}_{|\mathscr{A}|^\flat}}}) [\mathscr{V}_{|\mathscr{A}|^\flat} \{ y'; j' \}]$;

\item $\alpha_{\mathscr{A}_\epsilon(\tilde{y}, j')} \colonequals \alpha_{\mathscr{A}} \cup (\alpha_{\mathscr{A}}[0] \cap \mathscr{V}_{|\mathscr{A}|^\flat})\{ y' ; j' \} \cup \{ \, e \{ y' ; j' \} \mid e \in \alpha_{\mathscr{A}}[2] \cap \wp(\mathscr{V}_{|\mathscr{A}|^\flat}) \, \}$;

\item $\epsilon_{\mathscr{A}_\epsilon(\tilde{y}, j')} \colonequals \epsilon_{\mathscr{A}} \cup (\epsilon_{\mathscr{A}})_{\epsilon_{\mathscr{A}}(\tilde{y})(j')} [(\mathscr{V}_{\epsilon_{\mathscr{A}}} \cap \wp(\mathscr{V}_{|\mathscr{A}|^\flat})) \{ y'; j' \}]$, where unlike $|\mathscr{A}|^\bullet$ or $|\mathscr{A}|^\flat$ defined above $\epsilon_{\mathscr{A}}(\tilde{y})(j')$ in $\epsilon_{\mathscr{A}_\epsilon(\tilde{y}, j')}$ serves not as a set but as a vertex of $\epsilon_{\mathscr{A}}$.

\end{itemize}
Further, let $\mathscr{A}_\epsilon(\tilde{y}) \colonequals \mathscr{A}_\epsilon(\tilde{y},1)_\epsilon (\tilde{y}, 2)_\epsilon \dots (\tilde{y}, |\epsilon_{\mathscr{A}}(\tilde{y})|)$ if $|\epsilon_{\mathscr{A}}(\tilde{y})| > 0$, and $\mathscr{A}_\epsilon(\tilde{y}) \colonequals \mathscr{A}$ otherwise. 

The \emph{\bfseries exponential dynamics} of $\mathscr{A}$ on a move $\vec{m} \in \mathscr{M}_{\mathscr{A}}$ is the combinatorial arena 
\begin{equation*}
\mathscr{A}_\epsilon(\vec{m}) \colonequals \begin{cases} \mathscr{A}_\epsilon(\vec{m}(1))_\epsilon(\vec{m}(2))_\epsilon \dots (\vec{m}(|\vec{m}|)) &\text{if $|\vec{m}| > 0$;} \\ \mathscr{A} &\text{otherwise.} \end{cases}
\end{equation*}
\end{definition}

\begin{example}
\label{ExActions} 
Recall the combinatorial arena $\mathscr{A}$ defined in Example~\ref{ExCombinatorialArenas}. 
For understanding the dynamics of $\mathscr{A}$ described below in terms of its syntactic counterpart, we strongly recommend the reader to refer to the inductive construction of $\mathscr{A}$ given in \ref{ExSecondCombinatorialArenas}.
First, observe the dynamics of $\mathscr{A}$ on the move $yx \in \mu_{\mathscr{A}}$: The additive dynamics $\mathscr{A}_\alpha(yx)$ is the set $\{ x, z, t \}$, and the exponential dynamics $\mathscr{A}_\epsilon(yx)$ is the combinatorial arena that consists of the multiplicative pair
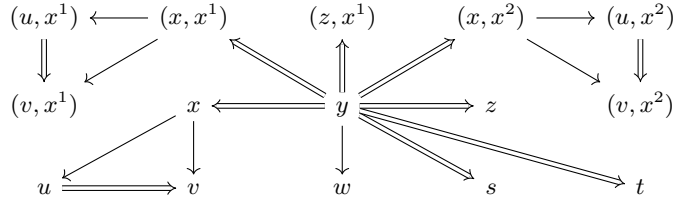
\begin{small}
\begin{mathpar}
\begin{tikzcd}
(u, x^1) \arrow[d, Rightarrow] & \arrow[l] \arrow[ld] (x, x^1) &(z, x^1) & (x, x^2) \arrow[r] \arrow[rd] & (u, x^2) \arrow[d, Rightarrow] \\
(v, x^1) & \arrow[ld] \arrow[d] x & \arrow[l, Rightarrow] \arrow[d] \arrow[lu, Rightarrow] y \arrow[u, Rightarrow] \arrow[rd, Rightarrow] \arrow[rrd, Rightarrow] \arrow[ru, Rightarrow] \arrow[r, Rightarrow] & z & (v, x^2) \\
u \arrow[r, Rightarrow] & v & w & s & t
\end{tikzcd}
\end{mathpar} 
\end{small}together with the additive structure 
\begin{equation*}
\alpha_{\mathscr{A}} \cup \{ \{ (x, x^1), (z, x^1) \} \}
\end{equation*}
and the exponential structure
\begin{small}
\begin{mathpar}
\begin{tikzcd}
\arrow[d] \{ x, z \} \\
\{ x \}
\end{tikzcd}
\and
\begin{tikzcd}
\arrow[d] \{ (x, x^1), (z, x^1) \} \\
\{ (x, x^1) \}
\end{tikzcd}
\and
\begin{tikzcd}
\{ (x, x^2) \}
\end{tikzcd}
\and
\begin{tikzcd}
\arrow[d] \{ u \} \\
\{ v \}
\end{tikzcd}
\and
\begin{tikzcd}
\arrow[d] \{ (u, x^1) \} \\
\{ (v, x^1) \}
\end{tikzcd}
\and
\begin{tikzcd}
\arrow[d] \{ (u, x^2) \} \\
\{ (v, x^2) \}
\end{tikzcd}
\and
\begin{tikzcd}
\{ w \}
\end{tikzcd}
\end{mathpar}
\end{small}

Next, the additive dynamics $\mathscr{A}_\epsilon(yx)_\alpha(w)$ of $\mathscr{A}_\epsilon(yx)$ on the move $w$ is the singleton set $\{ w \}$, and the exponential dynamics $\mathscr{A}_\epsilon(yx)_\epsilon(w)$ does nothing, i.e., $\mathscr{A}_\epsilon(yx)_\epsilon(w) = \mathscr{A}_\epsilon(yx)$.

Finally, consider the dynamicss of $\mathscr{A}_\epsilon(yx)_\epsilon(w)$ on the move $y(z, x^1)$: The additive dynamics $\mathscr{A}_\epsilon(yx)_\epsilon(w)_\alpha(y(z, x^1))$ is the set $\{ (x, x^1), (z, x^1) \}$, and the exponential dynamics $\mathscr{A}_\epsilon(yx)_\epsilon(w)_\epsilon(y(z, x^1))$ is the combinatorial arena that consists of the multiplicative pair
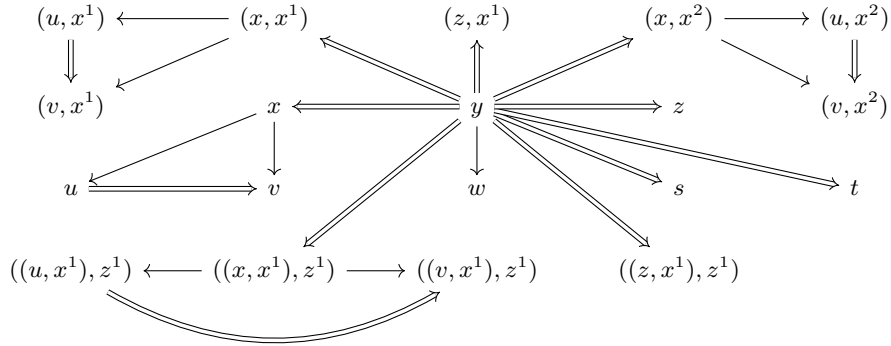
\begin{small}
\begin{mathpar}
\begin{tikzcd}
(u, x^1) \arrow[d, Rightarrow] & \arrow[l] \arrow[ld] (x, x^1) &(z, x^1) & (x, x^2) \arrow[r] \arrow[rd] & (u, x^2) \arrow[d, Rightarrow] \\
(v, x^1) & \arrow[ld] \arrow[d] x & \arrow[l, Rightarrow] \arrow[d] \arrow[lu, Rightarrow] \arrow[ddl, Rightarrow] y \arrow[ddr, Rightarrow] \arrow[u, Rightarrow] \arrow[rd, Rightarrow] \arrow[rrd, Rightarrow] \arrow[ru, Rightarrow] \arrow[r, Rightarrow] & z & (v, x^2) \\
u \arrow[r, Rightarrow] & v & w & s & t \\
((u, x^1), z^1) \arrow[rr, Rightarrow, bend right] & \arrow[l] ((x, x^1), z^1) \arrow[r] & ((v, x^1), z^1) & ((z, x^1), z^1)  &
\end{tikzcd}
\end{mathpar} 
\end{small}together with the additive structure 
\begin{equation*}
\alpha_{\mathscr{A}_\epsilon(yx)_\epsilon(w)} \cup \{ \{ ((x, x^1), z^1), ((z, x^1), z^1) \} \}
\end{equation*}
and the exponential structure
\begin{small}
\begin{mathpar}
\begin{tikzcd}
\arrow[d] \{ x, z \} \\
\{ x \}
\end{tikzcd}
\and
\begin{tikzcd}
\arrow[d] \{ (x, x^1), (z, x^1) \} \\
\{ (x, x^1) \}
\end{tikzcd}
\and
\begin{tikzcd}
\arrow[d] \{ ((x, x^1), z^1), ((z, x^1), z^1) \} \\
\{ ((x, x^1), z^1) \}
\end{tikzcd}
\and
\begin{tikzcd}
\{ (x, x^2) \}
\end{tikzcd}
\\
\begin{tikzcd}
\arrow[d] \{ u \} \\
\{ v \}
\end{tikzcd}
\and
\begin{tikzcd}
\arrow[d] \{ (u, x^1) \} \\
\{ (v, x^1) \}
\end{tikzcd}
\and
\begin{tikzcd}
\arrow[d] \{ (u, x^2) \} \\
\{ (v, x^2) \}
\end{tikzcd}
\and
\begin{tikzcd}
\arrow[d] \{ ((u, x^1), z^1) \} \\
\{ ((v, x^1), z^1) \}
\end{tikzcd}
\and
\begin{tikzcd}
\{ w \}
\end{tikzcd}
\end{mathpar}
\end{small}

Because these structures easily become larger and larger during a play, one may wonder if the present framework is hard to manage.
Fortunately, it is not the case: One can always represent the structures concisely by suitable \emph{syntactic sugars}; e.g., see Example~\ref{ExArithmetic}.
\end{example}

This example also illustrates the following:
\begin{definition}[extended vertices and moves]
An \emph{\bfseries extended vertex} of a combinatorial arena $\mathscr{A}$ is a nested pair $x' \{ \boldsymbol{x} ; \boldsymbol{i} \}$ with $x' \in \mathscr{V}_{|\mathscr{A}|}$, $\boldsymbol{x} \in \mathscr{V}_{|\mathscr{A}|}^\ast$ and $\boldsymbol{i} \in \mathbb{N}_+^\ast$, where $x'$ is called its \emph{\bfseries core}, and an \emph{\bfseries extended O-} (respectively, \emph{\bfseries P-}) \emph{\bfseries move} in $\mathscr{A}$ is a sequence $x'_1 \{ \boldsymbol{x}_1 ; \boldsymbol{i}_1 \} x'_2 \{ \boldsymbol{x}_2 ; \boldsymbol{i}_2 \} \dots x'_n \{ \boldsymbol{x}_n ; \boldsymbol{i}_n \}$ of extended vertices of $\mathscr{A}$ such that $x'_1 x'_2 \dots x'_n \in \mathcal{M}_{\mathscr{A}}$ is an O- (respectively, P-) move in $\mathscr{A}$.
\end{definition}

\begin{convention}
An \emph{\bfseries extended move} in a combinatorial arena $\mathscr{A}$ refers to an extended O- or P-move in $\mathscr{A}$.
We write $\mathscr{V}_{|\mathscr{A}|}^+$ for the set of all extended vertices of $\mathscr{A}$, and $\mathcal{M}_{\mathscr{A}}^+$ for the set of all extended moves in $\mathscr{A}$.
A tag of the form $\_\{ \boldsymbol{y} ; \boldsymbol{j} \}$ is called an \emph{\bfseries exponential tag}.
If a vertex $x' \in \mathscr{V}_{|\mathscr{A}|}$ is switching, then any extended one of the form $x' \{ \boldsymbol{x} ; \boldsymbol{i} \}$ is said to be \emph{\bfseries switching} too.
\end{convention}

Next, recall that a position in an arena is a finite sequence of vertices equipped with a \emph{pointer} \cite{hyland2000full}.
Such a sequence together with a pointer is called a \emph{justified sequence}.
The idea is that each non-initial element of a justified sequence is made for a specific previous one, and a pointer represents this relation.
This idea traces back to Coquand \cite{coquand1995semantics}. 
An advantage of pointers is that they enable an arena to define its positions in a \emph{finitary} way unlike other variants of games.

Let us then adapt justified sequences to combinatorial arenas:

\begin{notation}
Given an extended vertex $x = x_0 \{ \boldsymbol{x} ; \boldsymbol{i} \}$, we write $x^\sharp$ for its core, i.e., $x^\sharp \colonequals x_0$, and this operation $(\_)^\sharp$ extends to  extended moves in the evident way.
\end{notation}

\begin{definition}[justified sequences]
\label{DefJustifiedSequences}
A \emph{\bfseries justified sequence} in a combinatorial arena $\mathscr{A}$ is a pair $\boldsymbol{s} = (\boldsymbol{s}, J_{\boldsymbol{s}})$ of a finite sequence $\boldsymbol{s}$ of extended moves in $\mathscr{A}$ and a map $J_{\boldsymbol{s}} : \overline{|\boldsymbol{s}|} \rightarrow \{ 0 \} \cup \overline{|\boldsymbol{s}|-1}$ such that $0 \leqslant J_{\boldsymbol{s}}(i) < i$ for all $i \in \overline{|\boldsymbol{s}|}$, called the \emph{\bfseries pointer}, that satisfies
\begin{enumerate}

\item If $\boldsymbol{t} \vec{m} \preceq \boldsymbol{s}$, then $\vec{m} \in \mathcal{M}_{\mathscr{A}_\epsilon(\boldsymbol{t})}$, where $\mathscr{A}_\epsilon(\boldsymbol{t}) \colonequals \mathscr{A}_\epsilon(\boldsymbol{t}(1))_\epsilon(\boldsymbol{t}(2))_\epsilon \dots (\boldsymbol{t}(|\boldsymbol{t}|))$ if $|\boldsymbol{t}| > 0$, and $\mathscr{A}_\epsilon(\boldsymbol{t}) \colonequals \mathscr{A}$ otherwise, and each $x \in \underline{\vec{m}}$ has the shortest exponential tag among all $y \in \mathscr{V}_{|\mathscr{A}_\epsilon(\boldsymbol{t})|}$ such that $y^\sharp = x^\sharp$;

\item If $i \in |\boldsymbol{s}|$, then $J_{\boldsymbol{s}}(i) = 0$ implies that $\boldsymbol{s}(i)$ consists of roots of $|\mathscr{A}_\epsilon(\boldsymbol{s}_{\leqslant i})|$, and $J_{\boldsymbol{s}}(i) > 0$ that there is $x \in \underline{\boldsymbol{s}(J_{\boldsymbol{s}}(i))}$ such that $x \rightarrow_{|\mathscr{A}_\epsilon(\boldsymbol{s}_{\leqslant i})|} y$ for all $y \in \underline{\boldsymbol{s}(i)}$;

\item If $\boldsymbol{t} \vec{m} \preceq \boldsymbol{s}$, then no element of $\vec{m}$ is in the set $\mathscr{A}_\alpha(\boldsymbol{t}) \colonequals \bigcup_{i = 1}^{|\boldsymbol{t}|} \mathscr{A}_\epsilon(\boldsymbol{t})_\alpha(\boldsymbol{t}(i))$ whose elements are said to be \emph{\bfseries unavailable} at $\boldsymbol{t}$.

\end{enumerate}
\end{definition}

\begin{notation}
We write $\mathcal{J}_{\mathscr{A}}$ for the set of all justified sequences in a combinatorial arena $\mathscr{A}$.
\end{notation}

The second axiom is a plain modification of the axiom on justified sequences in an arena that there is an edge between a non-initial element and its justifier \cite{hyland2000full}.
The other two axioms take into account the \emph{dynamics} of combinatorial arenas.
The first axiom allows not only moves but also extended ones generated by exponential dynamics to be elements of a justified sequence $\boldsymbol{s}$ in a combinatorial arena, but only those with the \emph{shortest} exponential tags.
This restriction is for \emph{finite presentations} of strategies (Definition~\ref{DefViewAlgorithms}); see the remark right after Definition~\ref{DefFinitePresentations}.
The third axiom requires that each element of $\boldsymbol{s}$ consists of \emph{available} vertices only. 

\begin{example}
\label{ExJustifiedSequences}
Recall the combinatorial arena $\mathscr{A}$ given in Examples~\ref{ExCombinatorialArenas} and \ref{ExActions}.
The pairs
\begin{mathpar}
(yt . w . ys, \{1 \mapsto 0, 2 \mapsto 1, 3 \mapsto 0 \})
\and
(yt . ys . w, \{1 \mapsto 0, 2 \mapsto 0, 3 \mapsto 1 \})
\end{mathpar}
are both justified sequences in $\mathscr{A}$, but neither of the pairs
\begin{mathpar}
(yt . w . yz, \{1 \mapsto 0, 2 \mapsto 1, 3 \mapsto 0 \})
\and
(yt . yz . w, \{1 \mapsto 0, 2 \mapsto 0, 3 \mapsto 1 \}).
\end{mathpar}
This illustrates additive dynamics, which distinguishes with from tensor (cf.~Example~\ref{ExSecondCombinatorialArenas}).

Next, as an illustration of exponential dynamics, which models the duplication of formulae given by of-course, observe that the pairs 
\begin{mathpar}
(yx . w . y(z, x^1) . y(x, x^2), \{ 1 \mapsto 0, 2 \mapsto 1, 3 \mapsto 0, 4 \mapsto 0\})
\and
(yx . uv . y(x, x^1) . (u, x^1)(v, x^1), \{ 1 \mapsto 0, 2 \mapsto 1, 3 \mapsto 0, 4 \mapsto 3 \})
\end{mathpar}
are both justified sequences in $\mathscr{A}$, but neither of the pairs
\begin{mathpar}
(yx . w . yz . yx, \{ 1 \mapsto 0, 2 \mapsto 1, 3 \mapsto 0, 4 \mapsto 0 \})
\and
(yx . uv . yx . uv, \{ 1 \mapsto 0, 2 \mapsto 1, 3 \mapsto 0, 4 \mapsto 3 \}).
\end{mathpar}
\end{example}

\begin{convention}
Let $\boldsymbol{s} = (\boldsymbol{s}, J_{\boldsymbol{s}})$ be a justified sequence in a combinatorial arena $\mathscr{A}$.
\begin{itemize}

\item An \emph{\bfseries occurrence} of an element $\vec{m}$ in $\boldsymbol{s}$ is a pair $(\boldsymbol{s}(i), i)$ such that $\boldsymbol{s}(i) = \vec{m}$ and $i \in \overline{|\boldsymbol{s}|}$;

\item The occurrence $(\boldsymbol{s}({J_{\boldsymbol{s}}(i)}), J_{\boldsymbol{s}}(i))$ is called the \emph{\bfseries justifier} of the one $(\boldsymbol{s}(i), i)$ in $\boldsymbol{s}$, and equivalently $(\boldsymbol{s}(i), i)$ is said to be \emph{\bfseries justified} by $(\boldsymbol{s}({J_{\boldsymbol{s}}(i)}), J_{\boldsymbol{s}}(i))$ in $\boldsymbol{s}$;

\item An occurrence $(\boldsymbol{s}(i), i)$ in $\boldsymbol{s}$ is said to be \emph{\bfseries initial} if $J_{\boldsymbol{s}}(i) = 0$, and \emph{\bfseries pseudo-initial} if $J_{\boldsymbol{s}}(i) > 0$ with all elements of $\boldsymbol{s}(i)$ the children of a switching element of $\boldsymbol{s}({J_{\boldsymbol{s}}(i)})$.

\end{itemize}

Henceforth, we are casual about the distinction between extended moves and their occurrences in a sequence. 
Besides, we usually keep the pointer $J_{\boldsymbol{s}}$ of a justified sequence $\boldsymbol{s} = (\boldsymbol{s}, J_{\boldsymbol{s}})$ implicit and abbreviate occurrences $(\boldsymbol{s}(i), i)$ as $\boldsymbol{s}(i)$.
We even write $J_{\boldsymbol{s}}(\boldsymbol{s}(i)) = \boldsymbol{s}(j)$ if $J_{\boldsymbol{s}}(i) = j$.
\end{convention}

As a typical example of the use of this convention, we call a justified sequence $\boldsymbol{t}$ a \emph{\bfseries justified subsequence} of a justified sequence $\boldsymbol{s}$ if $\boldsymbol{t}$ is a subsequence of $\boldsymbol{s}$, and $J_{\boldsymbol{t}}(\vec{n}) = \vec{m}$ if and only if there are elements $\vec{m}_1, \vec{m}_2, \dots, \vec{m}_k$ ($k \in \mathbb{N}$) of $\boldsymbol{s}$ yet deleted in $\boldsymbol{t}$ such that $J_{\boldsymbol{s}}(\vec{n}) = \vec{m}_1$, $J_{\boldsymbol{s}}(\vec{m}_1) = \vec{m}_2$, \dots, $J_{\boldsymbol{s}}(\vec{m}_{k-1}) = \vec{m}_k$ and $J_{\boldsymbol{s}}(\vec{m}_{k}) = \vec{m}$.

As the last preparation for positions, let us adapt a central concept in \cite{hyland2000full}, called \emph{views}, to combinatorial arenas.
In this adaptation, we modify views in such a way that they \emph{forget all but the last exponential tags} on extended moves in order for \emph{innocent} strategies (Definition~\ref{DefConstraintsOnStrategies}) to be \emph{finitely presentable} in a suitable sense (\S\ref{FinitelyPresentableStrategies}).
We also let views handle the new concept of pseudo-initial occurrences in the same way as initial ones so that Theorem~\ref{ThmFullCompleteness} will hold.

\begin{definition}[views \cite{coquand1995semantics,hyland2000full}] 
\label{DefViews}
The \emph{\bfseries P-view} $\lceil \boldsymbol{s} \rceil$ and the \emph{\bfseries O-view} $\lfloor \boldsymbol{s} \rfloor$ of a justified sequence $\boldsymbol{s}$ in a combinatorial arena $\mathscr{A}$ are respectively the justified sequences in $\mathscr{A}$ (consisting of moves in $\mathscr{A}$, not extended ones) defined inductively by
\begin{enumerate}

\item $\lceil \boldsymbol{\epsilon} \rceil \colonequals \boldsymbol{\epsilon}$; 

\item $\lceil \boldsymbol{s} \vec{m} \rceil \colonequals \lceil \boldsymbol{s} \rceil . \vec{m}^\sharp$ if $\vec{m}$ is a P-move; 

\item $\lceil \boldsymbol{s} \vec{m} \rceil \colonequals \vec{m}^\sharp$ if $\vec{m}$ is initial or pseudo-initial;

\item $\lceil \boldsymbol{s} \vec{m} \boldsymbol{t} \vec{n} \rceil \colonequals \lceil \boldsymbol{s} \rceil . \vec{m}^\sharp . \vec{n}^\sharp$ if $\vec{n}$ is an extended O-move such that $\vec{m}$ justifies $\vec{n}$; 

\item $\lfloor \boldsymbol{\epsilon} \rfloor \colonequals \boldsymbol{\epsilon}$;

\item $\lfloor \boldsymbol{s} \vec{m} \rfloor \colonequals \lfloor \boldsymbol{s} \rfloor . \vec{m}^\sharp$ if $\vec{m}$ is an O-move; 

\item $\lfloor \boldsymbol{s} \vec{m} \boldsymbol{t} \vec{n} \rfloor \colonequals \lfloor \boldsymbol{s} \rfloor . \vec{m}^\sharp . \vec{n}^\sharp$ if $\vec{n}$ is an extended P-move such that $\vec{m}$ justifies $\vec{n}$. 

\end{enumerate}
\end{definition}

\begin{remark}
The P-view $\lceil \boldsymbol{s} \rceil$ or the O-view $\lfloor \boldsymbol{s} \rfloor$ may not be a justified sequence because the justifier of $\vec{m}$ may be lost in the clause (2) or (6).
However, this problem is insignificant because we shall focus on a class of justified sequences, called \emph{positions}, in which this problem does not occur. 
\end{remark}

The idea is that, for a nonempty position $\boldsymbol{s} \vec{m}$ in a combinatorial arena $\mathscr{A}$ such that $\vec{m}$ is of odd (respectively, even) depth, the O-view $\lfloor \boldsymbol{s} \rfloor$ (respectively, the P-view $\lceil \boldsymbol{s} \rceil$) is the currently `relevant part' of $\boldsymbol{s}$ for Opponent (respectively, Player) from the logical standpoint. 
I.e., for logic Opponent (respectively, Player) is concerned only with the last occurrence in $\boldsymbol{s}$, its justifier and that justifier's O- (respectively, P-) view, which recursively proceeds.
P-views are defined as such since intuitively Opponent's aim is to \emph{attack} a specific point in Player's argument, and Player's goal is to \emph{defend} against the attack; i.e., every extended O-move in a position attacks specifically the prefix of the position ending with the justifier of the extended O-move, so the P-view of the position focuses on the last one and the P-view of the prefix (recursively).
O-views are \emph{P-views in the domain of linear implication} since Opponent and Player are switched in the domain. 

We are now ready to define \emph{positions} in a combinatorial arena: 

\begin{notation}
Let $\boldsymbol{s}$ be a justified sequence in a combinatorial arena $\mathscr{A}$.
\begin{itemize}

\item If $\mathscr{A} = \neg \mathscr{A}'$ and $\boldsymbol{s} = \neg \boldsymbol{t}$, then $\boldsymbol{s}{\upharpoonright_{\mathscr{A}'}} \colonequals \boldsymbol{t} \in \mathcal{J}_{\mathscr{A}'}$ (up to tags);

\item If $\mathscr{A} = \mathscr{B} \times \mathscr{C}$, where $\times$ is $\otimes$ or $\&$, then let $\boldsymbol{s}{\upharpoonright_{\mathscr{B}}} \in \mathcal{J}_{\mathscr{B}}$ and $\boldsymbol{s}{\upharpoonright_{\mathscr{C}}} \in \mathcal{J}_{\mathscr{C}}$ be the justified subsequences of $\boldsymbol{s}$ (up to tags) that consist of extended moves in $\mathscr{B}$ and $\mathscr{C}$, respectively. 

\item If $\mathscr{A} = \oc \mathscr{A}''$, then each extended vertex of $\mathscr{A}$ is written uniquely in the form
\begin{equation}
\label{ExponentialForm}
x_0 \{ x_1x_2 \dots x_n ; 1^n \} \{ y_1y_2 \dots y_m; j_1j_2 \dots j_m \} \quad (n, m \in \mathbb{N})
\end{equation} 
such that $x_i \in \mathscr{R}_{|\mathscr{A}''|} \cap \mathscr{R}_{\mu_{\mathscr{A}''}}$ for all $i = 0, 1, \dots, n$ and $j_k > 1$ for $k = 1, 2, \dots, m$.
Let $\boldsymbol{s}{\upharpoonright_{1, l}}$ be the justified subsequence of $\boldsymbol{s}$ consisting of extended moves in $\mathscr{A}$ whose elements are of the form (\ref{ExponentialForm}) with $l = n$ except they are changed into 
\begin{equation}
\label{Transformation}
x_0 \{ y_1y_2 \dots y_m ; (j_1-1)(j_2-1) \dots (j_m-1) \}.
\end{equation}
Note that $\boldsymbol{s}{\upharpoonright_{1, l}} \in \mathcal{J}_{\mathscr{A}''}$ holds thanks to the transformation of (\ref{ExponentialForm}) into (\ref{Transformation}).

\item If $\mathscr{A} = \mathscr{B} \multimap \mathscr{C}$, then $\boldsymbol{s}{\upharpoonright_{\mathscr{B}}} \in \mathcal{J}_{\mathscr{B}}$ is defined as above, and $\boldsymbol{s}{\upharpoonright_{\mathscr{C}}} \in \mathcal{J}_{\mathscr{C}}$ is obtained from the justified subsequence of $\boldsymbol{s}$ (up to tags) that consists of extended moves in $\mathscr{C}$ by deleting the switching (extended) vertex attributed to the linear implication.

\end{itemize}

\end{notation}


\begin{definition}[positions in a combinatorial arena]
\label{DefPositionsInCombinatorialSemiArenas}
A justified sequence $\boldsymbol{s}$ in a combinatorial arena $\mathscr{A}$ is called a \emph{\bfseries position} in $\mathscr{A}$ if it satisfies
\begin{enumerate}

\item \textsc{(Recursive alternation)} If $\boldsymbol{s} = \boldsymbol{t} \vec{m} \vec{n} \boldsymbol{u}$ with $\vec{m}$ an extended O- (respectively, P-) move, then $\vec{n}$ is an extended P- (respectively, O-) move, and this \emph{\bfseries alternation} between Opponent and Player holds recursively for $\boldsymbol{s}{\upharpoonright_{\mathscr{A}'}}$ in $\mathscr{A}'$ (respectively, for $\boldsymbol{s}{\upharpoonright_{\mathscr{B}}}$ in $\mathscr{B}$ and $\boldsymbol{s}{\upharpoonright_{\mathscr{C}}}$ in $\mathscr{C}$, for $\boldsymbol{s}{\upharpoonright_{1, l}}$ in $\mathscr{A}''$ for each $l \in \mathbb{N}$, for $\boldsymbol{s}{\upharpoonright_{\mathscr{D}}}$ in $\mathscr{D}$ and $\boldsymbol{s}{\upharpoonright_{\mathscr{E}}}$ in $\mathscr{E}$) if $\mathscr{A} = \neg \mathscr{A}'$ (respectively, if $\mathscr{A} = \mathscr{B} \times \mathscr{C}$, if $\mathscr{A} = \oc \mathscr{A}''$, if $\mathscr{A} = \mathscr{D} \multimap \mathscr{E}$);

\item \textsc{(Visibility)} If $\boldsymbol{s} = \boldsymbol{t} \vec{m} \boldsymbol{u} \vec{n} \boldsymbol{v}$ with $\vec{m}$ justifying $\vec{n}$, then $\vec{m}^\sharp$ occurs in the P-view $\lceil \boldsymbol{t} \vec{m} \boldsymbol{u} \rceil$ (respectively, in the O-view $\lfloor \boldsymbol{t} \vec{m} \boldsymbol{u} \rfloor$) when $\vec{n}$ is an extended P- (respectively, O-) move;

\item \textsc{(Joker)} Each element $\boldsymbol{s}(i)$ ($i \in \overline{|\boldsymbol{s}|}$) contains no extended vertex whose core is in $\alpha_{\mathscr{A}}[0]$.

\end{enumerate}

\end{definition}

\begin{notation}
We write $\mathcal{P}_{\mathscr{A}}$ for the set of all positions in a combinatorial arena $\mathscr{A}$.
\end{notation}

\begin{convention}
A \emph{\bfseries play} in a combinatorial arena $\mathscr{A}$ refers to a nonempty, prefix-closed sequence of positions in $\mathscr{A}$.
A play can be infinite in length; however, as in the case of other infinitary structures, we talk about infinite plays only temporarily; we will eventually dispense with them. 
\end{convention}

The recursive alternation axiom requires that every position is \emph{alternating} not only globally but also locally and recursively.
By Theorem~\ref{ThmFreeCharacterisation}, this condition is applied to \emph{all} combinatorial arenas, not only to inductively constructed ones. 
Positions in an arena \cite{hyland2000full} only satisfy global alternation, which is insufficient to achieve fully complete semantics of linear logic. 
Meanwhile, fully complete game semantics of (fragments of) linear logic \cite{abramsky1994games,murawski2003exhausting,abramsky1999concurrent,mellies2005asynchronous} equips each game with a potentially infinite set of (selected) positions that satisfy recursive alternation. 
We shall obtain fully complete yet finitary semantics of linear logic partly by imposing recursive alternation in a finitary way (specifically via Lemma~\ref{LemLeftLinearImplication}).
Next, visibility \cite{hyland2000full} is a standard axiom in the literature; it ensures that each view is a justified sequence. 
Finally, the joker axiom is a novel one, and it is for our full completeness to admit 0-ary additives, i.e., one and zero. 

\begin{remark}
An idea similar to the joker axiom, called \emph{joker moves}, is used by Murawski and Ong \cite{murawski2003exhausting} for their fully complete semantics of intuitionistic linear logic to include top.
In contrast to our method, they interpret top, not one, by a joker move, and endow it with a more special status.
Besides, the method of joker moves does not work in the presence of one or of-course. 
\end{remark}

\begin{example}
Among the justified sequences in the combinatorial arena $\mathscr{A}$ given in Example~\ref{ExJustifiedSequences}, only the first one is a position in $\mathscr{A}$; the other ones do not satisfy global alternation.

Next, the justified sequence $id . b . ic . a$ in the combinatorial arena $(\bot_{[a]} \otimes \bot_{[b]}) \multimap_{[i]} (\bot_{[c]} \otimes \bot_{[d]})$ is not a position since it does not satisfy recursive alternation (though it meets global alternation).

Further, the justified sequence $rg . mnf' . le' . e . f . e\{ e' ; 1 \}$ in the combinatorial arena $((\oc \bot_{[e]} \multimap_{[l]} \bot_{[e']}) \multimap_{[m]} (\bot_{[f]} \multimap_{[n]} \bot_{[f']})) \multimap_{[r]} \bot_{[g]}$ is not a position as the last move violates visibility. 

Finally, the justified sequence $io . p$ in the combinatorial arena $1_{[p]} \multimap_{[i]} 1_{[o]}$ is not a position because neither of the two moves satisfies joker.
\end{example}

Although positions in the examples given so far are not very eye-friendly, we can always make positions in a combinatorial arena readable by suitable \emph{syntactic-sugars}.
This is because each combinatorial arena has only finitely many vertices together with a finite number of vertex generators given by the exponential structure. 
The following example illustrates this point: 

\begin{example}
\label{ExArithmetic}
The \emph{\bfseries arithmetic combinatorial arena} refers to the combinatorial arena $\mathscr{N} \colonequals \oc (\bot_{[q]} \multimap_{[p]} \bot_{[y]}) \otimes \bot_{[n]} \multimap_{[o]} \bot_{[\hat{q}]}$.
Let us assign names to extended moves in $\mathscr{N}$ by
\begin{mathpar}
\vec{q}_0 \colonequals o\hat{q}
\and
\vec{\mathrm{yes}}_{i} \colonequals py \{ p^i ; 1^i \}
\and
\vec{q}_{i+1} \colonequals q \{ p^i ; 1^i \}
\and
\vec{\mathrm{no}} \colonequals n.
\end{mathpar}

The arithmetic combinatorial arena $\mathscr{N}$ is a finitary recast of the \emph{lazy natural number game} \cite[Example~23]{yamada2019game}, and it defines natural numbers in a syntax-free, primitive, finitary fashion through the idea of a \emph{counting game} as follows. 
First, each position in $\mathscr{N}$ is of the form
\begin{equation*}
\vec{q}_0 . \vec{\mathrm{yes}}_0 . \vec{q}_1 . \vec{\mathrm{yes}}_1 \dots \vec{q}_n . \vec{\mathrm{yes}}_n . \vec{q}_{n+1} . \vec{\mathrm{no}},
\end{equation*}
where $\vec{q}_0$ justifies $\vec{\mathrm{yes}}_i$ and $\vec{\mathrm{no}}$, and $\vec{\mathrm{yes}}_i$ justifies $\vec{q}_{i+1}$ ($0 \leqslant i \leqslant n$).
Then, this position can be read informally as follows: 
\begin{enumerate}

\item Opponent asks `Do you want to count one more?' by the initial move $\vec{q}_0$;

\item Player answers `Yes!' by the move $\vec{\mathrm{yes}}_0$ if she does, and `No!' by the move $\mathrm{no}$ otherwise, where the play stops in the latter case;

\item If the last move by Player is $\vec{\mathrm{yes}}_i$, then Opponent asks `Do you want to count one more?' by the move $\vec{q}_{i+1}$;
\label{Opponent}

\item Player answers `Yes!' by the move $\vec{\mathrm{yes}}_{i+1}$ if she does, and `No!' by the move $\vec{\mathrm{no}}$ otherwise, where again the play stops in the latter case;
\label{Player}

\item Opponent and Player iterate the steps \ref{Opponent} and \ref{Player} until Player makes the move $\vec{\mathrm{no}}$.

\end{enumerate}

\end{example}

\subsection{Combinatorial sequents}
\label{CombinatorialSequents}
Recall that games do not capture \emph{intensionality} in logic and computation such as \emph{cuts} \cite{gentzen1936widerspruchsfreiheit}, but \emph{dynamic games} \cite{yamada2020dynamic}, a generalisation of games, overcome this limitation.
Based on that work, we generalise combinatorial arenas to \emph{combinatorial sequents}, which may capture cuts. 

\begin{notation}
Given a switching vertex $v$ of a combinatorial arena $\mathscr{A}$, let $\mathscr{A}^v$ be the combinatorial subarena of $\mathscr{A}$ induced by moves containing $v$ (i.e., $|\mathscr{A}^v| = |\mathscr{A}|_{\{ \, u \in \mathscr{V}_{|\mathscr{A}|} \mid \exists m \in \mu_{\mathscr{A}} . \, \{ u, v \} \subseteq \underline{m} \, \}}$).
\end{notation}

\begin{definition}[combinatorial sequents]
\label{DefCombinatorialSequents}
A \emph{\bfseries combinatorial sequent} $\mathscr{S}$ is a combinatorial arena $\underline{\mathscr{S}}$ together with a fraternal set $\mathrm{Int}_\mathscr{S} \subseteq \mathscr{V}_{|\underline{\mathscr{S}}|}^1$ of switching vertices such that
\begin{enumerate}

\item The unique parent of elements of $\mathrm{Int}_\mathscr{S}$ is also switching, and its children are all in $\mathrm{Int}_\mathscr{S}$; 

\item For each $v \in \mathrm{Int}_\mathscr{S}$, $\underline{\mathscr{S}}^v \cong \oc^k (\mathscr{A} \multimap \mathscr{A})$ for some combinatorial arena $\mathscr{A}$ and number $k \in \mathbb{N}$;

\item If $v \in \mathrm{Int}_\mathscr{S}$ with $\{ v, w \} \in \alpha_{\underline{\mathscr{S}}}[2]$ or $\{ v, w \} \subseteq V \in \mathscr{V}_{\epsilon_{\underline{\mathscr{S}}}}$, then $w \in \mathrm{Int}_\mathscr{S}$,



\end{enumerate}
where $\mathrm{Int}_\mathscr{S}$ is called the \emph{\bfseries intensionality} of $\mathscr{S}$.
An element of $\alpha_{\underline{\mathscr{S}}}[2]$ or a vertex of $\mathscr{V}_{\epsilon_{\underline{\mathscr{S}}}}$ is said to be \emph{\bfseries internal} if its elements are all in $\mathrm{Int}_\mathscr{S}$, and \emph{\bfseries external} otherwise. 
\end{definition}

\begin{example}
\label{ExCombinatorialSequents}
The pair $((\bot_{[a]} \multimap_{[b]} \bot_{[c]}) \otimes (\top \multimap_{[d]} \top) \multimap_{[e]} \bot_{[f]}, \{ b, d \})$ is a combinatorial sequent, but the similar pair $((\bot_{[a]} \multimap_{[b]} (\bot_{[c]} \otimes (\top \multimap_{[d]} \top))) \multimap_{[e]} \bot_{[f]}, \{ b, d \})$ is not.

The pair $\mathscr{S} = (\underline{\mathscr{S}}, \mathrm{Int}_{\mathscr{S}})$ of $\underline{\mathscr{S}} \colonequals \bot_{[a]} \multimap_{[b]} \oc ((\bot_{[c]} \multimap_{[d]} \oc \bot_{[e]}) \multimap_{[f]} (\bot_{[g]} \multimap_{[h]} \oc \bot_{[i]})) \mathbin{\&} (\bot_{[m]} \mathbin{\&} \bot_{[n]} \multimap_{[j]} \bot_{[p]} \mathbin{\&} \bot_{[q]}) \multimap_{[k]} \bot_{[l]}$ and $\mathrm{Int}_{\mathscr{S}} \colonequals \{ f, j \}$ is a combinatorial sequent, where $\underline{\mathscr{S}}^f = \oc ((\bot_{[c]} \multimap_{[d]} \oc \bot_{[e]}) \multimap_{[f]} (\bot_{[g]} \multimap_{[h]} \oc \bot_{[i]}))$ and $\underline{\mathscr{S}}^j = \bot_{[m]} \mathbin{\&} \bot_{[n]} \multimap_{[j]} \bot_{[p]} \mathbin{\&} \bot_{[q]}$. 
Note that the elements $\{ m, n \}, \{ p, q \} \in \alpha_{\underline{\mathscr{S}}}[2]$ are external, while the one $\{ f, j \} \in \alpha_{\underline{\mathscr{S}}}[2]$ is internal.
\end{example}

The name of combinatorial sequents comes from their free characterisation (Proposition~\ref{PropFreeChracterisationsOfCombinatorialSequents}), which resembles sequents \cite{gentzen1936widerspruchsfreiheit}. 
This characterisation exploits the axioms on combinatorial sequents. 
In fact, a combinatorial sequent $\mathscr{S}$ corresponds to a sequent in our variant of a sequent calculus (\S\ref{CombinatorialFormalSystems}), where elements of $\mathrm{Int}_\mathscr{S}$ to cuts. 
Under this correspondence, the internal part of an additive or an exponential structure is between or on cuts in the sense clarified shortly.

\begin{notation}
Given a combinatorial sequent $\mathscr{S}$ together with an isomorphism $\underline{\mathscr{S}}^v \cong \oc^k (\mathscr{A} \multimap \mathscr{A})$ for some $v \in \mathrm{Int}_\mathscr{S}$, we write $\mathscr{S}(\vec{m}, \vec{n})$ or $\mathscr{S}(\vec{n}, \vec{m})$ if $\vec{m}$ is an extended move in one of the two copies of $\oc^k \mathscr{A}$, and $\vec{n}$ is the corresponding one in the other copy.
\end{notation}

\begin{definition}[positions in a combinatorial sequent]
A \emph{\bfseries position} in a combinatorial sequent $\mathscr{S}$ is a position $\boldsymbol{s}$ in $\underline{\mathscr{S}}$ such that if $\boldsymbol{t}\vec{m}\vec{n} \preceq \boldsymbol{s}$ is even, then $\mathscr{S}(\vec{m}, \vec{n})$.
\end{definition}

\begin{example}
Recall the combinatorial sequent $\mathscr{S} = (\underline{\mathscr{S}}, \mathrm{Int}_{\mathscr{S}})$ defined in Example~\ref{ExCombinatorialSequents}.
The position $bkl . fhi . de . c . g . a$ in $\underline{\mathscr{S}}$ is a position in $\mathscr{S}$, but the one $bkl . fhi . de . c . a$ in $\underline{\mathscr{S}}$ is not.
\end{example}

That is, a position in a combinatorial sequent $\mathscr{S}$ is a position in $\underline{\mathscr{S}}$ in which each extended O-move $\vec{n}$ following an extended P-move $\vec{m}$ is the copy of $\vec{m}$ in the sense specified by $\mathrm{Int}(\mathscr{S})$.
A combinatorial arena $\mathscr{A}$ is the same as the combinatorial sequent $(\mathscr{A}, \emptyset)$, and positions in $\mathscr{A}$ coincide with those in $(\mathscr{A}, \emptyset)$.
Thus, combinatorial sequents generalise combinatorial arenas. 

\begin{notation}
Given a combinatorial arena $\mathscr{A}$, we also write $\mathscr{A}$ for the combinatorial sequent $(\mathscr{A}, \emptyset)$.

Given a combinatorial sequent $\mathscr{S}$, we write $\mathcal{M}_{\mathscr{S}}$ (respectively, $\mathcal{M}_{\mathscr{S}}^+$, $\mathcal{P}_{\mathscr{S}}$) for the set of all moves (respectively, extended moves, positions) in the underlying combinatorial arena $\underline{\mathscr{S}}$.
\end{notation}

\begin{proposition}[closure of positions under views]
\label{PropClosureOfPositionsUnderViews}
If $\boldsymbol{s}$ is a position in a combinatorial sequent $\mathscr{S}$, then so are the P-view $\lceil \boldsymbol{s} \rceil$ and the O-view $\lfloor \boldsymbol{s} \rfloor$.
\end{proposition}
\begin{proof}
By induction on the length $|\boldsymbol{s}|$ of $\boldsymbol{s}$.
\end{proof}

Finally, let us establish a free characterisation of combinatorial sequents: 
\begin{definition}[combinatorial cuts]
\emph{\bfseries Combinatorial pre-cuts} are the class of combinatorial arenas generated from linear implication of the form $\mathscr{A} \multimap \mathscr{A}$ with $\mathscr{A}$ a combinatorial arena by with and of-course, and \emph{\bfseries combinatorial cuts} are the class of combinatorial arenas generated from top and the same kind of linear implication by tensor, with and of-course.
\end{definition}


\begin{notation}
Given combinatorial cuts $\mathscr{C}_j$ ($j \in \overline{m}$) and combinatorial arenas $\mathscr{A}_i$ ($i \in \overline{n}$) and $\mathscr{B}$, we write $\mathscr{A}_1, \mathscr{A}_2, \dots, \mathscr{A}_n \dashv \mathscr{C}_1, \mathscr{C}_2, \dots, \mathscr{C}_m \vdash \mathscr{B}$ for the combinatorial sequent $\mathscr{S}$ defined by
\begin{mathpar}
\underline{\mathscr{S}} \colonequals \bigotimes_{i=1}^n \mathscr{A}_i \multimap_{[\dashv]} \bigotimes_{j=1}^m \mathscr{C}_j \multimap_{[\vdash]} \mathscr{B}
\and
\mathrm{Int}_{\mathscr{S}} \colonequals \bigcup_{j=1}^m \{ \, v \in \mathscr{R}_{|\mathscr{C}_j|} \mid \text{$v$ is switching in $\mu_{\mathscr{C}_j}$} \, \},
\end{mathpar}
where $\bigotimes_{i=1}^n \mathscr{A}_i \colonequals \top$ if $n = 0$, $\bigotimes_{j=1}^m \mathscr{B}_j \colonequals \top$ if $m = 0$, and the identifiers $(\_)_{[\dashv]}$ and $(\_)_{[\vdash]}$ are fixed for convenience.
We also write $\mathscr{A}_1, \mathscr{A}_2, \dots, \mathscr{A}_n \dashv \mathscr{C}_1, \mathscr{C}_2, \dots, \mathscr{C}_m \vdash$ for $\mathscr{S}$ if $\mathscr{B} = \bot$.

The letters $\Gamma$, $\Delta$, etc. range over finite sequences of combinatorial arenas, and the ones $\Pi$, $\Sigma$, etc. over those of combinatorial cuts.
Of-course is applied to these sequences componentwisely.
\end{notation}

\begin{proposition}[a free characterisation of combinatorial sequents]
\label{PropFreeChracterisationsOfCombinatorialSequents}
Every combinatorial sequent can be written, up to graph isomorphisms, in the form $\mathscr{A}_1, \mathscr{A}_2, \dots, \mathscr{A}_n \dashv \mathscr{C}_1, \mathscr{C}_2, \dots, \mathscr{C}_m \vdash \mathscr{B}$ for combinatorial cuts $\mathscr{C}_j$ ($j \in \overline{m}$) and combinatorial arenas $\mathscr{A}_i$ ($i \in \overline{n}$) and $\mathscr{B}$, and further made into the form $\mathscr{A}_1, \mathscr{A}_2, \dots, \mathscr{A}_n \dashv \mathscr{C} \vdash \mathscr{B}$ with $\mathscr{C}$ a combinatorial cut.
\end{proposition}
\begin{proof}
The first part is by induction on the cardinality $|\mathrm{Int}_{\mathscr{S}}|$ for combinatorial sequents $\mathscr{S}$ with the help of the three axioms on $\mathscr{S}$ (Definition~\ref{DefCombinatorialSequents}), and the second by replacing the case $m = 0$ with $\mathscr{C} \colonequals \top$, and the case $m > 0$ with $\mathscr{C} \colonequals \bigotimes_{j=1}^m \mathscr{C}_j$.
\end{proof}

This result is the second step towards the biequivalences between logic and combinatorics (\S\ref{Bijections}).
In addition, this free characterisation, together with Theorem~\ref{ThmFreeCharacterisation}, gives us a practical notation to concisely represent combinatorial sequents.

\subsection{Finitely presentable strategies}
\label{FinitelyPresentableStrategies}
Next, let us adapt another central concept in game semantics to our combinatorial setting:
\begin{definition}[strategies \cite{nickau1994hereditarily,abramsky2000full,hyland2000full}]
\label{DefStrategies}
A \emph{\bfseries strategy} on a combinatorial sequent $\mathscr{S}$ is a subset $\varphi \subseteq \mathcal{P}_{\mathscr{S}}^{\mathrm{Even}}$, written $\varphi : \mathscr{S}$, that is nonempty, \emph{even-prefix-closed} (i.e., $\boldsymbol{s}\vec{m}\vec{n} \in \varphi$ implies $\boldsymbol{s} \in \varphi$) and \emph{deterministic} (i.e., $\boldsymbol{s}\vec{m}\vec{n}, \boldsymbol{s}\vec{m}\vec{n}' \in \varphi$ implies $\boldsymbol{s}\vec{m}\vec{n} = \boldsymbol{s}\vec{m}\vec{n}'$).
\end{definition}

A strategy $\varphi : \mathscr{S}$ depicts for Player \emph{how to play} on the combinatorial sequent $\mathscr{S}$ by mapping an odd-length position $\boldsymbol{s}\vec{m}$ in $\mathscr{S}$ to its extension $\boldsymbol{s}\vec{m}\vec{n} \in \varphi$ if it exists (n.b., $\vec{m}$ is an extended \emph{O-move}), where the extension, if any, is \emph{unique} by the determinacy of $\varphi$, and the map is \emph{partial} as there can be no extension of some odd-length position in $\varphi$.
For convenience, we often confuse a strategy with the induced partial map from odd-length positions to extended P-moves. 

\begin{example}
\label{ExamplesOfStrategies}
There are strategies $\{ \boldsymbol{\epsilon} \}, \underline{n} : \mathscr{N}$ (Example~\ref{ExArithmetic}), where $n \in \mathbb{N}$ and
\begin{equation*}
\underline{n} \colonequals \{ \, \boldsymbol{s} \in \mathcal{P}_{\mathscr{N}}^{\mathrm{Even}} \mid \boldsymbol{s} \preceq \vec{q}_0 . \vec{\mathrm{yes}}_0 . \vec{q}_1 . \vec{\mathrm{yes}}_1 \dots \vec{q}_n . \vec{\mathrm{yes}}_n . \vec{q}_{n+1} . \vec{\mathrm{no}} \, \}.
\end{equation*}
Loosely speaking, the trivial strategy $\{ \boldsymbol{\epsilon} \}$ does nothing, while the one $\underline{n}$ plays by counting $n$.
\end{example}
 
In contrast to these finite strategies, the strategy on ${\oc} {\neg} {\neg} {\top}$ that computes by the unique map for each copy of $\neg \neg \top$ generated by of-course is \emph{infinite}. 
Because our aim is to establish a \emph{finitary} foundation (\S\ref{Foreword}), we will eventually focus on \emph{finitely presentable} strategies in a suitable sense.

Next, let us recall that not every strategy corresponds to a (formal) proof.
For example, there is the trivial strategy on bottom $\bot \colonequals \neg \top$, but bottom models \emph{falsity}, which has no proof. 
Thus, this strategy cannot be an interpretation of a proof.
A standard way to carve out strategies for proofs is to impose \emph{winning} \cite[\S 2]{abramsky1997semantics} on strategies.
We adapt this idea to our framework: 
\begin{definition}[winning strategies \cite{hyland2000full,clairambault2010totality,coquand1995semantics,laird1997full}]
\label{DefConstraintsOnStrategies}
A strategy $\varphi : \mathscr{S}$ is 
\begin{itemize}

\item \emph{\bfseries Total} if it always responds to the last extended O-move: If $\boldsymbol{s} \vec{m} \in \mathcal{P}_{\mathscr{S}}^{\mathrm{Odd}}$ and $\boldsymbol{s} \in \varphi$, then there is some (necessarily unique) $\boldsymbol{s} \vec{m} \vec{n} \in \varphi$;

\item \emph{\bfseries Innocent} if it only depends on P-views: If $\boldsymbol{s} \vec{m} \vec{n} \in \varphi$, $\boldsymbol{t} \vec{l} \in \mathcal{P}_{\mathscr{S}}^{\mathrm{Odd}}$ and $\boldsymbol{t} \in \varphi$ satisfy $\lceil \boldsymbol{s} \vec{m} \rceil = \lceil \boldsymbol{t} \vec{l} \rceil$, then there is some $\boldsymbol{t} \vec{l}\vec{r} \in \varphi$ such that $\lceil \boldsymbol{s} \vec{m} \vec{n} \rceil = \lceil \boldsymbol{t} \vec{l} \vec{r} \rceil$;

\item \emph{\bfseries Noetherian} if there is no strictly increasing (with respect to the prefix relation $\preceq$) infinite sequence of elements in the set $\lceil \varphi \rceil = \{ \, \lceil \boldsymbol{s} \rceil \mid \boldsymbol{s} \in \varphi \, \}$ of all P-views in $\varphi$;

\item \emph{\bfseries Winning} if it is total, innocent and noetherian.

\end{itemize}
\end{definition}

\begin{example}
The trivial strategies $\{ \boldsymbol{\epsilon} \} : \bot$ and $\{ \boldsymbol{\epsilon} \} : \mathscr{N}$ are not total, while the ones $\{ \boldsymbol{\epsilon} \} : \top$ and $\underline{n} : \mathscr{N}$ for all $n \in \mathbb{N}$ are winning. 
\end{example}

One sees winning strategies as \emph{proofs} as follows.
First, proofs never get `stuck,' so strategies for proofs must be \emph{total}.
Next, recall that views are the `relevant' parts of positions from the logical perspective.
In syntax, the `irrelevant' parts correspond to \emph{states}.
Thus, imposing \emph{innocence} on strategies corresponds to excluding \emph{stateful} terms \cite{abramsky1997linearity,abramsky1998fully,abramsky1999full}.
Logic is concerned with the \emph{truth} of a formula, independently of `states of arguments,' so strategies for proofs are innocent.
Lastly, \emph{noetherianity} is to decide who `wins' in infinite plays: If a play by an innocent, noetherian strategy grows forever, then it is by Opponent so that the play is Player's `win.'

In addition to this standard constraint on strategies in the literature, our full completeness result (Theorem~\ref{ThmFullCompleteness}) needs another one that corresponds to the \emph{linearity} of proofs \cite{girard1987linear}:
\begin{definition}[linear strategies]
A P-move\footnote{Because linear strategies are innocent by definition, it suffices to focus on P-moves here, not extended ones.} $\vec{m}$ in a combinatorial sequent $\mathscr{S}$ is said to be
\begin{itemize}

\item \emph{\bfseries Subject to O-joker} if there is an edge $x \rightarrow y$ of $|\underline{\mathscr{S}}|$ such that $x \in \alpha_{\mathscr{S}}[0]$ and $y \in \underline{\vec{m}}$;

\item \emph{\bfseries Subject to P-exponential} if an element of $\underline{\vec{m}}$ is a vertex of $|\underline{\mathscr{S}}|_{V}$ for some $V \in \mathscr{V}_{\epsilon_{\underline{\mathscr{S}}}}$. 

\end{itemize}

An innocent strategy $\varphi : \mathscr{S}$ is said to be \emph{\bfseries linear} if the set $\{ \, \lceil \boldsymbol{s}\vec{o} \rceil \mid \boldsymbol{s}\vec{o}\vec{p} \in \varphi \, \}$ is singleton for each P-move $\vec{p} \in \mathcal{M}_{\mathscr{S}}$ not subject to O-joker or P-exponential, and \emph{\bfseries linearly winning} if it is linear and winning.
\end{definition}

That is, a strategy is linear if it injectively visits all P-moves (in the \emph{nondeterministic} sense explained in the examples below) in the underlying combinatorial arena except those in the effect of Opponent's one $1$ or Player's of-course $\oc$. 
Intuitively, this axiom matches the linearity of proofs \cite{girard1987linear}: A proof is linear precisely when it consumes and produces formulae (or resources) exactly as specified by the underlying sequent except those generated by the rule \textsc{$1$R} or \textsc{$\oc$W}.

\begin{remark}
Murawski and Ong \cite{murawski2003exhausting} impose \emph{P-exhaustivity} on strategies, which is similar to our linearity, for their fully complete game semantics of the multiplicative fragment of intuitionistic linear logic. 
A difference between the two is that their approach entails a drastic modification of the very notion of (deterministic) strategies \cite[p.~296]{murawski2003exhausting}, but ours does not. 

Another existing game-semantic approach to linear proofs is \emph{payoff functions} \cite[\S 4]{mellies2010resource}.
In contrast to linearity or P-exhaustivity of strategies, payoff functions are extrinsic to the structure of games or strategies, and a priori they do not reflect the intuition behind linear proofs.
\end{remark}

\begin{example}
The strategy on the combinatorial sequent $\bot_{[0]}, \oc \bot_{[1]}, \bot_{[2]} \dashv \vdash \bot_{[3]} \otimes \bot_{[4]}$ that maps $3 \mapsto 2$ and $4 \mapsto 0$ is linearly winning.
The same computation also forms a linearly winning strategy on the combinatorial sequent $\bot_{[0]}, \bot_{[1]}, \bot_{[2]} \dashv \vdash \bot_{[3]} \otimes \bot_{[4]} \otimes 1_{[5]}$.
In the both cases, the strategy injectively visits all P-moves not subject to O-joker or P-exponential \emph{nondeterministically}: Any single play by the strategy does not cover all the P-moves, but its plays collectively do.
The sequents $\bot_{[0]}, \oc \bot_{[1]}, \bot_{[2]} \vdash \bot_{[3]} \otimes \bot_{[4]}$ and $\bot_{[0]}, \bot_{[1]}, \bot_{[2]} \vdash \bot_{[3]} \otimes \bot_{[4]} \otimes 1_{[5]}$, which correspond to the two combinatorial sequents, respectively, are both provable in intuitionistic linear logic.

In contrast, the same computation yields a non-linear, though still winning, strategy on the combinatorial sequent $\bot_{[0]}, \bot_{[1]}, \bot_{[2]} \dashv \vdash \bot_{[3]} \otimes \bot_{[4]}$ due to the absence of of-course on $\bot_{[1]}$.

Another strategy on the first combinatorial sequent that plays $3 \mapsto 0$ and $4 \mapsto 0$ is winning but not linear: Its computation is (nondeterministically) not injective. 
\end{example}

\begin{example}
The strategy on the combinatorial sequent $\bot_{[0]}, \bot_{[1]} \multimap \top, \bot_{[2]} \dashv \vdash \bot_{[3]}$ mapping $3 \mapsto 2$ and $1 \mapsto 0$ is linearly winning.
The sequent $\bot_{[0]}, \bot_{[1]} \multimap \top, \bot_{[2]} \vdash \bot_{[3]}$, which corresponds to the combinatorial sequent, is provable in intuitionistic linear logic.
This example illustrates one of the difficult challenges in achieving fully complete semantics of intuitionistic linear logic: Standard game semantics satisfies the equality $(\bot \multimap \top) = \top$, so it does not have a strategy that interprets a proof of the sequent. 
We solve this problem by our novel linear implication, which attains the inequality $(\bot \multimap \top) \neq \top$, and pseudo-initial occurrences; see Theorem~\ref{ThmFullCompleteness}. 
\end{example}

\begin{example}
There are linearly winning strategies on the combinatorial sequent $\bot, \mathscr{A} \multimap \mathscr{A} \dashv \vdash \bot$ if $\mathscr{A}$ is top, bottom, $\oc \bot$, etc., but only winning, not linear, ones if $\mathscr{A}$ is $1$ or $\oc (\bot \multimap \top)$.
The corresponding sequent  $\bot, \mathscr{A} \multimap \mathscr{A} \vdash \bot$ is provable in intuitionistic linear logic in the former cases, but not in the latter cases.
\end{example}


Let us proceed to a key algorithmic property of linearly winning strategies: \emph{polynomial time decidability} (Theorem~\ref{ThmPolyTime}).
This property is necessary for those strategies to constitute a formal system or \emph{proof system} in the sense of \cite{cook1979relative}.
To this end, we introduce the following concepts:


\begin{definition}[flatly innocent strategies]
A strategy $\varphi : \mathscr{S}$ is said to be \emph{\bfseries flatly innocent} if the relation $\mathrm{fun}_\varphi \subseteq \underline{\lceil \mathcal{P}_{\mathscr{S}}^{\mathrm{Even}} \rceil} \times \mathcal{M}_{\mathscr{S}} \times \mathcal{M}_{\mathscr{S}}$ given by $\mathrm{fun}_\varphi (\underline{\boldsymbol{s}}, \vec{o}, \vec{p}) \ratio \Leftrightarrow \boldsymbol{s}\vec{o}\vec{p} \in \lceil \varphi \rceil$ forms a partial map $\underline{\lceil \mathcal{P}_{\mathscr{S}}^{\mathrm{Even}} \rceil} \times \mathcal{M}_{\mathscr{S}} \rightharpoonup \mathcal{M}_{\mathscr{S}}$ such that $\varphi = \tilde{\varphi}$ if $\mathrm{fun}_\varphi = \mathrm{fun}_{\tilde{\varphi}}$ for all innocent $\tilde{\varphi} : \mathscr{S}$.
\end{definition}

While innocent strategies compute on P-views, flatly innocent ones on \emph{flattened} P-views in the sense that they forget the order of elements of each input position except the last element. 

\begin{definition}[circular strategies]
An innocent strategy $\varphi : \mathscr{S}$ is said to be \emph{\bfseries circular} if $\mathrm{fun}_\varphi(S, \vec{o}, \vec{p})$ for some $S \in \underline{\lceil \mathcal{P}_{\mathscr{S}}^{\mathrm{Even}} \rceil}$ and $\vec{o}, \vec{p} \in \mathcal{M}_{\mathscr{S}}$ such that $\vec{p} \in S$ and $\vec{o}$ justifies $\vec{p}$.
\end{definition}

In other words, an innocent strategy is circular if its computation exhibits \emph{circularity} in the sense that it outputs what has been already played before. 

\begin{example}
\emph{Fixed-point strategies}, which will be given in Example~\ref{ExFixedPoints}, are an example of total innocent strategies that are not noetherian, flatly innocent or non-circular. 
\end{example}

The following lemma is interesting in its own right since it implies that an innocent strategy fails to be noetherian precisely when it is circular or not flatly innocent. 
This characterisation holds thanks to the finitary nature of combinatorial arenas in the sense that an infinite growth of a P-view occurs only through the action of of-course yet the action is invisible in P-views.

\begin{lemma}[flat innocence lemma]
\label{LemCircularity}
Every flatly innocent strategy is innocent, and an innocent strategy is flatly innocent if and only if it is non-circular if and only if it is noetherian.
\end{lemma}
\begin{proof}
Clearly, a strategy is not flatly innocent if it it is not innocent; this verifies the first part.

For the second part, let $\varphi : \mathscr{S}$ be an innocent strategy.
First, noetherianity clearly implies non-circularity.
Next, let us show that flat innocence implies non-circularity.
Assume that $\varphi$ is circular. 
Then, $\mathrm{fun}_\varphi$ does not distinguish two different P-views (by flattening them), so $\varphi$ is not flatly innocent. 
Finally, let us show that non-circularity implies noetherianity and flat innocence.
Let $\varphi$ be non-circular. 
Then, each P-view $\boldsymbol{s} \vec{m} \vec{n} \in \lceil \varphi \rceil$ does not have more than one occurrence of the same P-move.
This also implies that the P-view does not admit more than one occurrence of the same O-move either because in a P-view each occurrence of a P-move justifies at most one occurrence. 
Hence, the triple $(\underline{\boldsymbol{s}}, \vec{m}, \vec{n})$ completely recovers the P-view $\boldsymbol{s} \vec{m} \vec{n}$.
This verifies that $\varphi$ is flatly innocent. 
The same argument also shows that $\varphi$ is noetherian since otherwise one gets a contradiction to the finiteness of the set $\underline{\lceil \varphi \rceil}$.
\end{proof}

Because a strategy can be infinite, it is in general undecidable in conventional game semantics whether a given strategy is total, innocent or noetherian.
Nevertheless, Lemma~\ref{LemCircularity} implies that our combinatorial approach overcomes this undecidability in game semantics:
\begin{theorem}[polynomial time decidability of winning]
\label{ThmPolyTime}
Let $\mathscr{S}$ be a combinatorial sequent.
It only takes polynomial time to decide if a given (necessarily finite) relation $R \subseteq \underline{\lceil \mathcal{P}_{\mathscr{S}}^{\mathrm{Even}} \rceil} \times \mathcal{M}_{\mathscr{S}} \times \mathcal{M}_{\mathscr{S}}$ encodes a winning (respectively, linearly winning) strategy $\varphi : \mathscr{S}$ by $R = \mathrm{fun}_\varphi$.
\end{theorem}
\begin{proof}
The following algorithm only takes polynomial time (with respect to the size of $R$):
\begin{enumerate}

\item Decide if $R$ is a total map $\underline{\lceil \mathcal{P}_{\mathscr{S}}^{\mathrm{Even}} \rceil} \times \mathcal{M}_{\mathscr{S}} \rightarrow \mathcal{M}_{\mathscr{S}}$. Output `No!' and stop if not; go to the next step otherwise. 

\item Decide if $R$ is circular. 
Output `Yes!' and stop if not; output `No!' and stop otherwise.

\end{enumerate} 

By Lemma~\ref{LemCircularity}, this algorithm outputs `Yes!' (respectively, `No!') if and only if $R = \mathrm{fun}_\varphi$ holds (respectively, does not hold) for a unique winning strategy $\varphi : \mathscr{S}$. 

In addition, it is also polynomial time decidable whether a winning strategy is linear. 
\end{proof}

Let us next turn to constructions on strategies.
We adapt the novel, \emph{intensional} constructions on strategies introduced by Yamada and Abramsky \cite{yamada2020dynamic} to our combinatorial setting:

\begin{notation}
We generalise the notation $\upharpoonright$ used in Definition~\ref{DefPositionsInCombinatorialSemiArenas} and write $\boldsymbol{s}{\upharpoonright_{\mathscr{A}_{1}, \mathscr{A}_{2}, \dots, \mathscr{A}_{n}}}$ for the justified subsequence of a justified sequence $\boldsymbol{s}$ in a combinatorial arena $\mathscr{A}$, where $\mathscr{A}_{i}$ ($i \in \overline{n}$) are combinatorial subarenas of $\mathscr{A}$, that consists of extended moves in $\mathscr{A}_{i}$ for some $i \in \overline{n}$.
\end{notation}

\begin{definition}[constructions on strategies]
Given strategies $\varphi : (\Gamma \dashv \Pi \vdash \mathscr{A})$, $\varphi' : (\Gamma' \dashv \Pi' \vdash \mathscr{A}')$, $\psi : (\Gamma \dashv \Sigma \vdash \mathscr{B})$, $\tau : (\Delta, \mathscr{A} \dashv \Sigma \vdash \mathscr{B})$ and $\theta : (\oc \Gamma \dashv \oc \Pi \vdash \mathscr{A})$, we define
\begin{itemize}

\item The \emph{\bfseries copy-cat} $\mathrm{cp}_{\mathscr{A}} : (\mathscr{A}_{[0]} \dashv \vdash \mathscr{A}_{[1]})$ on $\mathscr{A}$ by
\begin{equation*}
\mathrm{cp}_{\mathscr{A}} \colonequals \{ \, \boldsymbol{s} \in \mathcal{P}_{\mathscr{A}_{[0]} \dashv \vdash \mathscr{A}_{[1]}}^{\mathrm{Even}} \mid \forall \boldsymbol{t} \preceq \boldsymbol{s} . \, \mathrm{Even}(\boldsymbol{t}) \Rightarrow \boldsymbol{t} {\upharpoonright_{\mathscr{A}_{[0]}}} = \boldsymbol{t} {\upharpoonright_{\mathscr{A}_{[1]}}}, \mathrm{Init}_{\mathscr{A}_{[0]}, \mathscr{A}_{[1]}}(\boldsymbol{s}) \, \},
\end{equation*}
where the predicate $\mathrm{Init}_{\mathscr{A}_{[0]}, \mathscr{A}_{[1]}}(\boldsymbol{s})$ means that each initial occurrence in $\boldsymbol{s}$ on $\mathscr{A}_{[0]}$ is justified by the last initial one on $\mathscr{A}_{[1]}$;

\item The \emph{\bfseries parallel composition} $\varphi \between \tau : (\Gamma, \Delta \dashv \Pi, \mathscr{A}_{[0]} \multimap \mathscr{A}_{[1]}, \Sigma \vdash \mathscr{B})$ of $\varphi$ and $\tau$ by
\begin{equation*}
\varphi \between \tau \colonequals \{ \, \boldsymbol{s} \in \mathcal{P}_{\Gamma, \Delta \dashv \Pi, \mathscr{A}_{[0]} \multimap \mathscr{A}_{[1]}, \Sigma \vdash \mathscr{B}}^{\mathrm{Even}} \mid {\boldsymbol{s} {\upharpoonright_{\Gamma, \Pi, \mathscr{A}_{[0]}}}} \in \varphi, {\boldsymbol{s} {\upharpoonright_{\Delta, \mathscr{A}_{[1]}, \Sigma, \mathscr{B}}}} \in \tau \, \};
\end{equation*}

\item The \emph{\bfseries tensor} $\varphi \otimes \varphi' : (\Gamma, \Gamma' \dashv \Pi, \Pi' \vdash \mathscr{A} \otimes \mathscr{A}')$ of $\varphi$ and $\varphi'$ by
\begin{equation*}
\varphi \otimes \varphi' \colonequals \{ \, \boldsymbol{s} \in \mathcal{P}_{\Gamma, \Gamma' \dashv \Pi, \Pi' \vdash \mathscr{A} \otimes \mathscr{A}'}^{\mathrm{Even}} \mid \boldsymbol{s} {\upharpoonright_{\Gamma, \Pi, \mathscr{A}}} \in \varphi, \boldsymbol{s} {\upharpoonright_{\Gamma', \Pi', \mathscr{A}'}} \in \varphi' \, \};
\end{equation*}

\item The \emph{\bfseries pairing} $\langle \varphi, \psi \rangle : (\Gamma \dashv \Pi \mathbin{\&} \Sigma \vdash \mathscr{A} \mathbin{\&} \mathscr{B})$ of $\varphi$ and $\psi$, if $\Pi$ and $\Sigma$ are both singleton (n.b., this does not lose generality since $\Pi$ and $\Sigma$ are combinatorial cuts, not pre-cuts), by
\begin{equation*}
\langle \varphi, \psi \rangle \colonequals \{ \, \boldsymbol{s} \in \mathcal{P}_{\Gamma \dashv \Pi \mathbin{\&} \Sigma \vdash \mathscr{A} \mathbin{\&} \mathscr{B}}^{\mathrm{Even}} \mid \boldsymbol{s} {\upharpoonright_{\Gamma, \Pi, \mathscr{A}}} \in \varphi, \boldsymbol{s} {\upharpoonright_{\Gamma, \Sigma, \mathscr{B}}} \in \psi \, \};
\end{equation*}

\item The \emph{\bfseries dereliction} $\mathrm{der}_{\mathscr{A}} : \oc \mathscr{A} \dashv \vdash \mathscr{A}$ on $\mathscr{A}$ by 
\begin{equation*}
\mathrm{der}_{\mathscr{A}} \colonequals \{ \, \boldsymbol{s} \in \mathcal{P}_{\oc \mathscr{A} \dashv \vdash \mathscr{A}}^{\mathrm{Even}} \mid \forall \boldsymbol{t} \preceq \boldsymbol{s} . \, \mathrm{Even}(\boldsymbol{t}) \Rightarrow \boldsymbol{t} {\upharpoonright_{\oc \mathscr{A}}} = \boldsymbol{t} {\upharpoonright_{\mathscr{A}}}, \mathrm{Init}_{\oc \mathscr{A}, \mathscr{A}}(\boldsymbol{s}) \, \};
\end{equation*}

\item The \emph{\bfseries promotion} $\theta^\dagger : \oc \Gamma \dashv \oc \Pi \vdash \oc \mathscr{A}$ of $\theta$ by
\begin{equation*}
\theta^\dagger \colonequals \{ \, \boldsymbol{s} \in \mathcal{P}_{\oc \Gamma \dashv \oc \Pi \vdash \oc \mathscr{A}}^{\mathrm{Even}} \mid {\boldsymbol{s} {\upharpoonright_{\oc \Gamma, \oc \Pi, \oc \mathscr{A}}}} \in \theta \, \};
\end{equation*} 

\item The \emph{\bfseries transpose} $\lambda(\tau) : \Delta \dashv \Sigma \vdash \mathscr{A} \multimap \mathscr{B}$ of $\tau$ by adjusting the tags on $\mathscr{A}$ in $\tau$.

\end{itemize}

\end{definition}


There is one more operation imported from dynamic game semantics \cite{yamada2020dynamic}, which is significantly simplified thanks to the intensionality \emph{explicitly} displayed in a combinatorial sequent:

\begin{definition}[the big-step hiding operation]
\label{DefCombinatorialBigStepCutElimination}
The \emph{\bfseries big-step hiding operation} is the function $\mathcal{H}^\omega$ that maps each combinatorial strategy $\varphi : (\Gamma \dashv \Pi \vdash \Phi)$ to the one 
\begin{equation*}
\mathcal{H}^\omega(\varphi) \colonequals \{ \, \boldsymbol{s}{\upharpoonright_{\Gamma, \Phi}} \mid \boldsymbol{s} \in \varphi \, \} : \mathcal{H}^\omega(\mathscr{S}) \colonequals (\Gamma \dashv \vdash \Phi), 
\end{equation*}
and the \emph{\bfseries hiding equivalence} is the equivalence relation $\simeq_{\mathcal{H}}^\omega$ between strategies defined by
\begin{equation*}
\varphi : \mathscr{S} \simeq_{\mathcal{H}}^\omega \varphi' : \mathscr{S}' \ratio \Leftrightarrow \mathcal{H}^\omega(\varphi) = \mathcal{H}^\omega(\varphi') : \mathcal{H}^\omega(\mathscr{S}) = \mathcal{H}^\omega(\mathscr{S}').
\end{equation*}
\end{definition}

The big-step hiding operation $\mathcal{H}^\omega$ deletes the intensionality of a combinatorial sequent and a strategy \emph{in one go}.
We refine it into a \emph{step-by-step one} $\mathcal{H}$ by Corollary~\ref{CorCombinatorialCutElim}, whose countable iteration coincides with $\mathcal{H}^\omega$, in such a way that $\mathcal{H}$ corresponds precisely to cut-elimination.

\begin{proposition}[well-defined constructions on strategies]
\label{PropWellDefinedConstructionsOnStrategies}
Copy-cats and derelictions are linearly winning strategies; strategies are closed under parallel composition, tensor, pairing, promotion, transpose and hiding, and they preserve winning, linear winning and the hiding equivalence.
\end{proposition}
\begin{proof}
Straightforward and left to the reader. 
\end{proof}

At the end of the present section, let us introduce a \emph{finite presentation} of a strategy:
\begin{definition}[P-view algorithms]
\label{DefViewAlgorithms}
A \emph{\bfseries P-view algorithm} on a combinatorial sequent $\mathscr{S}$ is a (necessarily finite) partial map $\underline{\lceil \mathcal{P}_{\mathscr{S}}^{\mathrm{Even}} \rceil} \times \mathcal{M}_{\mathscr{S}} \rightharpoonup \mathcal{M}_{\mathscr{S}}$.
\end{definition}

\begin{definition}[finitely presentable strategies]
\label{DefFinitePresentations}
The strategy $\mathrm{st}_{\mathscr{S}}(f) : \mathscr{S}$ \emph{\bfseries finitely presented} by a P-view algorithm $f$ on a combinatorial sequent $\mathscr{S}$ is defined inductively by
\begin{equation}
\label{EquationOfFinitePresentations}
\mathrm{st}_{\mathscr{S}}(f) \colonequals \{ \boldsymbol{\epsilon} \} \cup \{ \, \boldsymbol{s}\vec{m}\vec{n} \in \mathcal{P}_{\mathscr{S}}^{\mathrm{Even}} \mid \boldsymbol{s} \in \mathrm{st}_{\mathscr{S}}(f), f(\underline{\lceil \boldsymbol{s} \rceil}, \vec{m}^\sharp) = \vec{n}^\sharp \, \}.
\end{equation}

A strategy $\varphi : \mathscr{S}$ is \emph{\bfseries finitely presentable} if there is a P-view algorithm $f$ on $\mathscr{S}$ that satisfies the equation $\mathrm{st}_{\mathscr{S}}(f) = \varphi$, where $f$ is called a \emph{\bfseries finite presentation} of $\varphi$.
\end{definition}

\begin{remark}
By the second half of the first axiom of Definition~\ref{DefJustifiedSequences}, $p$ in the equation (\ref{EquationOfFinitePresentations}) is \emph{unique}.
It then follows that this equation (\ref{EquationOfFinitePresentations}) yields a well-defined flatly innocent strategy.
Clearly, every flatly innocent strategy is finitely presentable, but not vice versa as we shall see shortly. 
\end{remark}

\begin{notation}
We write $f :: \mathscr{S}$ if $f$ is a P-view algorithm of a combinatorial sequent $\mathscr{S}$, and $\mathrm{st}(f)$ for the strategy $\mathrm{st}_{\mathscr{S}}(f) : \mathscr{S}$ when the underlying combinatorial sequent $\mathscr{S}$ is evident. 
\end{notation}

As the naming indicates, P-view algorithms are to serve as \emph{finite presentations} of strategies. 
A finite presentation of a strategy exhibits \emph{computability} of the strategy in the evident sense.
We show in \S\ref{Computation} that finitely presentable strategies form a powerful model of higher-order computation that is not only Turing complete but also \emph{PCF-complete}, i.e., at least as strong as PCF.

As indicated by the notion of flat innocence, a P-view algorithm may finitely present more than one innocent strategy; i.e., P-view algorithms are more \emph{abstract} than innocent strategies (by flattening P-views). 
This abstract nature of P-view algorithms is notable since it stands in sharp contrast with existing approaches to computability which rely on symbolic methods with superfluous details for defining computability of more abstract, semantic objects.
As a result, one does not have to care if a property of a strategy in terms of its finite presentation is invariant under the choice of the finite presentation; the proof of Theorem~\ref{ThmPolyTime} is an example.
This is a strong advantage as it is often nontrivial to prove the invariance of a property under the choice of representations. 
More generally, inessential details of the symbolic manipulations in theory of computation, e.g., Turing machines and lambda-calculi, keep mathematician away from the field, but the semantic, abstract nature of our method changes this unfortunate situation. 

Parallel composition, tensor and pairing take the disjoint unions of underlying combinatorial sequents, and promotion and transpose essentially do not alter innocent strategies. 
Hence:
\begin{proposition}[preservation of flat innocence and finite presentability]
\label{PropPreservationOfFlatInnocenceAndFinitePresentability}
Copy-cats and derelictions are flatly innocent and therefore finitely presentable; parallel composition, tensor, pairing, promotion and transpose preserve flat innocence and finite presentability.
\end{proposition}

On the other hand, the big-step hiding operation does not preserve the finite presentability, let alone the flat innocence, of a strategy as the following example demonstrates: 

\begin{example}
\label{ExFinitePresentableStrategies}
The strategies $\underline{0}, \underline{1} : \mathscr{N}$ are both linearly winning and flatly innocent, but the one $\underline{n} : \mathscr{N}$ is neither for each $n > 1$.
How shall we finitely present the strategy $\underline{n} : \mathscr{N}$ then?

To answer this question, observe that the strategy $\mathrm{succ} : (\mathscr{N} \dashv \vdash \mathscr{N}')$, where the superscript $(\_)'$ is a tag to distinguish the two copies of the arithmetic combinatorial arena $\mathscr{N}$, defined by
\begin{footnotesize}
\begin{equation*}
\mathrm{succ} \colonequals \{ \, {\dashv\vdash}(\vec{q}_0)'  . \vec{\mathrm{yes}}'_0 ((\vec{q}_1)' . \vec{q}_0 . \vec{\mathrm{yes}}_0 . \vec{\mathrm{yes}}'_{1}) ((\vec{q}_2)' . \vec{q}_1 . \vec{\mathrm{yes}}_{1} . \vec{\mathrm{yes}}'_{2}) \dots ((\vec{q}_{n})' . \vec{q}_{n-1} . \vec{\mathrm{yes}}_{n-1} . \vec{\mathrm{yes}}'_{n}) ((\vec{q}_{n+1})' . \vec{q}_{n} . \vec{\mathrm{no}} . \vec{\mathrm{no}}') \mid n \in \mathbb{N} \, \}
\end{equation*}
\end{footnotesize}is linearly winning and flatly innocent; we leave the verification to the reader. 
Then, the strategy 
\begin{equation*}
\mathrm{succ}^n(\underline{0}) \colonequals \underline{0} \between \mathrm{succ} \between \mathrm{succ} \between \dots \between \mathrm{succ} : (\dashv (\mathscr{N} \multimap \mathscr{N})^n \vdash \mathscr{N})
\end{equation*}
is linearly winning and flatly innocent by Propositions~\ref{PropWellDefinedConstructionsOnStrategies} and \ref{PropPreservationOfFlatInnocenceAndFinitePresentability}, and has  $\mathcal{H}^\omega(\mathrm{succ}^n(\underline{0})) = \underline{n}$.
Hence, although $\underline{n}$ itself is not finitely presentable, its \emph{intensional refinement} $\mathrm{succ}^n(\underline{0})$ is.
\end{example}

The above example also illustrates the fine analysis of \emph{computational complexity} given by our intensional approach. 
For instance, take $n \colonequals 100$.
Recall that the strategy $\underline{n} : \mathscr{N}$ answers each question by Opponent, and the answer depends on the number of previous question answering. 
Hence, it is intuitively much easier for the strategy $\underline{n}$ to answer the \emph{initial} question by saying `counting one more' than the \emph{99th} question because in the first case it is immediately obvious if the answer is the correct one but not in the second case.
However, the strategy $\underline{n}$ counts the both answers \emph{equally as single steps}; i.e., it cannot reflect the algorithmic difference between the two answers. 
Nevertheless, by shifting from strategies to finitely presentable ones, `too big' computational steps such as answering the 99th question in $\mathscr{N}$ are banned, and we are led to the intensional one $\mathrm{succ}^n(\underline{0})$ that captures the algorithmic difference between the two answers.  

In comparison with Turing machines \cite{turing1937computable}, the standard foundation of computational complexity theory, the significance of this framework is that it is applicable to computational complexity of arbitrarily higher types (\S\ref{Computation}).
Note that \emph{higher-order computational complexity theory} lacks a proper mathematical foundation, and the concept of higher-order computational complexity has stopped at type 2.
In contrast, our game-semantic model of computation given in \S\ref{Computation} captures computation at arbitrarily higher types, and it seems to be a promising solution to this problem. 
For lack of space, however, we leave it as future work to exploit this research direction.

\subsection{Combinatorial bicategories of logic and computation}
\label{Categories}
This section studies the algebras of combinatorial arenas and finitely presentable strategies in terms of B\`{e}nabou's \emph{bicategories} \cite{benabou1967introduction}.
Whereas ordinary semantics of logic and computation \cite{lambek1988introduction} identifies proofs and programs with their \emph{values}, which makes its structure (1-)categories, Yamada and Abramsky \cite{yamada2020dynamic} showed that its extension to non-values entails a generalisation of the categorical structure into a bicategorical one. 
We follow this bicategorical account.

In terms of our framework, their idea is as follows.
A strategy $\varphi : \mathscr{S}$ is said to be \emph{normalised} or a \emph{value} if $\mathrm{Int}_{\mathscr{S}} = \emptyset$, and the big-step hiding operation $\mathcal{H}^\omega$ maps non-values to their values. 
In the (conventional) category of games, morphisms are normalised strategies, and the composition corresponds to the operation $\mathcal{H}^\omega(\_ \between \_)$ in our approach.
Thus, for admitting non-values, it seems necessary to replace this composition with parallel composition. 
However, parallel composition does not satisfy the unit law, e.g., $(\mathrm{id}_{\mathscr{N}} \between \mathrm{succ}) \neq \mathrm{succ}$, where identities are copy-cats.
The idea of Yamada and Abramsky is then to relax the category into a bicategory, where 2-cells are the hiding equivalence, because copy-cats satisfy the bicategorical unit law, e.g., $(\mathrm{id}_{\mathscr{N}} \between \mathrm{succ}) \simeq_{\mathcal{H}}^\omega \mathrm{succ}$.

Why should one study the algebras of our combinatorics?
One reason is clarity: The familiar categorical language will help one understand exotic combinatorial structures and relate them to other mathematical structures.
Another reason is generality: The present categorical framework, once established, will be general enough to be applied to other semantics as well.

We focus on a class of bicategories or \emph{E-categories} as in \cite{yamada2020dynamic}.\footnote{E-categories are called $\beta$-categories in \cite{yamada2020dynamic}. We change the name to avoid the conflict with $\beta$-reduction.}
An E-category is a category except that the equality between morphisms is replaced with an equivalence relation.
Thus, the composition and identities in an E-category respect the equivalence relation, so does a functor between E-categories or an \emph{E-functor}, and so does a natural transformation between E-functors or an \emph{E-natural transformation}.
If we regard the equivalence relation as 2-cells, then each E-category forms a bicategory in the evident way \cite{yamada2020dynamic}.
For brevity, we work on E-categories: 

\begin{convention}
We say that an operation on morphisms is \emph{up to an equivalence relation $\simeq$} between morphisms if the operation preserves the relation $\simeq$, and similarly for a property of morphisms. 
\end{convention}

\begin{definition}[combinatorial bicategories]
The E-category $\mathcal{G}$ consists of the following data:
\begin{itemize}

\item An object is a combinatorial arena;

\item A morphism $\mathscr{A} \rightarrow \mathscr{B}$ is a pair $(\Pi, \varphi)$ of a finite sequence $\Pi$ of combinatorial cuts and a finitely presentable strategy $\varphi$ on the combinatorial sequent $\mathscr{A} \dashv \Pi \vdash \mathscr{B}$, which is said to be \emph{\bfseries normalised} or a \emph{\bfseries value} if $\Pi = \boldsymbol{\epsilon}$;

\item The composition of morphisms $(\Pi, \varphi) : \mathscr{A} \rightarrow \mathscr{B}$ and $(\Sigma, \psi) : \mathscr{B} \rightarrow \mathscr{C}$ is the pair 
\begin{equation*}
(\Pi, \varphi) ; (\Sigma, \psi) \colonequals ((\Pi, \mathscr{B} \multimap \mathscr{B}, \Sigma), \varphi \between \psi) : \mathscr{A} \rightarrow \mathscr{C};
\end{equation*}

\item The identity $\mathrm{id}_{\mathscr{A}}$ is the pair $(\boldsymbol{\epsilon}, \mathrm{cp}_{\mathscr{A}}) : \mathscr{A} \rightarrow \mathscr{A}$;

\item The equivalence relation $\simeq_{\mathscr{A}, \mathscr{B}}$ between morphisms $(\Pi, \varphi), (\Pi', \varphi') : \mathscr{A} \rightrightarrows \mathscr{B}$ is given by
\begin{equation*}
(\Pi, \varphi) \simeq_{\mathscr{A}, \mathscr{B}} (\Pi', \varphi') \ratio\Leftrightarrow \mathcal{H}^\omega(\varphi) = \mathcal{H}^\omega(\varphi'),
\end{equation*}
i.e., $\simeq_{\mathscr{A}, \mathscr{B}}$ is the hiding equivalence $\simeq_{\mathcal{H}}^\omega$, where we often omit the subscripts $(\_)_{\mathscr{A}, \mathscr{B}}$ on $\simeq$.

\end{itemize}

An subcategory $\mathcal{LG}$ of $\mathcal{G}$ up to $\simeq$ is obtained from $\mathcal{G}$ by requiring the strategy $\varphi$ of each morphism $(\Pi, \varphi) : \mathscr{A} \rightarrow \mathscr{B}$ to be linearly winning and satisfy the following two conditions, called the \emph{\bfseries logical intensionality axioms}, on its computation on $\Pi$:

\begin{enumerate}

\item Each internal edge $e = \{ x, y \} \in \alpha_\Pi[2]$ with $x$ or $y$ occurring in $\varphi$ admits a unique external edge $e' = \{ x', y' \} \in \alpha_{\mathscr{A} \dashv \vdash \mathscr{B}}[2]$ such that, for all $\boldsymbol{s}\vec{o} \boldsymbol{t} \vec{p} \in \lceil \varphi \rceil$ with $\vec{p}$ justified by $\vec{o}$, $e' \cap \underline{\vec{o}} \neq \emptyset$ if and only if $e \cap \underline{\vec{p}} \neq \emptyset$;

\item Each internal vertex $V \in \mathscr{V}_{\epsilon_\Pi}$ with an element of $V$ occurring in $\varphi$ admits a unique external vertex $V' \in \mathscr{V}_{\epsilon_{\mathscr{A} \dashv \vdash \mathscr{B}}}$ such that, for all $\boldsymbol{s}\vec{o} \boldsymbol{t} \vec{p} \in \lceil \varphi \rceil$ with $\vec{p}$ justified by $\vec{o}$, $\vec{p}$ is of the form $p_0\{ y'; j' \}$ with $y' \in V'$ if and only if $\vec{o}$ is of the form $o_0\{ y; j \}$ with $y \in V$.

\end{enumerate}
\end{definition}

\begin{notation}
We also write $\varphi : (\mathscr{A} \dashv \Pi \vdash \mathscr{B})$ for a 1-cell $(\Pi, \varphi) : \mathscr{A} \rightarrow \mathscr{B}$ in these bicategories.
\end{notation}

It is straightforward to verify that $\mathcal{G}$ and $\mathcal{LG}$ form E-categories, thence bicategories; we leave it as an exercise.
As we shall see, the linear winning condition on 1-cells suffices for $\mathcal{LG}$ to attain fully complete semantics of intuitionistic linear logic with respect to cut-free proofs; the logical intensionality axioms are to extend the full completeness to proofs with cuts (Theorem~\ref{ThmFullCompleteness}). 

\begin{example}
The strategy on $\bot_{[0]} \otimes \bot_{[1]} \dashv (\bot_{[2]} \multimap \bot_{[3]}) \mathbin{\&} (\bot_{[4]} \multimap \bot_{[5]}) \vdash \bot_{[6]} \otimes (\bot_{[7]} \mathbin{\&} \bot_{[8]})$ that computes $6 \mapsto 0$, $7 \mapsto 3$, $8 \mapsto 5$, $3 \mapsto 0$ and $4 \mapsto 1$ does not satisfy the logical intensionality axioms, but the one that computes $6 \mapsto 3$, $6 \mapsto 5$, $2 \mapsto 0$, $4 \mapsto 0$, $7 \mapsto 1$ and $8 \mapsto 1$ does.
\end{example}

These E-categories form standard categorical semantics of intuitionistic linear logic:
\begin{definition}[new Seely categories \cite{bierman1995categorical}]
\label{DefNewSeelyCategories}
A \emph{\bfseries new Seely category} is a symmetric monoidal closed category $\mathcal{C} = (\mathcal{C}, \otimes, \top, \multimap)$ that has finite products $(1, \&)$ and is equipped with a comonad $\oc$ on $\mathcal{C}$ such that the canonical adjunction between the co-Kleisli category $\mathcal{C}_\oc$ and $\mathcal{C}$ is monoidal.
\end{definition}

\begin{proposition}[combinatorial new Seely categories]
The E-categories $\mathcal{G}$ and $\mathcal{LG}$ give rise to new Seely categories up to the hiding equivalence $\simeq$ between morphisms.
\end{proposition}
\begin{proof}[Proof]
The notations for the categorical constructions in Definition~\ref{DefNewSeelyCategories} signify the corresponding ones in $\mathcal{G}$ and $\mathcal{LG}$.
We leave out the details because the 1-categorical case for conventional game semantics is well-known and outlined in \cite{hyland1997game}, which the present case follows.
\end{proof}

In conventional game semantics, the unit $\top$ and the terminal object $1$ \emph{coincide}. 
Clearly, this degeneracy makes it hopeless for the game semantics to achieve full completeness for intuitionistic linear logic because the logic distinguishes top and one (i.e., the corresponding formulae). 

In contrast, the E-category $\mathcal{LG}$ satisfies $\top \neq 1$.
It is then interesting to see how top fails to be the unit for product in $\mathcal{LG}$.
Assume for a contradiction that top is the unit for product in $\mathcal{LG}$. 
Then, there is a unique morphism $\tau : \bot \rightarrow \top$ because the assumption induces the isomorphism $\top \cong 1$.
However, there is no such $\tau$ in $\mathcal{LG}$ due to the linearity axiom on $\tau$, a contradiction.

Last but not least, the (co-)Kleisli constructions on $\mathcal{G}$ or $\mathcal{LG}$ associated to the comonad $\oc$ and the monads $\neg\neg$ and $\wn \colonequals \neg \oc \neg$ are well-defined as they respect the hiding equivalence $\simeq$.
These constructions yield \emph{cartesian closed categories}, \emph{star-autonomous categories} \cite{barr2006autonomous} and \emph{control categories} \cite{selinger2001control} up to $\simeq$; they are respectively the categorical semantics of intuitionistic logic \cite{lambek1988introduction}, classical linear logic \cite{seely1987linear} and classical logic \cite{selinger2001control}.
Although the following fact is not strictly necessary for the present work, it says that our combinatorics induces these categories:

\begin{corollary}[(co-)Kleisli extensions to non-linearity and classicality]
The co-Kleisli categories $\mathcal{G}_\oc$ and $\mathcal{LG}_\oc$ are cartesian closed, the Kleisli ones $\mathcal{G}_{\neg\neg}$ and $\mathcal{LG}_{\neg\neg}$ are star-autonomous, and the Kleisli ones $\mathcal{G}_{\oc\wn}$ and $\mathcal{LG}_{\oc\wn}$ are control, up to the hiding equivalence $\simeq$ on morphisms. 
\end{corollary}

As advertised in \S\ref{Foreword}, the \emph{consistency} of our logical bicategories are immediate: 
\begin{proposition}[consistency]
\label{PropConsistency}
The bicategories $\mathcal{LG}$, $\mathcal{LG}_{\neg\neg}$, $\mathcal{LG}_\oc$ and $\mathcal{LG}_{\oc\wn}$ are all consistent in the sense that they have no 1-cells $\top \rightarrow \bot$ or $\top \rightarrow 0$.
\end{proposition}
\begin{proof}
By the totality of 1-cells in these bicategories. 
\end{proof}

\section{Combinatorial formal systems}
\label{CombinatorialFormalSystems}
This section presents a combinatorial reformulation of formal systems. 
Specifically, we establish \emph{biequivalences} between $\mathcal{LG}$ and a bicategory $\mathsf{ILL}$ of intuitionistic linear logic, between $\mathcal{LG}_{\neg\neg}$ and that $\mathsf{CLL}$ of classical linear logic, between $\mathcal{LG}_\oc$ and that $\mathsf{IL}$ of intuitionistic logic, and between $\mathcal{LG}_{\oc\wn}$ and that $\mathsf{CL}$ of classical logic. 
Moreover, we define for each of the combinatorial bicategories a step-by-step operation on 1-cells that precisely corresponds to the cut-elimination in the syntax.
These results, together with Theorem~\ref{ThmPolyTime}, justify our combinatorics as formal systems. 

This result resolves the bottleneck of mathematical semantics (\S\ref{Foreword}) completely by recasting cuts and cut-elimination syntax-freely and non-inductively.
Even when one focuses on cut-free proofs, our result solves a problem open for thirty years: fully complete semantics of intuitionistic linear logic.
In addition, the correspondence between proofs and strategies are remarkably tight: One may directly and non-inductively read off proofs as strategies, and vice versa. 

Also, our method captures the relation between the logics in terms of categorical algebras. 
This method admits an \emph{intuitive} reading too: The Kleisli extension $(\_)_\wn$ allows \emph{unlimited hypothesis consumptions}, and the co-Kleisli one $(\_)_\oc$ does \emph{unlimited reasoning do-overs}.


As the syntax of the logics, we adopt the \emph{sigma-calculus} [Yam23], a novel term calculus for intuitionistic linear logic and its (co-)Kleisli extensions. 
This calculus has the tightest correspondence with $\mathcal{LG}$ as mentioned above.
This is our primary motivation to use the calculus. 
Besides, whilst other term calculi for linear logic \cite{abramsky1993computational,benton1993term,barber1996dual,bierman1999classical} suffer from complex \emph{commuting conversions}, the sigma-calculus does not.
This abstract nature of the sigma-calculus makes the correspondences between the combinatorics and the syntax plain and direct. 

We first review the syntax of the logics in terms of the sigma-calculus and interpret them by combinatorics in \S\ref{TheLogics}.
We finally prove that the interpretations form biequivalences in \S\ref{Bijections}.

\subsection{Linear, intuitionistic and classical logics}
\label{TheLogics}
Let us first recall the formal languages of (propositional) linear, intuitionistic and classical logics:
\begin{definition}[formulae in classical linear logic \cite{girard1987linear}]
\label{DefCLL}
Formulae in classical linear logic are the formal expressions defined by the grammar (of the \emph{Backus–Naur form} \cite{backus1959syntax})
\begin{equation*}
A, B \colonequals \top \mid \bot \mid 1 \mid 0 \mid A \otimes B \mid A \invamp B \mid A \mathbin{\&} B \mid A \oplus B \mid A^\bot \mid \oc A \mid \wn A,
\end{equation*}
where $A \multimap B \colonequals \neg A \invamp B$, and we call $\top$ \emph{\bfseries top}, $\bot$ \emph{\bfseries bottom}, $1$ \emph{\bfseries one}, $0$ \emph{\bfseries zero}, $\otimes$ \emph{\bfseries tensor}, $\invamp$ \emph{\bfseries par}, $\&$ \emph{\bfseries with}, $\oplus$ \emph{\bfseries plus}, $(\_)^\bot$ \emph{\bfseries linear negation}, $\oc$ \emph{\bfseries of-course}, $\wn$ \emph{\bfseries why-not}, and $\multimap$ \emph{\bfseries linear implication}.
\end{definition}

\begin{remark}
We do not include propositional variables as (atomic) formulae for the following reason. 
In the literature, fully complete game semantics of (a fragment of) linear logic with propositional variables was achieved \cite{abramsky1994games}, and we may certainly adapt that method. 
However, the adaptation significantly digresses our main contributions. 
For the same reason, other authors \cite{laurent2004polarized,clairambault2021tale} also exclude propositional variables from their fully complete semantics of linear logic.
\end{remark}

\begin{definition}[formulae in intuitionistic linear logic \cite{girard1987lazy}]
\label{DefILL}
Formulae in intuitionistic linear logic are the class of formulae in classical linear logic defined by the grammar
\begin{equation*}
A, B \colonequals \top \mid 1 \mid A \otimes B \mid A \mathbin{\&} B \mid A \multimap B \mid \oc A \mid \neg A,
\end{equation*}
where $\bot \colonequals \neg \top$, and $\neg$ is called \emph{\bfseries tensorial negation} (adopted from \cite{mellies2010resource}).
\end{definition}

\begin{remark}
We do not include zero or plus in intuitionistic linear logic as in \cite[\S 3.2]{mellies2009categorical}, while some authors \cite{girard1987lazy,troelstra1991lectures,abramsky1993computational} do.
The reason is that the \emph{sequential} computation in the bicategory $\mathcal{LG}$, which interprets intuitionistic linear logic, does not have coproducts.
\end{remark}

\begin{definition}[formulae in classical and intuitionistic logics \cite{frege1879begriffsschrift,heyting1930formalen,troelstra2000basic}]
\label{DefCL}
Formulae in classical and intuitionistic logics are both the formal expressions defined by the grammar
\begin{mathpar}
A, B \colonequals \mathrm{tt} \mid \mathrm{ff} \mid A \wedge B \mid A \vee B \mid A \Rightarrow B,
\end{mathpar}
where ${\sim} A \colonequals A \Rightarrow \mathrm{ff}$, and we call $\mathrm{tt}$ \emph{\bfseries truth}, $\mathrm{ff}$ \emph{\bfseries falsity}, $\wedge$ \emph{\bfseries conjunction}, $\vee$ \emph{\bfseries disjunction}, $\Rightarrow$ \emph{\bfseries implication}, and $\sim$ \emph{\bfseries negation}. 
\end{definition}

Let us next review the provability of these logics in terms of the \emph{sigma-calculus} [Yam23]:
\begin{definition}[the sigma-calculus {[Yam23]}]
\label{DefSigmaCalculus}
The \emph{\bfseries sigma-calculus} consists of the following:
\begin{itemize}

\item \textsc{(Types)} A \emph{\bfseries type} is a formula in intuitionistic linear logic equipped with positive integers $i$ in the form of the subscripts $(\_)_{[i]}$ on tensorial negation, called an \emph{\bfseries identifier}, defined by
\begin{equation*}
S, T \colonequals \top \mid 1 \mid S \otimes T \mid S \mathbin{\&} T \mid S \multimap T \mid \oc S \mid \neg_{[i]} S
\quad
(i \in \mathbb{N}_+)
\end{equation*}
and $\bot_{[i]} \colonequals \neg_{[i]} \top$ (the identifier $(\_)_{[0]}$ is reserved for the \emph{empty codomain} defined below).

\begin{notation}
We write $\underline{S}$ for the formula obtained from a type $S$ by deleting all identifiers.
This operation extends to finite sequences of types in the componentwise fashion.
We write $\iota$, $\jmath$, etc. for an assignment of identifiers that transforms a formula into a type; let $T(\iota)$ be a type $T$ with such an assignment $\iota$ explicit, and $T\{ \iota' / \iota \} \colonequals T(\iota')$.
Abusing notation, let $T\{ i'/i \}$ be the type obtained from $T$ by renaming the identifier $(\_)_{[i]}$ with the one $(\_)_{[i']}$.
\end{notation}

\item \textsc{(Particles)} 
An \emph{\bfseries atom} is a formal expression $[V]o \mapsto p$ such that $o, p \in \mathbb{N}$ and $V$ is a finite set of natural numbers. 
A \emph{\bfseries particle}, written $\mathcal{P}$, $\mathcal{A}$, etc., is a finite set of atoms.

\begin{notation}
From an atom $a = ([V]o \mapsto p)$ and a number $i \in \mathbb{N}_+$, we obtain the atom $a\{ i/0 \}$ by replacing $0$ with $i$ occurring in $a$.
This operation extends to a particle $\mathcal{P}$ by $\mathcal{P}\{ i/0 \} \colonequals \{ \, a\{ i/0 \} \mid a \in \mathcal{P} \, \}$, and to a finite sequence $\mathcal{P}_1\mathcal{P}_2 \dots \mathcal{P}_n$ of particles by $(\mathcal{P}_1\mathcal{P}_2 \dots \mathcal{P}_n)\{ i/0 \} \colonequals \mathcal{P}_1\{ i/0 \}\mathcal{P}_2\{ i/0 \} \dots \mathcal{P}_n\{ i/0 \}$.
We also define another atom $a[+0i] \colonequals ([V \cup \{ 0, i \}]o \mapsto p)$, and similarly this operation extends to a particle and to a finite sequence of particles. 
\end{notation}

\item \textsc{(Contexts)} A \emph{\bfseries context} is a finite sequence $(\mathcal{P}_i : T_i)_{i \in \overline{n}}$ of pairs, written $\mathcal{P}_i : T_i$, of a particle $\mathcal{P}_i$ and a type $T_i$, where the context is written $\gamma : \Gamma$ if $\gamma = (\mathcal{P}_i)_{i \in \overline{n}}$ and $\Gamma = (T_i)_{i \in \overline{n}}$.


\item \textsc{(Cuts)} A \emph{\bfseries cut} is a formal expression defined by
\begin{equation*}
C, D \colonequals \top \mid \mathrm{Cut}(T \mid T') \mid C \otimes D \mid C \mathbin{\&} D \mid \oc C,
\end{equation*}
where $T$ and $T'$ are types that satisfy the equation $\underline{T} = \underline{T'}$.
An \emph{\bfseries intensionality} is a finite sequence $(\mathcal{M}_i : C_i)_{i \in \overline{n}}$ of pairs, written $\mathcal{M}_i : C_i$, of a particle $\mathcal{M}_i$ and a cut $C_i$, where the intensionality is written $\pi : \Pi$ if $\pi = (\mathcal{M}_i)_{i \in \overline{n}}$ and $\Pi = (C_i)_{i \in \overline{n}}$.

\item \textsc{(Sigma-sequents)}
A \emph{\bfseries sigma-sequent} is a formal expression $\Gamma \dashv \Pi \vdash \Phi$ such that $\Gamma$, called the \emph{\bfseries domain}, is a finite sequence of types, $\Pi$ is that of cuts, and $\Phi$, called the \emph{\bfseries codomain}, is that of types of length at most $1$.
We identify the empty codomain with the symbol $\bot_{[0]}$.

\item \textsc{(Raw-terms)} A \emph{\bfseries raw-term} is a formal expression $\gamma : \Gamma \dashv \pi : \Pi \vdash \phi : \Phi$, identified up to renaming of identifiers (similarly to the \emph{variable convention} in the lambda-calculus \cite{barendregt1984lambda}), such that $\gamma : \Gamma$ is a context, $\pi : \Pi$ is an intensionality, and $\phi : \Phi$ is a context of length at most $1$, and the numbers occurring as identifiers in the raw-term are pairwise distinct.

\begin{notation}
If $t = (\gamma : \Gamma \dashv \pi : \Pi \vdash \phi: \Phi)$ is a raw-term, then $\mathrm{Sq}(t) \colonequals \Gamma \dashv \Pi \vdash \Phi$.
Clearly, $t$ is recovered from $\mathrm{Sq}(t)$ and the set $\mathrm{atm}(t)$ of all atoms in $t$.
The \emph{union} $t \cup u$ of raw-terms $t$ and $u$ such that $\mathrm{Sq}(t) = \mathrm{Sq}(u)$ is the raw-term on $\mathrm{Sq}(t)$ corresponding to $\mathrm{atm}(t) \cup \mathrm{atm}(u)$.
In raw-terms, $e_1, e_2, \dots, e_n$ denotes the set $\{ e_i \}_{i \in \overline{n}}$, and $s_1 ; s_2 ; \dots ; s_m$ the sequence $(s_j)_{j \in \overline{m}}$. 
\end{notation}

\item \textsc{(Sigma-terms)} A \emph{\bfseries sigma-term} is a raw-term derivable by the rules displayed in Figure~\ref{FigLinearSigmaCalculus}, where identifiers are renamed, if necessary, in such a way that they are always \emph{fresh}.

\begin{notation}
If $t$ is a sigma-sequent on a sigma-sequent $F$, then we write $t :: F$.
\end{notation}

\end{itemize}
\begin{figure}
\begin{center}
\begin{footnotesize}
\begin{mathpar}
\AxiomC{$\gamma : \Gamma; \mathcal{P} : S; \mathcal{P}' : S'; \delta : \Delta \dashv \pi : \Pi \vdash \phi : \Phi$}
\LeftLabel{\textsc{(XL)}}
\UnaryInfC{$\gamma : \Gamma; \mathcal{P}' : S'; \mathcal{P} : S; \delta : \Delta \dashv \pi : \Pi \vdash \phi : \Phi$}
\DisplayProof 
\and
\AxiomC{$\gamma : \Gamma \dashv \pi : \Pi \vdash \phi : \Phi$}
\LeftLabel{\textsc{($\oc$W)}}
\UnaryInfC{$\gamma : \Gamma; \emptyset : \oc S \dashv \pi : \Pi \vdash \phi : \Phi$}
\DisplayProof 
\and
\AxiomC{$\gamma : \Gamma; \mathcal{P} : \oc S(\iota); \mathcal{P}' :  \oc S(\iota') \dashv \pi : \Pi \vdash \phi : \Phi$}
\LeftLabel{\textsc{($\oc$C)}}
\UnaryInfC{$\gamma : \Gamma; \mathcal{P} \cup \mathcal{P}'\{ \iota / \iota' \} : \oc S(\iota) \dashv \pi : \Pi \vdash \phi : \Phi$}
\DisplayProof 
\and
\AxiomC{$\gamma : \Gamma; \mathcal{P} : S \dashv \pi : \Pi \vdash \phi : \Phi$}
\LeftLabel{\textsc{($\oc$D)}}
\UnaryInfC{$\gamma : \Gamma; \mathcal{P} : \oc S \dashv \pi: \Pi \vdash \phi : \Phi$}
\DisplayProof 
\and
\AxiomC{$\delta : \oc \Delta \dashv \pi : \Pi \vdash \mathcal{A} : S$}
\LeftLabel{\textsc{($\oc$R)}}
\UnaryInfC{$\delta : \oc \Delta \dashv \pi : \oc \Pi \vdash \mathcal{A} : \oc S$}
\DisplayProof 
\and
\AxiomC{$\gamma : \Gamma \dashv \pi : \Pi \vdash \mathcal{A} : S(\iota)$}
	\AxiomC{$\delta : \Delta; \mathcal{P} : S(\jmath) \dashv \sigma : \Sigma \vdash \phi : \Phi$}
	\LeftLabel{\textsc{(Cut)}}
\BinaryInfC{$\gamma : \Gamma; \delta : \Delta \dashv \pi : \Pi; \mathcal{A} \cup \mathcal{P} : \mathrm{Cut}(S(\iota) \mid S(\jmath)); \sigma : \Sigma \vdash \phi : \Phi$}
\DisplayProof
\\
\AxiomC{$\gamma : \Gamma \dashv \pi : \Pi \vdash \phi : \Phi$}
\LeftLabel{\textsc{($\top$L)}}
\UnaryInfC{$\gamma : \Gamma; \emptyset : \top \dashv \pi : \Pi \vdash \phi : \Phi$}
\DisplayProof \and
\AxiomC{}
\LeftLabel{\textsc{($\top$R)}}
\UnaryInfC{$\dashv \vdash \emptyset : \top$}
\DisplayProof 
\and
\AxiomC{}
\LeftLabel{\textsc{($1$R)}}
\UnaryInfC{$\Gamma \dashv \vdash \emptyset : 1$}
\DisplayProof 
\\
\AxiomC{$\gamma : \Gamma; \mathcal{P} : S; \mathcal{Q} : T \dashv \pi : \Pi \vdash \phi : \Phi$}
\LeftLabel{\textsc{($\otimes$L)}}
\UnaryInfC{$\gamma : \Gamma; \mathcal{P} \cup \mathcal{Q} : S \otimes T \dashv \pi : \Pi \vdash \phi : \Phi$}
\DisplayProof
\and
\AxiomC{$\gamma : \Gamma \dashv \pi : \Pi \vdash \mathcal{A} : S$}
		\AxiomC{$\delta : \Delta \dashv \sigma : \Sigma \vdash \mathcal{B} : T$}
	\LeftLabel{\textsc{($\otimes$R)}}
\BinaryInfC{$\gamma : \Gamma; \delta : \Delta \dashv \pi : \Pi; \sigma : \Sigma \vdash \mathcal{A} \cup \mathcal{B} : S \otimes T$} 
\DisplayProof 
\\
\AxiomC{$\gamma : \Gamma; \mathcal{P} : S \dashv \pi : \Pi \vdash \phi : \Phi$}
\LeftLabel{\textsc{($\&$L)}}
\UnaryInfC{$\gamma : \Gamma; \mathcal{P} : S \mathbin{\&} T \dashv \pi : \Pi \vdash \phi : \Phi$}
\DisplayProof
\and
\AxiomC{$\gamma : \Gamma; \mathcal{Q} : T \dashv \pi : \Pi \vdash \phi : \Phi$}
\LeftLabel{\textsc{($\&$L)}}
\UnaryInfC{$\gamma : \Gamma; \mathcal{Q} : S \mathbin{\&} T \dashv \pi : \Pi \vdash \phi : \Phi$}
\DisplayProof
\and
\AxiomC{$\gamma : \Gamma \dashv \pi : \Pi \vdash \mathcal{A} : S$}
\AxiomC{$\gamma' : \Gamma \dashv \sigma : \Sigma \vdash \mathcal{B} : T$}
	\LeftLabel{\textsc{($\&$R)}}
	\RightLabel{($|\Pi| = 1 = |\Sigma|$)}
\BinaryInfC{$\gamma \cup \gamma' : \Gamma \dashv \pi \cup \sigma : \Pi \mathbin{\&} \Sigma \vdash \mathcal{A} \cup \mathcal{B} : S \mathbin{\&} T$} 
\DisplayProof
\and
\AxiomC{$\gamma : \Gamma \dashv \pi : \Pi \vdash \mathcal{A} : S$}
\LeftLabel{\textsc{($\neg$L)}}
\UnaryInfC{$\gamma[+0i] : \Gamma; \mathcal{A}[+0i] \cup \{ [\emptyset]0 \mapsto i \} : \neg_{[i]} S \dashv \pi[+0i] : \Pi \vdash$}
\DisplayProof 
\and
\AxiomC{$\gamma : \Gamma; \mathcal{P} : S \dashv \pi : \Pi \vdash$}
\LeftLabel{\textsc{($\neg$R)}}
\UnaryInfC{$\gamma\{ i/0 \} : \Gamma \dashv \pi\{ i/0 \} : \Pi \vdash \mathcal{P}\{ i/0 \} : \neg_{[i]} S$}
\DisplayProof
\\
\AxiomC{$\gamma : \Gamma \dashv \pi : \Pi \vdash \mathcal{A} : S$}
\AxiomC{$\delta : \Delta; \mathcal{Q} : T \dashv \sigma : \Sigma \vdash \phi : \Phi$}
\LeftLabel{\textsc{($\multimap$L)}}
\BinaryInfC{$\gamma : \Gamma; \delta: \Delta; \mathcal{A} \cup \mathcal{Q} : S \multimap T \dashv \pi : \Pi; \sigma : \Sigma \vdash \phi : \Phi$}
\DisplayProof 
\and
\AxiomC{$\gamma : \Gamma; \mathcal{P} : S \dashv \pi : \Pi \vdash \mathcal{B} : T$}
\LeftLabel{\textsc{($\multimap$R)}}
\UnaryInfC{$\gamma : \Gamma \dashv \pi : \Pi \vdash \mathcal{P} \cup \mathcal{B} : S \multimap T$}
\DisplayProof
\end{mathpar}
\end{footnotesize}
\end{center}
\caption{Typing rules of the sigma-calculus}
\label{FigLinearSigmaCalculus}
\end{figure}
\end{definition}

\begin{example}
\label{FirstExampleOfSigmaTerm}
Figure~\ref{FigFirstExampleOfSigmaTerm} displays a derivation of a sigma-term.
\begin{figure}[h]
\begin{center}
\begin{footnotesize}
\begin{mathpar}
\AxiomC{}
\LeftLabel{\textsc{($\top$R)}}
\UnaryInfC{$\dashv \vdash \emptyset : \top$}
\LeftLabel{\textsc{($\top$L)}}
\UnaryInfC{$\emptyset : \top \dashv \vdash \emptyset : \top$}
\LeftLabel{\textsc{($\neg$L)}}
\UnaryInfC{$\emptyset : \top; [\emptyset]0 \mapsto p : \bot_{[p]} \dashv \vdash$}
\LeftLabel{\textsc{(XL)}}
\UnaryInfC{$[\emptyset]0 \mapsto p : \bot_{[p]}; \emptyset : \top \dashv \vdash$}
\LeftLabel{\textsc{($\neg$R)}}
\UnaryInfC{$[\emptyset]o \mapsto p : \bot_{[p]} \dashv \vdash \emptyset : \bot_{[o]}$}
\LeftLabel{\textsc{($\oc$D)}}
\UnaryInfC{$[\emptyset]o \mapsto p : \oc  \bot_{[p]} \dashv \vdash \emptyset : \bot_{[o]}$}
\LeftLabel{\textsc{($\oc$R)}}
\UnaryInfC{$[\emptyset]o \mapsto p : \oc  \bot_{[p]} \dashv \vdash \emptyset : \oc \bot_{[o]}$}
\AxiomC{}
\LeftLabel{\textsc{($\top$R)}}
\UnaryInfC{$\dashv \vdash \emptyset : \top$}
\LeftLabel{\textsc{($\top$L)}}
\UnaryInfC{$\emptyset : \top \dashv \vdash \emptyset : \top$}
\LeftLabel{\textsc{($\neg$L)}}
\UnaryInfC{$\emptyset : \top; [\emptyset]0 \mapsto q : \bot_{[q]} \dashv \vdash$}
\LeftLabel{\textsc{(XL)}}
\UnaryInfC{$[\emptyset]0 \mapsto q : \bot_{[q]}; \emptyset : \top \dashv \vdash$}
\LeftLabel{\textsc{($\neg$R)}}
\UnaryInfC{$[\emptyset]i \mapsto q : \bot_{[q]} \dashv \vdash \emptyset : \bot_{[i]}$}
\LeftLabel{\textsc{($\oc$D)}}
\UnaryInfC{$[\emptyset]i \mapsto q :\oc  \bot_{[q]} \dashv \vdash \emptyset : \bot_{[i]}$}
\AxiomC{}
\LeftLabel{\textsc{($\top$R)}}
\UnaryInfC{$\dashv \vdash \emptyset : \top$}
\LeftLabel{\textsc{($\top$L)}}
\UnaryInfC{$\emptyset : \top \dashv \vdash \emptyset : \top$}
\LeftLabel{\textsc{($\neg$L)}}
\UnaryInfC{$\emptyset : \top; [\emptyset]0 \mapsto q' : \bot_{[q']} \dashv \vdash$}
\LeftLabel{\textsc{(XL)}}
\UnaryInfC{$[\emptyset]0 \mapsto q' : \bot_{[q']}; \emptyset : \top \dashv \vdash$}
\LeftLabel{\textsc{($\neg$R)}}
\UnaryInfC{$[\emptyset]j \mapsto q' : \bot_{[q']} \dashv \vdash \emptyset : \bot_{[j]}$}
\LeftLabel{\textsc{($\oc$D)}}
\UnaryInfC{$[\emptyset]j \mapsto q' : \oc \bot_{[q']} \dashv \vdash \emptyset : \bot_{[j]}$}
\LeftLabel{\textsc{($\otimes$R)}}
\BinaryInfC{$[\emptyset]i \mapsto q : \oc  \bot_{[q]}; [\emptyset]j \mapsto q' : \oc \bot_{[q']} \dashv \vdash \emptyset : \bot_{[i]} \otimes \bot_{[j]}$}
\LeftLabel{\textsc{($\oc$C)}}
\UnaryInfC{$[\emptyset]i \mapsto q, [\emptyset]j \mapsto q : \oc \bot_{[q]} \dashv \vdash \emptyset : \bot_{[i]} \otimes \bot_{[j]}$}
\BinaryInfC{$[\emptyset]o \mapsto p : \oc  \bot_{[p]} \dashv \mathrm{Cut}(\emptyset : \bot_{[o]} \mid [\emptyset]i \mapsto q, [\emptyset]j \mapsto q : \oc \bot_{[q]}) \vdash \emptyset : \bot_{[i]} \otimes \bot_{[j]}$}
\DisplayProof
\end{mathpar}
\end{footnotesize}
\end{center}
\caption{An example of a sigma-term}
\label{FigFirstExampleOfSigmaTerm}
\end{figure} 
\end{example}

It is immediate to see that the sigma-calculus reformulates intuitionistic linear logic because the typing rules of the former corresponds precisely to the logical rules of the latter: 

\begin{notation}
Let $\lbag x \rbag$ be a finite sequence ranging over the empty one $\boldsymbol{\epsilon}$ or the singleton one $x$.
\end{notation}

\begin{proposition}[a term calculus for intuitionistic linear logic {[Yam23]}]
\label{PropCHIsForILL}
A sequent $\underline{\Gamma} \vdash \lbag \underline{T} \rbag$ is derivable in the sequent calculus for intuitionistic linear logic if and only if there is a sigma-term $\gamma : \Gamma \dashv \pi : \Pi \vdash \lbag \mathcal{B} : T \rbag$, where the sequent is derivable without cut if and only if $\Pi$ is empty.
\end{proposition}

Let us next explain the intuition behind the sigma-calculus.
This intuition is based on game semantics yet originally informal [Yam23], but it can be made precise by our combinatorics: A sigma-sequent represents a combinatorial sequent, and a sigma-term a P-view algorithm.

\begin{corollary}[a surjection between formulae and combinatorial arenas]
\label{CorCombinatorialArenasAndFormulaeInILL}
The assignment of a combinatorial arena $\llbracket A \rrbracket_{\mathcal{LG}}$ to each formula $A$ in intuitionistic linear logic defined by
\begin{mathpar}
\llbracket \top \rrbracket_{\mathcal{LG}} \colonequals \top
\and
\llbracket 1 \rrbracket_{\mathcal{LG}} \colonequals 1
\and
\llbracket A \otimes B \rrbracket_{\mathcal{LG}} \colonequals \llbracket A \rrbracket_{\mathcal{LG}} \otimes \llbracket B \rrbracket_{\mathcal{LG}}
\and
\llbracket A \mathbin{\&} B \rrbracket_{\mathcal{LG}} \colonequals \llbracket A \rrbracket_{\mathcal{LG}} \mathbin{\&} \llbracket B \rrbracket_{\mathcal{LG}}
\and
\llbracket A \multimap B \rrbracket_{\mathcal{LG}} \colonequals \llbracket A \rrbracket_{\mathcal{LG}} \multimap \llbracket B \rrbracket_{\mathcal{LG}}
\and
\llbracket \oc A \rrbracket_{\mathcal{LG}} \colonequals \oc \llbracket A \rrbracket_{\mathcal{LG}}
\and
\llbracket \neg A \rrbracket_{\mathcal{LG}} \colonequals \neg \llbracket A \rrbracket_{\mathcal{LG}}
\end{mathpar}
forms a surjection.
\end{corollary}
\begin{proof}
Immediate from Theorem~\ref{ThmFreeCharacterisation}.
\end{proof}

\begin{remark}
This surjection $\llbracket \_ \rrbracket_{\mathcal{LG}}$ fails to be injective on the nose: $\llbracket \top \otimes \top \rrbracket_{\mathcal{LG}} = \top = \llbracket \top \rrbracket_{\mathcal{LG}}$ yet $\top \otimes \top \neq \top$.
As we shall see, however, it is \emph{essentially injective} (Corollary). 
\end{remark}

\begin{definition}[sigma-sequents as combinatorial sequents]
The map $\llbracket \_ \rrbracket_{\mathcal{LG}}$ of Corollary~\ref{CorCombinatorialArenasAndFormulaeInILL} extends to a type $T$ by $\llbracket T \rrbracket_{\mathcal{LG}} \colonequals \llbracket \underline{T} \rrbracket_{\mathcal{LG}}$, to a cut $C$ by $\llbracket \mathrm{Cut}(T \mid T') \rrbracket_{\mathcal{LG}} \colonequals \llbracket T \rrbracket_{\mathcal{LG}} \multimap \llbracket T' \rrbracket_{\mathcal{LG}}$, and to a sigma-sequent $F = (S_1, S_2, \dots, S_n \dashv C_1, C_2, \dots, C_m \vdash T)$ by
\begin{equation*}
\llbracket F \rrbracket_{\mathcal{LG}} \colonequals \llbracket S_1 \rrbracket_{\mathcal{LG}}, \llbracket S_2 \rrbracket_{\mathcal{LG}}, \dots, \llbracket S_n \rrbracket_{\mathcal{LG}} \dashv  \llbracket C_1 \rrbracket_{\mathcal{LG}}, \llbracket C_2 \rrbracket_{\mathcal{LG}}, \dots, \llbracket C_m \rrbracket_{\mathcal{LG}} \vdash \llbracket T \rrbracket_{\mathcal{LG}},
\end{equation*}
where the number $i$ of each identifier $(\_)_{[i]}$ in $F$ is inherited to the corresponding vertex in $\llbracket F \rrbracket_{\mathcal{LG}}$, and called an \emph{\bfseries O-numeral} if it is contained in an O-move, and a \emph{\bfseries P-numeral} otherwise. 
\end{definition} 

\begin{definition}[raw-terms as P-view algorithms]
\label{DefDirectReading}
Given a raw-term $t$ on a sigma-sequent $F$, the P-view algorithm $\mathrm{alg}(t)$ on the combinatorial sequent $\llbracket F \rrbracket_{\mathcal{LG}}$ is defined by
\begin{equation*}
\mathrm{alg}(t) \colonequals \{ \, (M^F_V, m^F_o, m^F_p) \mid \text{$[V]o \mapsto p$ is an atom occurring in $t$} \, \},
\end{equation*}
where $m^F_i$ for each number $i$ is the move in $\llbracket F \rrbracket_{\mathcal{LG}}$ that contains $i$, and $M^F_V \colonequals \{ \, m^F_i \mid i \in V \, \}$.
\end{definition}

\begin{proposition}[sigma-terms as finite presentations]
\label{PropSigmaTermsAsFinitePresentations}
For every sigma-term $t :: F$, the pair $(\llbracket F \rrbracket_{\mathcal{LG}}, \llbracket t \rrbracket_{\mathcal{LG}})$, where $\llbracket t \rrbracket_{\mathcal{LG}} \colonequals \mathrm{st}(\mathrm{alg}(t))$, is a well-defined 1-cell in the bicategory $\mathcal{LG}$, and this mapping is injective, i.e., $\llbracket t \rrbracket_{\mathcal{LG}} = \llbracket t' \rrbracket_{\mathcal{LG}}$ implies $t = t'$ for all sigma-terms $t' :: F$.
\end{proposition}
\begin{proof}
By induction on a derivation of a sigma-term. 
\end{proof}

\begin{example}
\label{ExampleOfInterpretation}
If we write $t$ for the sigma-term derived in Example~\ref{FigFirstExampleOfSigmaTerm}, and $\mathscr{S} = (\underline{\mathscr{S}}, \mathrm{Int}_{\mathscr{S}})$, where $\underline{\mathscr{S}} = \oc  \bot_{[p]} \dashv_{[l]} \bot_{[o]} \multimap_{[m]} \oc \bot_{[q]}) \vdash_{[r]} \bot_{[i]} \otimes \bot_{[j]}$ and $\mathrm{Int}_{\mathscr{S}} = \{ m \}$, for the underlying combinatorial sequent with additional tags $l$, $m$ and $r$ attached to the switching vertices, then $\mathrm{alg}(t) = \{ (\emptyset, lri, mq), (\emptyset, lrj, mq), (\emptyset, o, p) \} :: \mathscr{S}$ and $\mathrm{st}(\mathrm{alg}(t)) = \mathrm{Pref}(\{ lri . mq . o . p \} \cup \{ lrj . mq . o . p \}) : \mathscr{S}$.
\end{example}

In this way, each sigma-term $t :: F$ is read off \emph{directly} and \emph{non-inductively} as a finite presentation $\mathrm{alg}(t)$ of the 1-cell $\llbracket t \rrbracket_{\mathcal{LG}} : \llbracket F \rrbracket_{\mathcal{LG}}$, which is said to be \emph{\bfseries definable} by the sigma-calculus.

Let us next recall the \emph{cut-elimination} for the sigma-calculus [Yam23], which deletes cuts in a sigma-term in a step-by-step fashion.
As the name indicates, this procedure corresponds under Proposition~\ref{PropCHIsForILL} to the cut-elimination procedure for intuitionistic linear logic \cite{gentzen1936widerspruchsfreiheit,troelstra2000basic}.

\begin{notation}
Let $\mathcal{P}$ be a particle, $\iota$ an assignment of identifiers to a type, and $t$ a raw-term.
\begin{itemize}

\item $\mathcal{P}{\upharpoonright_\iota} \colonequals \{ \, a \in \mathcal{P} \mid \mathbb{N}(a) \cap \mathbb{N}(\iota) \neq \emptyset \, \}$ and $\mathcal{P}{\downharpoonright_\iota} \colonequals \mathcal{P} \setminus \mathcal{P}{\upharpoonright_\iota}$, where $\mathbb{N}(a)$ and $\mathbb{N}(\iota)$ are the sets of all numbers occurring in $a$ and $\iota$, respectively.
These operations extend to $t$ via $\mathrm{atm}(t)$.

\item A pair $\mathcal{M} : \mathrm{Cut}(T(\iota) \mid T(\jmath))$ of a particle $\mathcal{M}$ and a cut $\mathrm{Cut}(T(\iota) \mid T(\jmath))$, called a \emph{\bfseries cut-pair}, is also written $\mathrm{Cut}(\mathcal{M}{\upharpoonright_\iota} : T(\iota) \mid \mathcal{M}{\upharpoonright_\jmath} : T(\jmath))$ for convenience. 

\item Given finite sequences $\boldsymbol{c}$ and $\boldsymbol{c'}$ of cut-pairs with $\boldsymbol{c}$ occurring in $t$, we write $t[\boldsymbol{c}]$ for $t$ with the \emph{hole} $[\_]$ assigned to $\boldsymbol{c}$, and $t[\boldsymbol{c}']$ for what is obtained from $t$ by replacing $\boldsymbol{c}$ with $\boldsymbol{c}'$. 

\item We write $t\{ \iota'/\iota \}$ for the raw-term obtained from $t$ by replacing the assignment $\iota$ of identifiers to the one $\iota'$, and $t\{ i'/i \}$ for $t\{ \iota'/\iota \}$ if $\iota$ and $\iota$ are single assignments of $i$ and $i'$, respectively. 

\item For a cut-pair $\mathrm{Cut}(\mathcal{L} : T(\iota) \mid \mathcal{R} : T(\jmath))$, we write $\#\mathcal{R} = 1$ if, for each P-numeral $p$ occurring in $\jmath$, $\mathcal{R}$ contains a unique atom of the form $[V]o \mapsto p$ yet no other atoms. 

\end{itemize}
\end{notation}

\begin{definition}[cut-elimination {[Yam23]}]
The \emph{\bfseries cut-elimination} for the sigma-calculus is the union ${\rightarrow} \colonequals \bigcup_{i = 1}^{7} \rightarrow_i$ of the binary relations $\rightarrow_i$ on raw-terms listed in Figure~\ref{FigLinearBetaReduction}, and the \emph{\bfseries big-step cut-elimination} $\rightarrow^\omega$ is the reflexive, transitive closure of $\rightarrow$.
\begin{figure}
\begin{center}
\begin{footnotesize}
\begin{gather*}
t[\mathrm{Cut}(\mathcal{L} : S(\iota) \mid \emptyset : S(\tilde{\iota}))] \rightarrow_1 t[\boldsymbol{\epsilon}]{\downharpoonright_{\iota}} \\
t[\mathrm{Cut}(\mathcal{L} : \neg_{[i]} S(\iota) \mid \mathcal{R}, [J]o \mapsto j : \neg_{[j]}S(\jmath))] \rightarrow_2 t[\mathrm{Cut}(\mathcal{R} : S(\jmath) \mid \mathcal{L} : S(\iota))]\{ o / i \} \\
t[\mathrm{Cut}(\mathcal{L} : S(\iota) \otimes T(\jmath) \mid \mathcal{R} : S(\tilde{\iota}) \otimes T(\tilde{\jmath}))] \rightarrow_3 t[\mathrm{Cut}(\mathcal{L}{\upharpoonright_{\iota}} : S(\iota) \mid \mathcal{R}{\upharpoonright_{\tilde{\iota}}} : S(\tilde{\iota})); \mathrm{Cut}(\mathcal{L}{\upharpoonright_{\jmath}} : T(\jmath) \mid \mathcal{R}{\upharpoonright_{\tilde{\jmath}}} : T(\tilde{\jmath}))] \\
t[\mathrm{Cut}(\mathcal{L} : S(\iota) \mathbin{\&} T(\jmath) \mid \mathcal{R} : S(\tilde{\iota}) \mathbin{\&} T(\tilde{\jmath}))] \rightarrow_4 t[\mathrm{Cut}(\mathcal{L}{\upharpoonright_{\iota}} : S(\iota) \mid \mathcal{R}{\upharpoonright_{\tilde{\iota}}} : S(\tilde{\iota})); \mathrm{Cut}(\mathcal{L}{\upharpoonright_{\jmath}} : T(\jmath) \mid \mathcal{R}{\upharpoonright_{\tilde{\jmath}}} : T(\tilde{\jmath}))] \\
t[\mathrm{Cut}(\mathcal{L} : S(\iota) \multimap T(\jmath) \mid \mathcal{R} : S(\tilde{\iota}) \multimap T(\tilde{\jmath}))] \rightarrow_5 t[\mathrm{Cut}(\mathcal{R}{\upharpoonright_{\tilde{\iota}}} : S(\tilde{\iota}) \mid \mathcal{L}{\upharpoonright_{\iota}} : S(\iota)); \mathrm{Cut}(\mathcal{L}{\upharpoonright_{\jmath}} : T(\jmath) \mid \mathcal{R}{\upharpoonright_{\tilde{\jmath}}} : T(\tilde{\jmath}))] \\
t[\mathrm{Cut}(\mathcal{L} : \oc S(\iota) \mid \mathcal{R} : \oc S(\jmath))] \rightarrow_{6} t[\mathrm{Cut}(\mathcal{L} : S(\iota) \mid \mathcal{R} : S(\jmath))] \quad (\text{if $\# \mathcal{R} = 1$}) \\
t[\mathrm{Cut}(\mathcal{L} : \oc S(\iota) \mid \mathcal{R}, \mathcal{Q} : \oc S(\jmath))] \rightarrow_{7} t[\mathrm{Cut}(\mathcal{L} : \oc S(\iota) \mid \mathcal{R} : \oc S(\jmath)); \mathrm{Cut}(\mathcal{L}\{ \iota'/\iota \} : \oc S(\iota') \mid \mathcal{Q}\{ \jmath' / \jmath \} : \oc S(\jmath'))] \cup t{\upharpoonright_{\iota, \jmath}}\{ \iota', \jmath' /\iota, \jmath \} \\ (\text{if $\# \mathcal{R} = 1$ and $\mathcal{Q} \neq \emptyset$})
\end{gather*}
\end{footnotesize}
\end{center}
\caption{Cut-elimination for the sigma-calculus}
\label{FigLinearBetaReduction}
\end{figure}
\end{definition}

\begin{example}
The sigma-term of Example~\ref{FirstExampleOfSigmaTerm} computes by the cut-elimination steps
\begin{footnotesize}
\begin{gather*}
[\emptyset]o \mapsto p : \oc \bot_{[p]} \dashv \mathrm{Cut}(\emptyset : \oc \bot_{[o]} \mid [\emptyset]i \mapsto q, [\emptyset]j \mapsto q : \oc \bot_{[q]}) \vdash \emptyset : \bot_{[i]} \otimes \bot_{[j]} \\
\downarrow \\
[\emptyset]o \mapsto p, [\emptyset]o' \mapsto p : \oc \bot_{[p]} \dashv \mathrm{Cut}(\emptyset : \bot_{[o]} \mid [\emptyset]i \mapsto q : \bot_{[q]}); \mathrm{Cut}(\emptyset : \bot_{[o']} \mid [\emptyset]j \mapsto q' : \bot_{[q']}) \vdash \emptyset : \bot_{[i]} \otimes \bot_{[j]} \\
\downarrow_\ast \\
[\emptyset]i \mapsto p, [\emptyset]j \mapsto p : \oc \bot_{[p]} \dashv \mathrm{Cut}(\emptyset : \top \mid \emptyset : \top); \mathrm{Cut}(\emptyset : \top \mid \emptyset : \top) \vdash \emptyset : \bot_{[i]} \otimes \bot_{[j]} \\
\downarrow_\ast \\
[\emptyset]i \mapsto p, [\emptyset]j \mapsto p : \oc \bot_{[p]} \dashv \vdash \emptyset : \bot_{[i]} \otimes \bot_{[j]}.
\end{gather*}
\end{footnotesize}
\end{example}

Yamada [Yam23] has proven basic properties of the cut-elimination such as \emph{subject reduction}, \emph{confluence} and \emph{strong normalisation}.
These properties collectively imply: 

\begin{theorem}[correctness of cut-elimination {[Yam23]}]
\label{ThmSyntacticCutElim}
If $t$ is a sigma-term on a sigma-sequent $\Gamma \dashv \Pi \vdash \Phi$, then there is a finite sequence $t \colonequals t_0 \rightarrow t_1 \rightarrow t_2 \rightarrow \dots \rightarrow t_n$ of cut-eliminations such that each $t_i$ ($0 \leqslant i \leqslant n$) is a sigma-term on a sigma-sequent $\Gamma \dashv \Pi_i \vdash \Phi$ with $t_n$ on $\Gamma \dashv \vdash \Phi$.
\end{theorem}

We next define a bicategory out of the sigma-calculus by a method similar to the one for $\mathcal{LG}$:
\begin{definition}[the sigma-calculus as a bicategory]
The E-category $\mathsf{ILL}$ is defined as follows:
\begin{itemize}

\item An object is a type;

\item A morphism $S \rightarrow T$ is a pair $(\Pi, t)$ of a finite sequence $\Pi$ of cuts and a sigma-term $t$ on the sigma-sequent $S \dashv \Pi \vdash T$, where the morphism $(\Pi, t)$ is said to be \emph{\bfseries normalised} or a \emph{\bfseries value} if $\Pi = \boldsymbol{\epsilon}$;

\item The composition of morphisms $(\Pi, t) : S \rightarrow T(\iota)$ and $(\Sigma, u) : T(\iota') \rightarrow U$ is the pair 
\begin{equation*}
(\Pi, t) ; (\Sigma, u) \colonequals ((\Pi, \mathrm{Cut}(T(\iota) \mid T(\iota')), \Sigma), t \cup u) : S \rightarrow U;
\end{equation*}

\item The identity $\mathrm{id}_{S}$ is the pair $(\boldsymbol{\epsilon}, \mathrm{id}_{S})$, where the sigma-term $\mathrm{id}_S :: (S(\jmath) \dashv \vdash S(\jmath'))$ is defined by induction on the type $S$ in the standard way (i.e., the same as the proof that the sequent $A \vdash A$ is \emph{derivable} for all formulae $A$ in the sequent calculus for intuitionistic linear logic);

\item The equivalence relation $\simeq_{S, T}$ between morphisms $(\Pi, t), (\Pi', t') : S \rightrightarrows T$ is given by
\begin{equation*}
(\Pi, t) \simeq_{S, T} (\Pi', t') \ratio\Leftrightarrow \text{$t \rightarrow^\ast t''$ and $t' \rightarrow^\ast t''$ for the same sigma-term $t''$ on $S \dashv \vdash T$,}
\end{equation*}
where we often omit the subscripts $(\_)_{S, T}$ on $\simeq$ (n.b., $\rightarrow^\ast {=} \rightarrow^\omega$ thanks to Theorem~\ref{ThmSyntacticCutElim}).
\end{itemize}
\end{definition}

It is straightforward, albeit tedious, to show that $\mathsf{ILL}$ forms a well-defined E-category.
Further, it forms a new Seely category up to $\simeq$ similarly to $\mathcal{LG}$ though we will not use this result.
We leave the details to the reader. 
We shall eventually prove that the two bicategories are \emph{biequivalent} by:

\begin{proposition}[a semantic functor]
\label{PropInterpretationFunctor}
The maps $\llbracket \_ \rrbracket_{\mathcal{LG}}$ give rise to a 2-functor (or E-functor) $\mathsf{ILL} \rightarrow \mathcal{LG}$, and it preserves the structures of new Seely categories up to 2-cells.
\end{proposition}
\begin{proof}
Immediate from the constructions of $\mathsf{ILL}$ and $\mathcal{LG}$.
\end{proof}

Next, Yamada [Yam23] employed the \emph{double negation translation} of classical linear logic into intuitionistic linear logic \cite[\S 5.12]{troelstra1991lectures} for the sigma-calculus to embody classical linear logic:
\begin{proposition}[a term calculus for classical linear logic {[Yam23]}]
\label{PropCHIsForCLL}
Let $\mathscr{T}_{\neg\neg}$ be the map from formulae in classical linear logic to those in intuitionistic linear logic defined by
\begin{mathpar}
\mathscr{T}_{\neg\neg}(\top) \colonequals \top
\and
\mathscr{T}_{\neg\neg}(1) \colonequals 1
\and
\mathscr{T}_{\neg\neg}(\bot) \colonequals \bot
\and
\mathscr{T}_{\neg\neg}(0) \colonequals \neg 1
\and
\mathscr{T}_{\neg\neg}(A^\bot) \colonequals \neg \mathscr{T}_{\neg\neg}(A)
\and
\mathscr{T}_{\neg\neg}(A \otimes B) \colonequals \neg \neg \mathscr{T}_{\neg\neg}(A) \otimes \neg \neg \mathscr{T}_{\neg\neg}(B)
\and
\mathscr{T}_{\neg\neg}(A \mathbin{\&} B) \colonequals \neg \neg \mathscr{T}_{\neg\neg}(A) \mathbin{\&} \neg \neg \mathscr{T}_{\neg\neg}(B)
\and
\mathscr{T}_{\neg\neg} (A \invamp B) \colonequals \neg (\neg \mathscr{T}_{\neg\neg} (A) \otimes \neg \mathscr{T}_{\neg\neg}(B))
\and
\mathscr{T}_{\neg\neg} (A \oplus B) \colonequals \neg (\neg \mathscr{T}_{\neg\neg} (A) \mathbin{\&} \neg \mathscr{T}_{\neg\neg}(B))
\and
\mathscr{T}_{\neg\neg}(A \multimap B) \colonequals \mathscr{T}_{\neg\neg}(A) \multimap \neg \neg \mathscr{T}_{\neg\neg}(B)
\and
\mathscr{T}_{\neg\neg}(\oc A) \colonequals \oc \neg \neg \mathscr{T}_{\neg\neg}(A) 
\and
\mathscr{T}_{\neg\neg}(\wn A) \colonequals \neg \oc \neg \mathscr{T}_{\neg\neg}(A).
\end{mathpar}
A sequent $\boldsymbol{A} \vdash \lbag B \rbag$ is derivable in the sequent calculus for classical linear logic if and only if there is a sigma-term $\gamma : \Gamma \dashv \pi : \Pi \vdash \lbag \mathcal{B} : T \rbag$ such that $\underline{\Gamma} = \mathscr{T}_{\neg\neg}^\ast(\boldsymbol{A})$ and $\underline{T} = \neg \neg \mathscr{T}_{\neg\neg}(B)$ if $\lbag B \rbag = B$ and $\lbag \mathcal{B} : T \rbag = \mathcal{B} : T$, where the sequent is derivable without cut if and only if $\Pi = \boldsymbol{\epsilon}$.
\end{proposition}

Moreover, the sigma-calculus also embodies intuitionistic logic through \emph{Girard's translation} \cite[\S 3.3]{girard1987lazy} of intuitionistic logic into intuitionistic linear logic:

\begin{proposition}[a term calculus for intuitionistic logic {[Yam23]}]
\label{PropCHIsForIL}
Let $\mathscr{T}_{\oc}$ be the map from formulae in intuitionistic logic to those in intuitionistic linear logic defined by
\begin{mathpar}
\mathscr{T}_{\oc}(\mathrm{tt}) \colonequals \top
\and
\mathscr{T}_{\oc}(\mathrm{ff}) \colonequals \bot
\and
\mathscr{T}_{\oc}(A \wedge B) \colonequals \mathscr{T}_{\oc}(A) \mathbin{\&} \mathscr{T}_{\oc}(B)
\and
\mathscr{T}_{\oc} (A \vee B) \colonequals \neg (\neg \oc \mathscr{T}_{\oc} (A) \mathbin{\&} \neg \oc \mathscr{T}_{\oc}(B))
\and
\mathscr{T}_{\oc}(A \Rightarrow B) \colonequals \oc \mathscr{T}_{\oc}(A) \multimap \mathscr{T}_\oc(B).
\end{mathpar}
A sequent $\boldsymbol{A} \vdash \lbag B \rbag$ is derivable in the sequent calculus for intuitionistic logic if and only if there is a sigma-term $\gamma : \Gamma \dashv \pi : \Pi \vdash \lbag \mathcal{B} : T \rbag$ such that $\underline{\Gamma} = \oc \mathscr{T}_{\oc}^\ast(\boldsymbol{A})$ and $\underline{\Phi} = \mathscr{T}_{\oc}(B)$ if $\lbag B \rbag = B$ and $\lbag \mathcal{B} : T \rbag = \mathcal{B} : T$, where the sequent is derivable without cut if and only if $\Pi = \boldsymbol{\epsilon}$.
\end{proposition}

Finally, the sigma-calculus also implements classical logic via \emph{T-translation} \cite{danos1995lkq,danos1997new}:

\begin{proposition}[a term calculus for classical logic {[Yam23]}]
\label{PropCHIsForCL}
Let $\mathscr{T}_{\oc\wn}$ be the map from formulae in classical logic to those in intuitionistic linear logic defined by
\begin{mathpar}
\mathscr{T}_{\oc\wn}(\mathrm{tt}) \colonequals \top
\and
\mathscr{T}_{\oc\wn}(\mathrm{ff}) \colonequals \bot
\and
\mathscr{T}_{\oc\wn}(A \wedge B) \colonequals \wn \mathscr{T}_{\oc\wn}(A) \mathbin{\&} \wn \mathscr{T}_{\oc\wn}(B)
\and
\mathscr{T}_{\oc\wn} (A \vee B) \colonequals \neg (\oc \neg \mathscr{T}_{\oc\wn} (A) \mathbin{\&} \oc \neg \mathscr{T}_{\oc\wn}(B))
\and
\mathscr{T}_{\oc\wn}(A \Rightarrow B) \colonequals \oc \wn \mathscr{T}_{\oc\wn}(A) \multimap \wn \mathscr{T}_{\oc\wn}(B).
\end{mathpar}
A sequent $\boldsymbol{A} \vdash [B]$ is derivable in the sequent calculus for classical logic if and only if there is a sigma-term $\gamma : \Gamma \dashv \pi : \Pi \vdash \lbag \mathcal{B} : T \rbag$ such that $\underline{\Gamma} = \oc \wn \mathscr{T}_{\oc}^\ast(\boldsymbol{A})$ and $\underline{T} = \wn \mathscr{T}_{\oc}(B)$ when $\lbag B \rbag = B$ and $\lbag \mathcal{B} : T \rbag = \mathcal{B} : T$, where the sequent is derivable without cut if and only if $\Pi = \boldsymbol{\epsilon}$.
\end{proposition}

We may obtain a clearer picture of these translations by applying the following (co-)Kleisli constructions to $\mathsf{ILL}$, which are similar to those applied to the bicategory $\mathcal{LG}$:
\begin{definition}[syntactic (co-)Kleisli constructions]
The bicategories $\mathsf{CLL}$ for classical linear logic, $\mathsf{IL}$ for intuitionistic logic and $\mathsf{CL}$ for classical logic are the (co-)Kleisli constructions
\begin{mathpar}
\mathsf{CLL} \colonequals \mathsf{ILL}_{\neg \neg}
\and
\mathsf{IL} \colonequals \mathsf{ILL}_{\oc}
\and
\mathsf{CL} \colonequals \mathsf{ILL}_{\oc \wn}.
\end{mathpar}
\end{definition} 

As expected, $\mathsf{CLL}$ forms a star-autonomous category, $\mathsf{IL}$ a cartesian closed category, and $\mathsf{CL}$ a control category, all up to the equivalence relation $\simeq$; we leave the details to the reader.
Note that the translations $\mathscr{T}_{\neg\neg}$, $\mathscr{T}_{\oc}$ and $\mathscr{T}_{\oc\wn}$ map $\mathsf{ILL}$ to $\mathsf{CLL}$, $\mathsf{IL}$ and $\mathsf{CL}$, respectively.
By the above results, the last three bicategories embody classical linear, intuitionistic and classical logics, respectively.
We show in the next section the \emph{biequivalences} $\mathsf{ILL} \simeq \mathcal{LG}$, $\mathsf{CLL} \simeq \mathcal{LG}_{\neg\neg}$, $\mathsf{IL} \simeq \mathcal{LG}_{\oc}$ and $\mathsf{CL} \simeq \mathcal{LG}_{\oc\wn}$.

\subsection{Biequivalences between formal systems and combinatorics}
\label{Bijections}
This section is to establish \emph{biequivalences} between formal systems and combinatorics for linear, intuitionistic and classical logics, respectively, where proofs may contain cuts.
Such \emph{intensional} biequivalences are, to the best of our knowledge, achieved for the first time in the literature. 

We begin with proving that the map $\llbracket \_ \rrbracket_{\mathcal{LG}}$ from sigma-terms to 1-cells in the bicategory $\mathcal{LG}$ (Definition~\ref{DefDirectReading}) is \emph{surjective} (Theorem~\ref{ThmFullCompleteness}).
To this end, we need two technical lemmata:

\begin{definition}[the domain and the codomain of a combinatorial cut]
Given a combinatorial cut $\mathscr{C}$,  its \emph{\bfseries domain} $\mathrm{dom}(\mathscr{C})$ and \emph{\bfseries codomain} $\mathrm{cod}(\mathscr{C})$ are the combinatorial arenas defined by
\begin{mathpar}
\mathrm{dom}(\mathscr{A}_{[0]} \multimap \mathscr{A}_{[1]}) \colonequals \mathscr{A}_{[0]}
\and
\mathrm{cod}(\mathscr{A}_{[0]} \multimap \mathscr{A}_{[1]}) \colonequals \mathscr{A}_{[1]}
\\
\mathrm{dom}(\mathscr{C} \mathbin{\&} \mathscr{C}') \colonequals \mathrm{dom}(\mathscr{C}) \mathbin{\&} \mathrm{dom}(\mathscr{C}')
\and
\mathrm{cod}(\mathscr{C} \mathbin{\&} \mathscr{C}') \colonequals \mathrm{cod}(\mathscr{C}) \mathbin{\&} \mathrm{cod}(\mathscr{C}')
\\
\mathrm{dom}(\oc \mathscr{C}) \colonequals \oc \mathrm{dom}(\mathscr{C})
\and
\mathrm{cod}(\oc \mathscr{C}) \colonequals \oc \mathrm{cod}(\mathscr{C})
\and
\mathrm{dom}(\top) \colonequals \mathrm{cod}(\top) \colonequals \top
\and
\mathrm{dom}(\mathscr{C} \otimes \mathscr{C}') \colonequals \mathrm{dom}(\mathscr{C}) \otimes \mathrm{dom}(\mathscr{C}')
\and
\mathrm{cod}(\mathscr{C} \otimes \mathscr{C}') \colonequals \mathrm{cod}(\mathscr{C}) \otimes \mathrm{cod}(\mathscr{C}').
\end{mathpar}
\end{definition}

\begin{lemma}[acyclicity lemma]
\label{LemLeftLinearImplication}
Let $\varphi : \llbracket F \rrbracket_{\mathcal{LG}}$ be a 1-cell in the bicategory $\mathcal{LG}$ with $F$ a sigma-sequent.
If $<_\varphi$ is the relation on the set of all occurrences of linear implication in the domain of $F$ and of cuts in the intensionality of $F$ such that $X <_\varphi X'$ if and only if there is $\boldsymbol{s} \vec{m} \boldsymbol{t} \vec{n} \in \lceil \varphi \rceil$ with $\vec{m}$ in the domain of $\llbracket X' \rrbracket_{\mathcal{LG}}$, and $\vec{n}$ in the codomain of $\llbracket X \rrbracket_{\mathcal{LG}}$ and justified by $\vec{m}$, then the relation constitutes the directed edges $X \rightarrow X'$ of a finite rooted dag. 
\end{lemma}
\begin{proof}
By the recursive alternation axiom on positions and Proposition~\ref{PropClosureOfPositionsUnderViews}, the directed edges never constitute a cycle. 
\end{proof}

\begin{definition}[the size of a sigma-sequent]
The \emph{\bfseries size} of a formula $A$ in intuitionistic linear logic (Definition~\ref{DefILL}) is the positive integer $|A|$ defined by
\begin{mathpar}
|\top| \colonequals |1| \colonequals 1
\and
|\neg A| \colonequals |\oc A| \colonequals |A| + 1
\and
|A \otimes B| \colonequals |A \mathbin{\&} B| \colonequals |A \multimap B| \colonequals |A| + |B| + 1,
\end{mathpar}
that $|T|$ of a type $T$ is $|\underline{T}|$, that $|\mathrm{Cut}(T \mid T')|$ of a cut $\mathrm{Cut}(T \mid T')$ is $|T|$, and that $| \Gamma \dashv \Pi \vdash \Phi |$ of a sigma-sequent $\Gamma \dashv \Pi \vdash \Phi$ is $|\Gamma| + |\Pi| + |\Phi|$, where $|\Gamma|$ is the sum of the sizes of the components of $\Gamma$, and similarly for $|\Pi|$ and $|\Phi|$.
\end{definition}

\begin{lemma}[prime lemma]
\label{LemSliceLemma}
Given a sigma-sequent $F$ and a 1-cell $\varphi : \llbracket F \rrbracket_{\mathcal{LG}}$ in the bicategory $\mathcal{LG}$, there is a finitely-indexed family $\{ (F_i, \varphi_i) \}_{i \in I}$ of such pairs that satisfies
\begin{enumerate}

\item $|F_i| < |F|$ for each $i \in I$;

\item $\varphi$ is definable by a sigma-term $t :: F$ if so is $\varphi_i$ by a sigma-term $t_i :: F_i$ for each $i \in I$;

\item For each $i \in I$, the intensionality of $F_i$ consists of of-course and/or with, the domain of one, negation, with and/or of-course, the codomain of top, one, tensor, with and/or of-course, and $\lceil \varphi_i \rceil$ visits an occurrence of of-course in the domain at most once. 

\end{enumerate}
\end{lemma}

\begin{proof}
Let $F = \Gamma \dashv \Pi \vdash \Phi$.
We first verify the claim on of-course in the third clause. 
Assume $\Gamma = \Delta, \oc S, \Theta$, and that $\lceil \varphi \rceil$ visits the occurrence $\oc S$ at most $n$-times. 
By taking sufficiently large $n$, we do not lose generality.
Clearly, there is a 1-cell $\varphi' : \llbracket \Delta, (\oc S)^n, \Theta \dashv \Pi \vdash \Phi \rrbracket_{\mathcal{LG}}$ in $\mathcal{LG}$ that coincides with $\varphi$ except that $\lceil \varphi' \rceil$ visits each component $\oc S$ of $(\oc S)^n$ at most once.
We iterate this procedure until we obtain a sigma-sequent $F_1$ and a 1-cell $\varphi_1 : \llbracket F_1 \rrbracket_{\mathcal{LG}}$ such that $\lceil \varphi_1 \rceil$ visits each occurrence of of-course in the domain at most once.
If there is a sigma-term $t_1 :: F_1$ such that $\llbracket t_1 \rrbracket_{\mathcal{LG}} = \varphi_1$, then by the rules \textsc{$\oc$C} and \textsc{XL}, we obtain a sigma-term $t :: F$ from $t_1$ such that $\llbracket t \rrbracket_{\mathcal{LG}} = \varphi$. 
Thus, the claim on of-course is always satisfied; thus, we assume it from now on.

In the following, we prove the remaining claims of the lemma by induction on $|F|$.
We first consider occurrences of top, tensor and linear implication in $\Gamma$, and those of cuts in $\Pi$:
\begin{itemize}

\item If $\Gamma$ contains top (respectively, tensor), then we reduce this case to the induction hypothesis (respectively, hypotheses) by the rules \textsc{$\top$L} (respectively, \textsc{$\otimes$L}) and \textsc{XL}.

\item Assume an occurrence of cut in $\Pi$ or linear implication in $\Gamma$.
By Lemma~\ref{LemLeftLinearImplication}, there is the least one among them for $<_\varphi$.
Suppose that it is $S \multimap T$ in $\Gamma$; write $\Gamma = \Delta, S \multimap T, \Theta$. 
For each $\boldsymbol{s} \in \varphi$, we write $\boldsymbol{s} @ S$ for its justified subsequence that consists of the elements $\boldsymbol{s}(i)$ such that the P-view $\lceil \boldsymbol{s}(1)\boldsymbol{s}(2) \dots \boldsymbol{s}(i) \rceil$ has a move in $\llbracket S \rrbracket_{\mathcal{LG}}$, and $\boldsymbol{s} @ T$ for the subsequence of $\boldsymbol{s}$ consisting of the remaining elements. 
Let $\varphi^S \colonequals \{ \, \boldsymbol{s} @ S \mid \boldsymbol{s} \in \varphi \, \}$ and $\varphi^T \colonequals \{ \, \boldsymbol{s} @ T \mid \boldsymbol{s} \in \varphi \, \}$. 
It suffices to show that $\varphi^S$ and $\varphi^T$ do not visit the same element of $\Gamma$ or $\Pi$ because then we can reduce this case to the induction hypotheses by the rules \textsc{$\multimap$L} and \textsc{XL}.
Because the occurrence $S \multimap T$ is the least one, it suffices to focus on elements of $\Gamma$ different from linear implication.
We can assume that $\Gamma$ does not have tensor, and that $\lceil \varphi \rceil$ visits an occurrence of of-course in $\Gamma$ at most once. 
Thus, the disjointness of $\varphi^S$ and $\varphi^T$ holds. 

\item If an element of $\Pi$ that is not of-course or with is the least one for $<_\varphi$, then similarly to the above case we reduce this case to the induction hypotheses by the rules \textsc{Cut} and \textsc{XL}.

\end{itemize}

Let us leave the cases on $\Phi$ to the reader since they are just straightforward. 
\end{proof}

\begin{definition}[prime form]
A pair $(F, \varphi)$ of a sigma-sequent $F$ and a 1-cell $\varphi : \llbracket F \rrbracket_{\mathcal{LG}}$ in the bicategory $\mathcal{LG}$ is said to be \emph{\bfseries prime} if it satisfies all conditions of the third clause of Lemma~\ref{LemSliceLemma}.
\end{definition}

We are now ready to prove the main theorem of the present section:
\begin{theorem}[an intensional surjection]
\label{ThmFullCompleteness}
Given a sigma-sequent $F$ and a 1-cell $\varphi : \llbracket F \rrbracket_{\mathcal{LG}}$ in the bicategory $\mathcal{LG}$, there is a sigma-term $t :: F$ such that $\llbracket t \rrbracket_{\mathcal{LG}} = \varphi$.
\end{theorem}
\begin{proof}
Let $F = \Gamma \dashv \Pi \vdash \Phi$. 
We proceed by induction on $|F|$.
By Lemma~\ref{LemSliceLemma}, we assume that the pair $(F, \varphi)$ is prime. 
In what follows, we proceed by case analysis on $\Phi$.
First, assume $\Phi = \boldsymbol{\epsilon}$.
If the first computational step of $\varphi$ happens in an occurrence of cut, then, by $\Phi = \boldsymbol{\epsilon}$ and the definition of 1-cells in $\mathcal{LG}$, the cut does not have an internal additive or an exponential structure, a contradiction.
If the first step arises in an occurrence of negation (respectively, with, of-course) in $\Gamma$, then we reduce this case to the induction hypothesis (or hypotheses) by the rules \textsc{$\neg$L} (respectively, \textsc{$\&$L}, \textsc{$\oc$D}) and \textsc{XL}.
The first step is never in an occurrence of one in $\Gamma$ by the joker axiom.
We have considered all cases for $\Phi = \boldsymbol{\epsilon}$.
The cases $\Phi = \top$ and $\Phi = 1$ are evident.

Next, suppose $\Phi = S \mathbin{\&} T$. 
By the definition of 1-cells in $\mathcal{LG}$, every element of $\Pi$ is with that corresponds to $S \mathbin{\&} T$ via $\varphi$. 
Thus, reduce this case to the induction hypotheses by the rule \textsc{$\&$R}.

Now, assume $\Phi = \oc S$. 
If $S = \top$, then $\oc S = \top$ so that this case coincides with the one $\Phi = \top$.
If $S$ is not top, then the linearity of $\varphi$ together with the definition of 1-cells in $\mathcal{LG}$ implies that $\Pi$ and $\Gamma$ contain only of-course, so we can reduce this case to the one $\Phi = S$ by the rule \textsc{$\oc$R}.

Finally, assume $\Phi = S \otimes T$. 
We reduce this case to the induction hypotheses by the rule \textsc{$\otimes$R} similarly to the case where we apply the rules \textsc{Cut} and \textsc{XL}.
\end{proof}

\begin{corollary}[an intensional biequivalence between logic and combinatorics]
\label{CorIntensionalBijection}
The 2-functor $\llbracket \_ \rrbracket_{\mathcal{LG}}$ of Proposition~\ref{PropInterpretationFunctor} defines a biequivalence $\mathsf{ILL} \simeq \mathcal{LG}$.
\end{corollary}
\begin{proof}
The 2-functor (or E-functor) is essentially surjective on objects by Corollary~\ref{CorCombinatorialArenasAndFormulaeInILL}, and fully faithful by Proposition~\ref{PropSigmaTermsAsFinitePresentations} and Theorem~\ref{ThmFullCompleteness}. 
\end{proof}

Under this biequivalence, the cut-eliminations on sigma-terms can be read as an operation on 1-cells in the bicategory $\mathcal{LG}$, called the \emph{(one-step) hiding operation}, for which we write $\mathcal{H}$. 

\begin{notation}
Fix a 2-functor $\llbracket \_ \rrbracket_{\mathcal{LG}}^\flat : \mathcal{LG} \rightarrow \mathsf{ILL}$, whose object-map is a right inverse of the object-map of the biequivalence $\llbracket \_ \rrbracket_{\mathcal{LG}}$, and arrow-map is the inverse of the arrow-map of $\llbracket \_ \rrbracket_{\mathcal{LG}}$. 
(It is unique up to E-natural isomorphisms thanks to Corollary~\ref{CorIntensionalBijection}.)
\end{notation}

\begin{definition}[the hiding operation]
The \emph{\bfseries (one-step) hiding operation} is the E-functor $\mathcal{H} : \mathcal{LG} \rightarrow \mathcal{LG}$ given by $\mathcal{H}(\mathscr{A}) \colonequals \mathscr{A}$ on objects $\mathscr{A}$, and by $\mathcal{H}(\Pi, \varphi) \colonequals (\mathcal{H}(\Pi), \mathcal{H}(\varphi)) \colonequals (\Pi', \varphi')$ if and only if $\llbracket \varphi \rrbracket_{\mathcal{LG}}^{\flat} :: (\llbracket \mathscr{A} \dashv \Pi \vdash \mathscr{B} \rrbracket_{\mathcal{LG}}^{\flat}) \rightarrow \llbracket \varphi' \rrbracket_{\mathcal{LG}}^{\flat} :: (\llbracket \mathscr{A} \dashv \Pi' \vdash \mathscr{B} \rrbracket_{\mathcal{LG}}^{\flat})$ on 1-cells $(\Pi, \varphi) : \mathscr{A} \rightarrow \mathscr{B}$.
\end{definition}

The hiding operation $\mathcal{H}$ is much more intricate and cumbersome to describe directly in terms of the combinatorial structures of $\mathcal{LG}$.
This is why we have instead defined it in terms of the cut-elimination $\rightarrow$ for the sigma-calculus by exploiting the biequivalence $\mathsf{ILL} \simeq \mathcal{LG}$.

Theorem~\ref{ThmSyntacticCutElim} is then translated into: 

\begin{corollary}[combinatorial cut-elimination]
\label{CorCombinatorialCutElim}
Each 1-cell $(\Pi, \varphi)$ in the bicategory $\mathcal{LG}$ has a finite sequence $(\Pi, \varphi), \mathcal{H}(\Pi, \varphi), \mathcal{H}^2(\Pi, \varphi), \dots, \mathcal{H}^n(\Pi, \varphi)$ of 1-cells in $\mathcal{LG}$ such that $\mathcal{H}^n(\Pi, \varphi) = (\boldsymbol{\epsilon}, \mathcal{H}^\omega(\varphi))$, and $t \rightarrow t'$ if and only if $\llbracket t \rrbracket_{\mathcal{LG}} \neq \llbracket t' \rrbracket_{\mathcal{LG}}$ and $\mathcal{H}(\llbracket t \rrbracket_{\mathcal{LG}}) = \llbracket t' \rrbracket_{\mathcal{LG}}$ for all sigma-terms $t$.
\end{corollary}

By construction, the hiding operation corresponds precisely to the cut-elimination; also, by Theorem~\ref{ThmFullCompleteness}, the former is defined on \emph{all} 1-cells in $\mathcal{LG}$.
Consequently, it \emph{completely} captures the cut-elimination: $t \rightarrow t'$ if and only if $\llbracket t \rrbracket_{\mathcal{LG}} \neq \llbracket t' \rrbracket_{\mathcal{LG}}$ and $\mathcal{H}(\llbracket t \rrbracket_{\mathcal{LG}}) = \llbracket t' \rrbracket_{\mathcal{LG}}$ for all sigma-terms $t$.
To the best of our knowledge, this is the first syntax-free, non-inductive recast of cut-elimination. 

Because the bicategories $\mathsf{CLL}$ and $\mathcal{LG}_{\neg\neg}$ (respectively, $\mathsf{IL}$ and $\mathcal{LG}_\oc$, $\mathsf{CL}$ and $\mathcal{LG}_{\oc\wn}$) are obtained from $\mathsf{ILL}$ and $\mathcal{LG}$, respectievly, through the same (co-)Kleisli construction, Corollary~\ref{CorCombinatorialCutElim} extends to this pair of logic and combinatorics too. 
As a result, we have reduced the proof theory of these propositional logics to the study of combinatorics, including their dynamics and intensionality.

Last but not least, another immediate consequence of Corollary~\ref{CorIntensionalBijection}, which is of theoretical interest, is the following technical result:
\begin{corollary}[linear winning implies flat innocence]
\label{CorLinearWinningImpliesFlatInnocence}
Every linearly winning strategy is flatly innocent (and therefore non-circular too).
\end{corollary}

\section{Combinatorial higher-order computation}
\label{Computation}
This last section shows that the bicategory $\mathcal{G}_\oc$ admits a model computation that can simulate the higher-order functional programming language \emph{PCF} \cite{scott1993type,plotkin1977lcf}.
This \emph{PCF-completeness} of $\mathcal{G}_\oc$ verifies that our combinatorics induces a quite general class of higher-order computation.

The main theorem of Yamada \cite{yamada2019game} states that strategies \emph{presentable by finitely presentable strategies} can simulate PCF.
Our PCF-completeness significantly improves this theorem by proving that finitely presentable ones suffice to simulate PCF. 
This advance is due to our combinatorial recast of games: The recast dispenses with the computation on infinitary tags in \cite{yamada2019game}.

The improvement is not only on the theorem itself but also on the underlying framework: The preceding approach \cite{yamada2019game} entails extremely intricate constructions on finitely presentable strategies that present strategies interpreting PCF; in contrast, our method is not bothered by such low-level constructions at all. 
This simplification is crucial for our framework to be \emph{usable}. 

\begin{remark}
Another, more conceptual advance is that, by dispensing with the finitely presentable strategies that \emph{extrinsically} and \emph{non-canonically} encode strategies interpreting PCF in \cite{yamada2019game}, and more generally by freeing from any syntactic encodings typical in existing models of computation, all structures for computation in $\mathcal{G}_\oc$ are \emph{intrinsic} and thus \emph{canonical} to $\mathcal{G}_\oc$.
Note that $\mathcal{G}_\oc$ consists of semantic, in particular syntax-free and non-inductive, objects. 
Hence, if the standard formalism in mathematics, or the category of sets, is replaced with $\mathcal{G}_\oc$, then the present result resolves the extrinsic, non-canonical nature of computability raised in \cite[\S 1.2]{abramsky2014intensionality}.
\end{remark}

To establish the PCF-completeness of 1-cells in $\mathcal{G}_\oc$, let us begin with showing that strategies modelling \emph{atomic} terms in PCF are all finitely presentable:

\begin{notation}
In the following examples, we use the superscripts $(\_)'$, $(\_)''$, etc., as informal tags on combinatorial arenas and their vertices. 
\end{notation}

\begin{example}
\label{ExSuccessor}
Recall the finitely presentable strategy $\mathrm{succ} : (\mathscr{N} \dashv \vdash \mathscr{N}')$ given in Example~\ref{ExFinitePresentableStrategies}.
This computation also forms a finitely presentable strategy on the combinatorial sequent $\oc \mathscr{N} \dashv \vdash \mathscr{N}'$.
Thus, the pair $(\boldsymbol{\epsilon}, \mathrm{succ})$ is a normalised 1-cell $\mathscr{N} \rightarrow \mathscr{N}$ in the bicategory $\mathcal{G}_\oc$.
It computes the successor $\mathbb{N} \rightarrow \mathbb{N}$ in the sense that $\mathrm{succ}(\underline{n}^\dagger) \simeq \underline{n+1}$ holds for all $n \in \mathbb{N}$.
\end{example}

\begin{example}
\label{ExPredecessor}
It is not hard to see that there are no normalised 1-cells $\mathscr{N} \rightarrow \mathscr{N}$ in $\mathcal{G}_\oc$ that play as the predecessor $\mathbb{N} \rightarrow \mathbb{N}$.
This difference indicates that predecessor is slightly more complex than successor.
This subtlety is owed to the present refinement of the previous work \cite{yamada2019game} in which the complexity-theoretic difference between successor and predecessor dissapears.  

To construct a 1-cell $\mathscr{N} \rightarrow \mathscr{N}$ for the predecessor, we define strategies $\mathrm{pred}_+ : (\oc \mathscr{N} \dashv \vdash \mathscr{N}_{\neg\neg}')$ and $\mathrm{pred}_- : (\oc \mathscr{N}_{\neg\neg} \dashv \vdash \mathscr{N}')$, where $\mathscr{N}_{\neg\neg} \colonequals \oc (\bot_{[q]} \multimap_{[p]} \bot_{[y]}) \otimes \bot_{[n]} \multimap_{[o]} \neg_{[i]} \neg_{[j]} \bot_{[\hat{q}]}$, by
\begin{scriptsize}
\begin{align*}
\mathrm{pred}_+ &\colonequals \mathrm{Pref}(\{ \, {\dashv \vdash} o'i' . \vec{q}_0 . \vec{\mathrm{no}} . \vec{\mathrm{no}}' \, \} \\
&\cup \{ \, {\dashv \vdash}o'i' . \vec{q}_0 . \vec{\mathrm{yes}}_0 . j' . (\vec{q}'_0 . \vec{q}_1 . \vec{\mathrm{yes}}_1 . \vec{\mathrm{yes}}'_0) (\vec{q}'_1 . \vec{q}_2 . \vec{\mathrm{yes}}_2 . \vec{\mathrm{yes}}'_1) \dots (\vec{q}'_n . \vec{q}_{n+1} . \vec{\mathrm{yes}}_{n+1} . \vec{\mathrm{yes}}'_n) . \vec{q}'_{n+1} . \vec{q}_{n+2} . \vec{\mathrm{no}} . \vec{\mathrm{no}}' \mid n \in \mathbb{N} \, \})^{\mathrm{Even}} \\
\mathrm{pred}_- &\colonequals \mathrm{Pref}(\{ \, {\dashv \vdash}\vec{q}'_0 . i . j . \vec{q}_0 . (\vec{\mathrm{yes}}_0 . \vec{\mathrm{yes}}'_0 . \vec{q}'_1 . \vec{q}_1) (\vec{\mathrm{yes}}_1 . \vec{\mathrm{yes}}'_1 . \vec{q}'_2 . \vec{q}_2) \dots (\vec{\mathrm{yes}}_n . \vec{\mathrm{yes}}'_n . \vec{q}'_{n+1} . \vec{q}_{n+1}) . \vec{\mathrm{no}} . \vec{\mathrm{no}}' \mid n \in \mathbb{N} \, \})^{\mathrm{Even}}.
\end{align*}
\end{scriptsize}We then define $\mathrm{pred} \colonequals \mathrm{pred}_+^\dagger \between \mathrm{pred}_-$ on $\oc \mathscr{N} \dashv \oc \mathscr{N}_{\neg\neg} \multimap \oc \mathscr{N}_{\neg\neg} \vdash \mathscr{N}$.
It is easy to see that $\mathrm{pred}_+$ and $\mathrm{pred}_-$ are finitely presentable, and thus so is $\mathrm{pred}$. 
Hence, the pair $(\mathscr{N}_{\neg\neg}, \mathrm{pred})$ is a 1-cell in $\mathcal{G}_\oc$.
It plays as the predecessor: $\mathrm{pred}(\underline{n+1}^\dagger) \simeq \underline{n}$ for all $n \in \mathbb{N}$ and $\mathrm{pred}(\underline{0}^\dagger) \simeq \underline{0}$ hold.
\end{example}

\begin{example}
We define the combinatorial arena $\mathscr{B}$, called the \emph{\bfseries boolean combinatorial arena}, by $\mathscr{B} \colonequals \bot_{[\mathrm{tt}]} \otimes \bot_{[\mathrm{ff}]} \multimap_{[o]} \bot_{[q]}$, and the strategies $\underline{\mathrm{tt}}, \underline{\mathrm{ff}} : \mathscr{B}$ by $\underline{\mathrm{tt}} \colonequals \{ \boldsymbol{\epsilon}, oq.\mathrm{tt} \}$ and $\underline{\mathrm{ff}} \colonequals \{ \boldsymbol{\epsilon}, oq. \mathrm{ff} \}$.

The \emph{\bfseries (binary) conditional} on a given combinatorial arena $\mathscr{A}$ refers to the strategy $\mathrm{case}_{\mathscr{A}} : \oc \mathscr{A}, \oc \mathscr{A}', \oc \mathscr{B}'' \dashv \vdash \mathscr{A}'''$ defined by
\begin{small}
\begin{equation*}
\mathrm{case}(\mathscr{A})_\oc \colonequals \mathrm{Pref}(\{ \, \vec{a}'''_1 . o''q'' . \mathrm{tt}'' . \vec{a}_1 . \boldsymbol{s} \mid \vec{a}'''_1 . \vec{a}_1 . \boldsymbol{s} \in \mathrm{der}_{\mathscr{A}} \, \} \cup \{ \, \vec{a}'''_1 . o''q'' . \mathrm{ff}'' . \vec{a}'_1 . \boldsymbol{t} \mid \vec{a}'''_1 . \vec{a}'_1 . \boldsymbol{t} \in \mathrm{der}_{\mathscr{A}} \, \})^{\mathrm{Even}}.
\end{equation*}
\end{small}It is easy to see that this strategy is finitely presentable, and it plays as the binary conditional: $\mathrm{case}_{\mathscr{A}}(\varphi_1^\dagger \otimes \varphi_2^\dagger \otimes \underline{\mathrm{tt}}^\dagger) \simeq \varphi_1$ and $\mathrm{case}_{\mathscr{A}}(\varphi_1^\dagger \otimes \varphi_2^\dagger \otimes \underline{\mathrm{ff}}^\dagger) \simeq \varphi_2$ for all strategies $\varphi_1, \varphi_2 : (\Gamma \dashv \Pi \vdash \mathscr{A})$.

In addition, the \emph{\bfseries ifzero} refers to the strategy $\mathrm{ifzero} : \oc \mathscr{N} \dashv \vdash \mathscr{B}$ defined by
\begin{small}
\begin{equation*}
\mathrm{ifzero} \colonequals \mathrm{Pref}(\{ \, {\dashv \vdash}oq . \vec{q}_0 . \vec{\mathrm{no}} . \mathrm{tt} \, \} \cup \{ \, {\dashv \vdash}oq . \vec{q}_0 . \vec{\mathrm{yes}}_0 . \mathrm{ff} \, \})^{\mathrm{Even}}.
\end{equation*}
\end{small}It is clear that this strategy is finitely presentable, and its computation checks whether a given winning strategy $\underline{n} : \mathscr{N}$ is $\underline{0}$ by $\mathrm{ifzero}(\underline{n}) \simeq \underline{\mathrm{tt}}$ if $n = 0$, and $\mathrm{ifzero}(\underline{n}) \simeq \underline{\mathrm{ff}}$ otherwise. 
\end{example}

\begin{example}
\label{ExFixedPoints}
The \emph{\bfseries fixed-point combinator} on a combinatorial arena $\mathscr{A}$ is the strategy $\mathrm{fix}_{\mathscr{A}} : \oc (\mathscr{A} \Rightarrow \mathscr{A}') \dashv \vdash \mathscr{A}''$ that computes, loosely speaking, by playing as the dereliction between $\mathscr{A}'$ and $\mathscr{A}''$ as well as between $\mathscr{A}$ and $\mathscr{A}'$; see \cite[\S 2.3.3]{hyland1997game} or \cite[Example~75]{yamada2019game} for the details. 

This strategy calculates the \emph{fixed-point} of a given strategy $\varphi : \mathscr{A} \Rightarrow \mathscr{A}$ in the sense that $\varphi(\mathrm{fix}_{\mathscr{A}}(\varphi^\dagger)^\dagger) \simeq \varphi$ holds.
Because the fixed-point combinator plays essentially as the two derelictions, it is finitely presentable. 
The significance of this finite presentability is that it completely dispenses with the intricate computation on tags for exponentials in \cite[Example~75]{yamada2019game}.
\end{example}

One way to enumerate strategies definable by PCF, up to the hiding equivalence, is to start from $\mathrm{der}_{\mathscr{A}}$, $\underline{0}_{\mathscr{A}}$, $\mathrm{succ}$, $\mathrm{pred}$, $\mathrm{case}_{\mathscr{A}}$, $\mathrm{ifzero}$ and $\mathrm{fix}_{\mathscr{A}}$, called \emph{PCF-atomic} strategies, where $\mathscr{A}$ is generated from $\mathscr{N}$, $\mathscr{B}$ and/or $\top$ by product and/or implication, and apply parallel composition, transpose, pairing and promotion; see the proof of \cite[Theorem~81]{yamada2019game} (n.b., a variant of parallel composition is called \emph{concatenation} and written $\ddagger$ in \cite{yamada2019game}).
Therefore, it suffices to show that PCF-atomic ones are finitely presentable, and the constructions preserve finite presentability. 

We have accomplished the first task by the above examples; the second task has been completed by Propositions~\ref{PropWellDefinedConstructionsOnStrategies} and \ref{PropPreservationOfFlatInnocenceAndFinitePresentability}.
Thus, the main theorem of the present section follows: 

\begin{theorem}[a combinatorial model of higher-order computation]
\label{ThmComputability}
Each strategy $\varphi : (\Gamma \dashv \vdash \mathscr{A})$ definable by PCF has a 1-cell $(\Pi, \kappa)$ in the bicategory $\mathcal{G}_\oc$ that satisfies $\mathcal{H}^\omega(\kappa) = \varphi$.
\end{theorem}

Most part of the long article \cite{yamada2019game} is dedicated to its main theorem \cite[Theorem~81]{yamada2019game}, which states that every strategy definable by PCF is presentable by a finitely presentable strategy. 
Our theorem improves this result by showing that in our combinatorial setting finitely presentable strategies suffice to simulate PCF. 
Moreover, by dispensing with the strategies presenting PCF-definable ones, our proof is much simpler and more straightforward than the previous one.

Although there is no doubt about the computability of finitely presentable strategies, one may wonder whether it is \emph{maximal}.
In particular, because a standard idea \cite{hyland2000full} in the literature of game semantics is to define an innocent strategy to be computable or \emph{\bfseries recursive} if its computation is calculable by a Turing machine, does one gain a stronger notion of computability by replacing finite presentability of innocent strategies with recursiveness?
The answer is negative:

\begin{corollary}[completeness of finite presentability]
Every recursive innocent strategy $\varphi$ on a combinatorial sequent $\oc \mathscr{A} \dashv \vdash \mathscr{B}$ has a 1-cell $(\Xi, \kappa)$ in the bicategory $\mathcal{G}_\oc$ such that $\kappa \simeq \varphi$.
\end{corollary}
\begin{proof}[Proof (sketch)]
Fix an effective encoding $\# : \lceil \mathcal{P}_{\oc \mathscr{A} \dashv \vdash \mathscr{B}} \rceil \rightarrowtail \mathbb{N}$ with its range decidable, and its inverse effective. 
Because PCF-completeness implies Turing completeness, Theorem~\ref{ThmComputability} finds a finitely presentable strategy $\kappa_\varphi$ on some combinatorial sequent $\oc \mathscr{N} \dashv \Pi \vdash \mathscr{N}$ that computes the P-view map $f_\varphi$ of $\varphi$ in the sense that $\kappa_\varphi (\underline{\# \lceil \boldsymbol{s}\vec{o} \rceil}) = \underline{\# \lceil \boldsymbol{s}\vec{o}\vec{p} \rceil}$ if and only if $f_\varphi(\lceil \boldsymbol{s}\vec{o} \rceil) =  \lceil \boldsymbol{s}\vec{o}\vec{p} \rceil$ for all odd-length P-views $\lceil \boldsymbol{s}\vec{o} \rceil$ and P-moves $\vec{p}$ in $\oc \mathscr{A} \dashv \vdash \mathscr{B}$.
Let $\tilde{\kappa}_\varphi \colonequals \lambda(\kappa_\varphi) : (\dashv \Pi \vdash \oc \mathscr{N} \multimap \mathscr{N})$.

For the same reason, there is a finitely presentable strategy $\mu$ on some combinatorial sequent $\oc \mathscr{A}, \oc (\mathscr{N} \Rightarrow \mathscr{N}) \dashv \oc \mathscr{N} \multimap \oc \mathscr{N}', \Sigma \vdash \mathscr{B}$ that computes as follows.
At each odd-length position such that the last O-move $\vec{o}$ is in $\oc \mathscr{A}$ or $\mathscr{B}$, $\mu$ first records $\vec{o}$ and its justifier $\vec{j}$ on a new copy of $\mathscr{N}'$ in $\oc \mathscr{N}'$; in this way, $\kappa_\varphi$ sees each P-view in $\oc \mathscr{A} \dashv \vdash \mathscr{B}$ even though it is flatly innocent. 
Next, $\mu$ reads off the current P-view $\lceil \boldsymbol{s}\vec{o} \rceil$ in $\oc \mathscr{A} \dashv \vdash \mathscr{B}$ from what has been played in $\oc \mathscr{N} \multimap \oc \mathscr{N}'$, and computes in $\oc (\mathscr{N} \Rightarrow \mathscr{N})$ by playing as $\underline{\#\lceil \boldsymbol{s}\vec{o} \rceil}$ in the domain of $\mathscr{N} \Rightarrow \mathscr{N}$.
Finally, $\mu$ reads off the output given in the codomain of $\mathscr{N} \Rightarrow \mathscr{N}$ and makes the corresponding P-move in $\oc \mathscr{A}$ or $\mathscr{B}$, if any; it does not make the next P-move if there is no corresponding P-move.
Then, it suffices to define $\Xi \colonequals \Pi, \oc (\mathscr{N} \Rightarrow \mathscr{N}), \oc \mathscr{N} \multimap \oc \mathscr{N}', \Sigma$ and $\kappa \colonequals \tilde{\kappa}_\varphi^\dagger \between \mu$.
\end{proof}

Note that an argument similar to that of the proof implies:
\begin{corollary}[completeness of flat innocence]
If $\varphi$ is an innocent strategy on a combinatorial sequent $\Gamma \dashv \Pi \vdash \Phi$, then there is a flatly innocent strategy $\ell$ on some combinatorial sequent $\Gamma \dashv \Pi, \Sigma \vdash \Phi$ that coincides with $\varphi$ up to the hiding equivalence, i.e., $\varphi \simeq \ell$.
\end{corollary}

By Church-Turing thesis, the corollaries imply that finite presentability defines the maximal or \emph{canonical} notion of computability of innocent strategies in our setting.
A main advantage of finite presentability over recursiveness is its fine analysis of computational complexity, which is explicitly exhibited in the intensionality of a combinatorial sequent (though again we leave it as future work to develop the present framework as a mathematical foundation of higher-order computational complexity).
Also, note that the former is intrinsic to the structure of strategies and free from the extrinsic manipulation of symbols on tapes.

\section*{Acknowledgments}
The author acknowledges the financial support from Funai Overseas Scholarship, and also he is grateful to Samson Abramsky for fruitful discussions.

\bibliographystyle{amsalpha}
\bibliography{LinearLogic,CategoricalLogic,GamesAndStrategies,HoTT,RecursionTheory,SigmaCalculi}

\end{document}